\documentclass[12pt]{amsart}
\usepackage{fullpage}
\usepackage{amssymb,amsmath,amsthm,amscd,mathrsfs,graphicx,color}
\usepackage[cmtip,all,matrix,arrow,tips,curve]{xy}
\usepackage{tikz}
\usepackage{url}
\usepackage{pdfpages}
\usepackage{longtable}

\usepackage{array}
\newcolumntype{H}{>{\setbox0=\hbox\bgroup}c<{\egroup}@{}}

\numberwithin{equation}{section}
\newtheorem{teo}{Theorem}[section]
\newtheorem*{teo*}{Main Theorem}
\newtheorem{pro}[teo]{Proposition}
\newtheorem{lem}[teo]{Lemma}
\newtheorem{cor}[teo]{Corollary}

\theoremstyle{definition}
\newtheorem{dfn}[teo]{Definition}

\newtheorem{ex}[teo]{Example}

\theoremstyle{remark}
\newtheorem{rem}[teo]{Remark}

\newcommand{\calA}{\mathcal{A}}

\newcommand{\calD}{\mathcal{D}}
\newcommand{\calE}{\mathcal{E}}
\newcommand{\calF}{\mathcal{F}}
\newcommand{\calG}{\mathcal{G}}

\newcommand{\calL}{\mathcal{L}}
\newcommand{\calM}{\mathcal{M}}
\newcommand{\calN}{\mathcal{N}}
\newcommand{\calO}{\mathcal{O}}

\newcommand{\SL}{\operatorname{SL}}
\newcommand{\GL}{\operatorname{GL}}

\newcommand{\Aut}{\operatorname{Aut}}
\renewcommand{\Im}{\operatorname{Im}}

\newcommand{\NS}{\operatorname{NS}}
\newcommand{\MW}{\operatorname{MW}}
\newcommand{\MWt}{\operatorname{MW}_{\tors}}
\renewcommand{\O}{\operatorname{O}}
\newcommand{\SO}{\operatorname{SO}}
\newcommand{\tO}{\operatorname{\widetilde{O}}}
\newcommand{\bO}{\operatorname{\overline{O}}}
\newcommand{\rank}{\operatorname{rank}}
\newcommand{\sign}{\operatorname{sign}}
\newcommand{\coker}{\operatorname{coker}}

\newcommand{\ZZ}{\mathbb{Z}}
\newcommand{\QQ}{\mathbb{Q}}
\newcommand{\RR}{\mathbb{R}}
\newcommand{\CC}{\mathbb{C}}

\newcommand{\PP}{\mathbb{P}}

\newcommand{\calFF}{\calM}
\newcommand{\slz}{\SL(2,\ZZ)}
\newcommand{\pslz}{\PSL(2,\ZZ)}
\newcommand{\HH}{\mathbb{H}}
\newcommand{\PSL}{\operatorname{PSL}}

\newcommand{\id}{\operatorname{id}}

\newcommand{\jGamma}{j_{\barGamma}}
\newcommand{\barGamma}{\bar\Gamma}

\newcommand{\tr}{\operatorname{tr}}
\newcommand{\tors}{\operatorname{tors}}

\makeatletter
\let\@wraptoccontribs\wraptoccontribs
\makeatother

\date{}

\begin{document}

\title[Elliptic $K3$ surfaces]{Moduli of elliptic $K3$ surfaces: monodromy and Shimada root lattice strata }

\author[K. Hulek]{Klaus Hulek}
\address{Institut f\"ur Algebraische Geometrie, Leibniz Universit\"at Hannover,  30060 Hannover, Germany}
\email{hulek@math.uni-hannover.de}

\author[M: L\"onne]{Michael L\"onne}
\address{Mathematisches Institut der Universit\"at Bayreuth, Universit\"atsstr.\ 30, 95447 Bayreuth, Germany}
\email{michael.loenne@uni-bayreuth.de}

\contrib[with an appendix by]{Markus Kirschmer}
\address{Universit\"at Paderborn, Fakult\"at EIM, Institut f\"ur Mathematik, Warburger Str. 100, 33098 Paderborn,Germany}
\email{markus.kirschmer@math.upb.de}

\begin{abstract}
In this paper we investigate two stratifications of the moduli space of elliptically fibred $K3$ surfaces. The first comes from Shimada's classification of connected components of 
the moduli of elliptically fibred $K3$ surfaces and is closely related to the root lattices of the fibration.
The second is the monodromy stratification defined 
by Bogomolov, Petrov and Tschinkel. 
The main result of the paper is a classification of all positive-dimensional ambi-typical strata, that is, strata which are  both Shimada root strata and monodromy strata.
 We also discuss the relationship with moduli spaces of lattice-polarised $K3$ surfaces. The appendix by M. Kirschmer contains
computational results about the $1$-dimensional ambi-typical strata.  
\end{abstract}

\date{}
\maketitle

%%%%%%%%%%%%%%%%%%%%%%%%%%%%%%%%%%%%%%%%%
%%%%%%%%%%%%%%%%%%%%%%%%%%%%%%%%%%%%%%%%%

\section{Introduction}
\label{sec:intro}

Elliptically fibred $K3$ surfaces have been studied over a long period  from many different angles. We refer the reader in particular to the book \cite{ScSh} where $K3$ surfaces are especially treated in Sections 11 and 12. Due to 
the work of Miranda \cite{Mi90} it is well known that the moduli space $\calF$ of elliptically fibred $K3$ surfaces with a section, also known as Jacobian fibrations, can be described as a GIT quotient of an open subset $V$ in the weighted projective space $\PP_{8,12}(9,13)$ by the 
group $\SL(2,\CC)$. 
Alternatively, the moduli space $\calF$  can also be constructed as the moduli space of lattice polarised $K3$ surfaces, more precisely $U$-quasi-polarised $K3$ surfaces (where $U$ is the hyperbolic plane spanned by the classes of the section and a general fibre).

The moduli space $\calF$ itself is a rational variety by a result of Lejarraga \cite{Le}. 
Geometric properties of elliptically fibred $K3$ surfaces provide this space with many interesting geometric features.     

The starting point of our work are two stratifications of the moduli space $\calF$. The first of these was introduced by Bogomolov, Petrov and Tschinkel in \cite{BPT}. The strata of this decomposition are defined by the property that they are maximal locally closed irreducible subvarieties with constant monodromy group $\Gamma$ where $\Gamma$ is a fixed subgroup of  $\SL(2,\ZZ)$ modulo conjugation. Bogomolov, Petrov and Tschinkel prove the remarkable fact that all of these {\em monodromy strata} are themselves rational varieties. 

The other stratification is due to 
the work of Shimada \cite{shimada},  \cite{shimada2}, \cite{shimada3} in which he classifies all connected components of the moduli of elliptic $K3$ surfaces with fixed combinatorial type. This means the following: Given an elliptically 
fibred $K3$ sufrace $f: S \to \PP^1$ the components of the singular fibres not meeting the $0$-section define a root lattice $R$ which is the direct sum of some $ADE$-lattices. This need not be be saturated in the N\'eron-Severi group 
$\NS(S)$. Its saturation $L$ corresponds, by lattice theory, to an isotropic subgroup $G$ of the discriminant group of $R$. Shimada's work provides a complete classification of all connected families of elliptically fibred $K3$ surfaces with 
given $(R,G)$. This leads to a total of 3932 families and in this way one obtains a second stratification of the moduli space $\calF$. We refer to the
strata in this stratification as {\em Shimada root strata} or, shorter, simply as {\em Shimada strata}. A priori the stratifications given by monodromy strata and Shimada strata are not related and none is a refinement of the other. However, both stratifications are refined by a third stratification, namely the one given by 
a fixed configuration of the singular fibres. We call these the {\em configuration strata}. Their properties were investigated by Kloosterman in \cite{Kl}.

The starting point of our work is the observation that there are some strata which appear in both the monodromy stratification and  Shimada's stratification. We call these {\em ambi-typical } strata and the main purpose of this
paper is to understand these special strata. Our main result  (Theorem \ref{teo:mainresult}) is 
\begin{teo*}
There are exactly 50 positive-dimensional  ambi-typcical strata. These are listed in Table \ref{table:list}.
\end{teo*}
We then prove in Theorem \ref{teo:liststrata2} that the ambi-typical Shimada strata are, with the exception of two strata, already completely determined by the root lattice $R$ itself.  
In Theorem \ref{teo:liststrata} and Table \ref{tab:jmaps} we further characterise the ambi-typical strata in terms of local monondromy around the singular fibres and the branching behaviour of the $j$-invariant.

Clearly, there is a connection with moduli spaces of lattice polarised $K3$ surfaces. Indeed, every ambi-typical stratum gives rise to a priori several moduli spaces of lattice-polarised $K3$ surfaces.
It turns out that the total number of possible components of moduli spaces associated to a given ambi-typical stratum is either 1,2, or 4, as follows from  Corollary \ref{cor:countingcomponets} and Proposition \ref{prop:embeddings}.  
For each such component $\calN$ associated to an ambi-typical stratum $\calM$ 
there is a natural finite dominant map $\calN \to \calM$ and we compute its degree. 
This is the content of Proposition \ref{pro:constructionmap}, Theorem \ref{pro:descriptionmap} and Corollary \ref{cor:countingdegree}.     

We will now discuss the contents of the paper in some more detail: In Section \ref{sec:ellfibered}  we recall basic facts about elliptically fibred $K3$ surfaces and Miranda's construction of the moduli space $\calF$. 
The monodromy and Shimada stratifications are introduced in Section \ref{sec:stratifications} where we also recall the principal results of Shimada's theory.  In Section \ref{sec:mainresult} we formulate the main results of our paper.

Sections \ref{sec:config} to \ref{sec:high} are dedicated to the proof of the Main Theorem. We start in Section \ref{sec:config} by recalling the configuration strata studied by Klosterman \cite{Kl}. It is sufficient to 
look for ambi-typical configuration strata and this will lead us to the severe restriction that the generic element of an ambi-typical stratum cannot have singular fibres of type $II, III, IV, II^*$ or $III^*$ (Proposition \ref{badfibres}).
In Section \ref{sec:jfactor} we recall a natural factorisation of the $j$-invariant $j(\calE)= \jGamma \circ j_{\calE}$ via the modular curve $X(\barGamma)$ defined by the modular monodromy group $\barGamma \subset \PSL(2,\ZZ)$. 
The restrictions on fibre configurations
then translates into branching properties of $\jGamma$ and $j_\calE$ over the points $0,1,\infty$. In Section \ref{sec:Euler} we use 
an Euler number consideration to find  a list of possible candidates for the modular monodromy group  $\barGamma$. In Proposition \ref{pro:indexmodular} we prove that the index of the modular group 
$\barGamma$ in $\PSL(2,\ZZ)$ is bounded by 18. For the rest of the proof we shall distinguish between low index $\leq 6$ and high index $> 6$. The next step in the proof is that we provide Weierstra{\ss} data 
for all possible ambi-typical  strata with modular monodromy group $\barGamma$ of low index in Proposition \ref{weier}. We also obtain some results about the Mordell-Weil groups and the monodromy of the 6 families 
listed in this proposition. In Sections \ref{sec:low} and \ref{sec:high} we finally complete the proof of the classification in the low and high index case respectively.

Section \ref{sec:mirrorpol} and \ref{sec:modulimaps} are devoted to relating the ambi-typical strata which we have found to moduli spaces of lattice polarised $K3$ surfaces. 

The appendix  by M. Kirschmer contains explicit calculations concerning the $1$-dimensional ambi-typical strata, in particular the genus of the moduli spaces of lattice polarised 
$K3$ surfaces covering these ambi-typical strata and the degree of the covering map.  We find it remarkable that although the ambi-typical $1$-dimensional strata all have genus $0$, in accordance with the rationality result of
Bogomolov, Petrov and Tschinkel, the genus of the moduli space of lattice-polarised $K3$ surfaces can be as high as $13$.

\subsection*{Acknowledgements}
We thank Ichiro Shimada for extensive discussions and for sharing his insight and calculations with us.
The first author thanks Simon Brandhorst for discussions on lattice genera and for supplying the proof of Proposition \ref{pro:independenceindex}. He also thanks Eduard Looijenga for an exchange of e-mails. 
We also would like to thank the referee for the careful reading of the text and many helpful remarks.

The first author is further grateful to DFG for partial support under grant Hu 337/7-1. The second author acknowledges the support of the ERC 2013 Advanced Research Grant 340258-TADMICAMT.  

%%%%%%%%%%%%%%%%%%%%%%%%%%%%%%%%%%%%%%%%%
%%%%%%%%%%%%%%%%%%%%%%%%%%%%%%%%%%%%%%%%%

\section{Elliptically fibred $K3$ surfaces}\label{sec:ellfibered}

A $K3$ surface $S$ is {\em elliptically fibred} if there exists a surjective morphism $f: S \to \PP^1$ whose generic geometric fibre is a genus $1$ curve. We say that 
$f: S \to \PP^1$ is a {\em Jacobian fibration} if it also has a section $s$. We denote the {\em N\'eron-Severi group} by $\NS(S)$ and its rank, also called the {\em Picard rank} by $\rho(S)$.
A general fibre of $f$ and the $0$-section $s$ define a hyperbolic sublattice $U\subset \NS(S)$. Conversely, the existence of a hyperbolic plane $U$ in $\NS(S)$ guarantees the existence of a Jacobian 
fibration on $S$. But it should be noted that the class of the fibre and the section need not be contained in $U$. This is only the case if $U$ contains a nef, or equivalently base point free, isotropic class and an irreducible $-2$ curve.  

We will typically denote $K3$ surfaces with a Jacobian fibration by $f: \calE \to \PP^1$ or simply by $\calE$, where we think of $\calE$ as an elliptic curve over the function field $\CC(\PP^1)\cong\CC(t)$. The sections of $\calE$ form a 
finitely generated abelian group, the {\em Mordell-Weil group} of $\calE$ which we denote by $\MW(\calE)$.

The components of all singular fibres of $\calE$ which do not meet the section generate a sublattlice $R(\calE)$ of $\NS(\calE)$ which is a root lattice, more precisely an orthogonal sum of  lattices of $ ADE $ type.  
The  {\em trivial part} of the  N\'eron-Severi group $\NS(\calE)$ is 
$$
\NS_{\tr}(\calE) = U + R(\calE)
$$
where $U$ is the hyperbolic plane defined by the Jacobian fibration. 
To simplify notation we write here and in the sequel simply $+$ for a direct orthogonal sum.
We will denote the rank of the trivial part by $\rho_{\tr}(\calE)$. In general $R(\calE)$ is not saturated in $\NS(\calE)$ and we will denote its saturation by $L(\calE)$. It is well known, see \cite[Corollary 6.20]{ScSh}, that
$$
L(\calE)/R(\calE) \cong \MWt(\calE).
$$ 
We recall that a Jacobian fibration $\calE$ is called {\em extremal } if $\rho_{\tr}(\calE)=20$.
Such surfaces correspond to isolated points in $\calF$ and will not be considered here,
see however \cite{bemo} for the classification of those with a semi-stable fibration.

Jacobian fibrations can be classified via their Weierstra{\ss} models and
this is the approach taken by Miranda.
For the basic theory of  Weierstra{\ss} equations we refer the reader to Miranda's paper \cite{Mir}.
Any Jacobian fibration $f: \calE \to \PP^1$, where $\calE$ is a $K3$ surface, is birational to a minimal Weierstra{\ss} model 
\begin{equation}\label{equ:weierstass} 
y^2z=4x^3 - g_2xz^2 - g_3z^3 \, {\mbox { where}} \,  g_2\in H^0(\PP^1, \calO_{\PP^1}(8)), g_3\in H^0(\PP^1, \calO_{\PP^1}(12))
\end{equation}
and where the following (open) conditions hold:
\begin{itemize}
\item[(1)] $\Delta=g_2^3 - 27 g_3^2 \not\equiv 0$.
\item[(2)] For every point $q\in \PP^1$ the inequality $\min \{3\nu_q(g_2), 2\nu_q(g_3) \} < 12$ holds, where $\nu_q(g)$ is the vanishing order of a polynomial $g$ at $q$. 
\end{itemize}
The latter condition ensures that the  Weierstra{\ss} equation is {\em minimal} in the sense that we cannot write 
$g_2=h^4\overline{g}_2$ and  $g_3=h^6{\overline{g}}_3$ with  
${\overline{g}}_2 \in  H^0(\PP^1, \calO_{\PP^1}(4))$ and  ${\overline{g}}_3 \in H^0(\PP^1, \calO_{\PP^1}(6))$.

 Conversely,  an equation given as above defines a surface $\calE'$  in the projective  bundle $\PP(\calO_{\PP^1}(4) \oplus   \calO_{\PP^1}(6) \oplus \calO_{\PP^1})$ with at most 
rational double points. 
A minimal resolution $\calE$ of $\calE'$ is then a $K3$ surface and the projection onto the base of the projective bundle gives a Jacobian fibration $f: \calE \to \PP^1$.    
Weierstra{\ss} equations with coefficients $g_2,g_3$ and $g'_2,g'_3$ respectively, define isomorphic Jacobian fibrations if and only if there exists a coordinate transformation 
$\GL(2,\CC)$, which maps the pair $(g_2,g_3)$ to $(g'_2,g'_3)$. This allows to describe the moduli of Jacobian fibrations in terms of a GIT quotient. 

In order to do this we consider the weighted projective space $\PP_{8,12}(9,13)$ associated to $H^0(\PP^1, \calO_{\PP^1}(8)) \oplus H^0(\PP^1, \calO_{\PP^1}(12))$. The open
conditions described above define an open subset $V \subset \PP_{8,12}(9,13)$ and the group $\SL(2,\CC)$ acts on $\PP_{8,12}(9,13)$ as well as on the open subset $V$.
By \cite[Proposition 5.1]{Mir} all points of $V$ are stable with respect to the action of $\SL(2,\CC)$. Hence we can form the quotient 
\begin{equation*}
\calF := V/\!/\SL(2,\CC)
\end{equation*}
which is a quasi-projective variety. Its geometric points correspond bijectively to isomorphism classes of Jacobian fibrations $f: \calE \to \PP^1$ where $\calE$ is a $K3$ surface.

In Section \ref{sec:mirrorpol} we will discuss the theory of moduli spaces of lattice (quasi-)polarised $K3$ surfaces. Taking the hyperbolic plane $U$ (which one should think of as being spanned by the class of the section
and of a general  fibre) one obtains an alternative construction of the moduli space $\calF$ as a quotient of a homogeneous domain of type IV by an arithmetic group. As we shall see, both points of view are relevant for our purposes. 
An interesting result of Odaka and Oshima \cite[Theorem 7.9]{OO} even says that the GIT compactification and the Baily-Borel compactification of this space coincide, i.e. $\calF^{\operatorname{GIT}} \cong \calF^{\operatorname{BB}}$.

%%%%%%%%%%%%%%%%%%%%%%%%%%%%%%%%%%%%%%%%%
%%%%%%%%%%%%%%%%%%%%%%%%%%%%%%%%%%%%%%%%%

\section{Monodromy and Shimada strata}\label{sec:stratifications}

In this section we will introduce the main protagonists of this paper, namely {\em monodromy strata} and {\em Shimada strata.}

%%%
\subsection*{Monodromy strata}
\label{sec:BPT}
Monodromy strata for Jacobian fibrations were first introduced by Bogomolov, Petrov and Tschinkel in \cite{BPT}. Here we recall some 
 of their results, restricting ourselves to the case in hand, namely Jacobian fibrations of $K3$ surfaces. 

The first step is to consider the open subset $V'\subset V$ corresponding to 
non-isotrivial fibrations.
It is naturally obtained by replacing  condition (1) above by the slightly stronger condition
\begin{itemize}
\item[(1')] $\Delta=g_2^3 - 27 g_3^2$ and $g_2^3$ are not proportional 
(which also excludes the case $g_2 \equiv 0$ and the case $\Delta\equiv 0$).
\end{itemize}
Since the $j$-invariant is the quotient of these two expressions, this translates
directly into several equivalent,
more geometric, characterisations of points $v\in V'$:
\begin{enumerate}
\item
The $j$-invariant of the Jacobian elliptic fibration $\calE_v$ is non-constant.
\item
The Jacobian elliptic fibration $\calE_v$ is not isotrivial.
\item
The number of singular fibres of multiplicative 
type $I_\nu$ or additive type 
$I_\nu^*$, $\nu>0$, is positive.
\end{enumerate}
Clearly $V'$ is invariant (as  a set) under the action of $\SL(2,\CC)$ and we denote the quotient by $\calF' \subset \calF$.  This is the open subset parameterising all non-isotrivial Jacobian 
fibrations of $K3$ surfaces. 

To decompose $V'$ (and thus $\calF'$) in a geometrically meaningful way, \cite{BPT} exploit the 
\emph{monodromy group} of elliptic fibrations: 
the complement $\calE'$ of the union of singular fibres of a Jacobian fibration $\calE$ is topologically equivalent
to a torus bundle over the base punctured at the critical values of the fibration.
The image of the associated monodromy representation in the automorphisms
of the first homology of a fibre is orientation preserving. Upon the choice of 
a basis, it becomes the \emph{monodromy group} $\Gamma(\calE)$,
a subgroup of $\SL(2,\ZZ)$, well-defined up to conjugacy, see \cite[VI,3]{Mi89}.
Note that according to this definition we will consider monodromy
groups as subgroups of $\SL(2,\ZZ)$, and keep in mind that it is the
conjugacy class which is the invariant. We will use $\sim$ to denote conjugate subgroups.
The group $\Gamma(\calE)$ determines a group $\barGamma(\calE) \subset \PSL(2,\ZZ)$ which we will call the {\em modular monodromy group}.

The initial observation in \cite{BPT} is the following semi-continuity property
with respect to the Zariski topology on the quasi-projective variety $\calF'$:
the monodromy group can only change if the configuration of singular fibres 
changes, i.e.\ if some fibres split or come together. The latter case only occurs
on closed subsets of the base of a family and the monodromy group will be
a subgroup after the process. Therefore, the property that the monodromy group is a subgroup
of a fixed subgroup of $\slz$ (up to conjugacy) defines a closed subset.

This can be rephrased as follows. We consider the set $P$ of all conjugacy classes of subgroups of $\slz$ together with the partial ordering induced by inclusion.
On this we define the {\em Alexandroff topology }  for which a set $U$ is open if and only if $p\in U, p\leq q$ implies $q\in U$. The semicontinuity property observed by 
Bogomolov, Petrov and Tschinkel then says that the map  $V' \to P$ is continuous with respect to the Alexandroff topology on $P$ and the Zariski topology on $V'$. 
It implies that the fibres $V'_{\Gamma}$ of this map are locally closed for a fixed (conjugacy class of a) subgroup  $\Gamma \subset \slz$.  We can decompose each 
such set into its irreducible components $V'_{\Gamma}= \cup_i V'_{\Gamma,i}$.  Bogomolov, Petrov and Tschinkel showed that the number of possible monodromy groups is finite and in this way obtained the 
\begin{lem}[cf.\ \cite{BPT}, Lemma 3.1]
The variety $V'$ is a finite 
union of locally closed irreducible subvarieties $V'_{\Gamma, i}$,
each preserved under the action of $\SL(2,\CC)$ such that for every $v\in V'_{\Gamma, i}$
one has $\Gamma( \calE_v) \sim \Gamma.$
\end{lem}

We denote $\calF'_{\Gamma, i}=V'_{\Gamma, i}/\!\!/ \SL(2,\CC)$ and in this way one obtains a finite decomposition 
\begin{equation} \label{equ:monodromystart}
\calF'= \cup_i \calF'_{ \Gamma, i}
\end{equation}
into irreducible locally closed subvarieties. 

\begin{dfn} We refer to (\ref{equ:monodromystart}) as the {\em monodromy stratification } of the moduli space of $K3$ surfaces with a non-isotrivial Jacobian fibration, and call the subvarieties 
 $\calF'_{ \Gamma, i}$ the  {\em monodromy strata } of $\calF'$.
 \end{dfn}
 
While the objective of \cite{BPT} is the rationality of the monodromy strata $\calF'_{ \Gamma, i}$
in the moduli space, for our arguments it is important to note the following
maximality condition
\begin{equation}
\label{closure}
v\in \overline{V'_{ \Gamma, i}}\,\setminus V'_{ \Gamma, i}
\quad \implies \quad
\Gamma(\calE_v)\not\sim\Gamma.
\end{equation}

\subsection*{Shimada strata}\label{subsec:root}
We have already seen the root lattice $R(\calE)$ associated to a Jacobian fibration $f: \calE \to \PP^1$ of a $K3$ surface. In general $R(\calE)$ is not saturated in $\NS(\calE)$ and we denoted its saturation by
$L(\calE)$. The overlattice $L(\calE)$ of $R(\calE)$ defines, and can be reconstructed, from the finite group  $G(\calE)=L(\calE)/R(\calE) \subset D(R(\calE))$. Here we use standard notation and standard facts form lattice theory: if $L$ 
is an even  
lattice we denote by  $L^{\vee}$ the {\em dual lattice } which can be defined as $L^{\vee}=\{x \in L_{\QQ} \mid (x,y) \in \ZZ {\mbox{ for all }} y \in L\}$.  The {\em discriminant } of $L$ is the quotient 
$D(L) = L^{\vee}/L$. This is a finite group and it is equipped with a quadratic form with values in $\QQ/ 2\ZZ$. The importance of the discriminant group in lattice theory was fully developed in Nikulin's 
seminal  paper \cite{Nik}. Overlattices of $L$ correspond to isotropic subgroups $G \subset D(L)$, see \cite[Proposition 1.4.1]{Nik}.

In his paper \cite{shimada} Shimada investigated the question which pairs $(R,G)$, where $R$ is an $ADE$ root lattice and $G \subset D(R)$ is a finite (isotropic) subgroup of the discriminant group, 
can occur as $(R(\calE), G(\calE))$ and 
how many irreducible families these Jacobian fibrations form. At this point 
it is important to note that in general a pair $(R,G)$ does not necessarily determine a unique family. 
There are several reasons for this. Firstly, the abstract group $G$ does not necessarily define a unique overlattice $L$ of $R$, more precisely one must choose a specific isotropic subgroup $G$ in $D(R)$ and it can happen 
that $G$ can be embedded in several ways as an isotropic subgroup in $D(R)$. Secondly, given $L$ one must also specify a primitive embedding of $M=U \oplus L$ into the $K3$ lattice
\begin{equation*}\label{equ:K3lattice}     
L_{K3}= 3U + 2E_8
\end{equation*}
where $E_8$ is the unique even unimodular negative definite lattice of rank $8$ and the  lattice $M$ may have several such embeddings (modulo the orthogonal group $\O(L_{K3})$). Given such an embedding one can consider 
the moduli space of {\em lattice  polarised} $K3$ surfaces with lattice polarisation $M$.   
Note however, that a lattice polarization contains possibly more information than the pair $(R,G)$ as different lattice polarizations can give rise to the same elliptic $K3$ surface with a given configuration of singular fibres.
We shall discuss this relationship  in detail in Sections \ref{sec:mirrorpol} and \ref{sec:modulimaps}.
The Shimada strata correspond to finite quotients of (open subsets) of moduli spaces of lattice polarised $K3$ surfaces. 
We also note that the moduli spaces can have $1$ or $2$ 
components and this, thirdly, can also increase the number of strata. If there are $2$ such components then they are complex conjugate to each other, 
i.e.\ surfaces of one component are complex conjugate to those of the other component, cf.\ \cite[p.206]{FM}.

Shimada's classification \cite{shimada2} gives the following result.

\begin{teo}[Shimada] The following holds:
\begin{itemize}
\item[{\rm{(1)}}] There are $3278$ different root lattices which occur as $R(\calE)$ for a Jacobian fibration $\calE$ of a $K3$ surface. Of these $2953$ belong to non-extremal fibrations.
\item[{\rm{(2)}}]  This decomposes the set of all Jacobian fibrations of $K3$ surfaces into $3932$ connected families, of which $3469$ belong to  non-extremal fibrations.
\end{itemize}
\end{teo}

Shimada's result defines a stratification of the moduli space $\calF'$ into maximal (with respect to inclusion) locally closed irreducible subvariety $\calFF'_{R,G,j}$ and this defines a finite 
decomposition
\begin{equation}\label{equ:rootstrat}
\calF' = \cup_j \calFF'_{R,G,j}.
\end{equation}

\begin{dfn} We call the decomposition (\ref{equ:rootstrat}) the {\em Shimada stratification} of $\calF'$ and the sub\-varieties 
$\calFF'_{R,G,j}$ the  {\em Shimada strata} of $\calF'$.
\end{dfn}
For future use we also note the following remark which is simply obtained by counting the conditions imposed by the rank of the N\'eron-Severi group.
\begin{rem}
The dimension of a component  $\calFF'_{R,G,j}$ is given by
\begin{equation}\label{rem:dimcomp}
\dim \calFF'_{R,G,j} = 18 - \rank(R).
\end{equation}
\end{rem}
We notice that this only depends on $R$ and not on $G$ nor the specific component we are considering. 
Note that all components of   $\calFF'_{R}$ have the same dimension, which we will also refer to as the dimension of $\calFF'_{R}$.  
  
We note that the decomposition in Shimada strata (\ref{equ:rootstrat}) is a topological poset stratification in the
sense of \cite{YaYo}.

%%%%%%%%%%%%%%%%%%%%%%%%%%%%%%%%%%%%%%
\section{The main classification result}\label{sec:mainresult}
The primary goal of this paper is to compare the two stratifications of the moduli space $\calF'$ of non-isotrivial Jacobian fibrations on $K3$ surfaces which we have defined above, namely the monodromy stratification 
and the Shimada stratification. More precisely, we want to determine all positive dimensional monodromy and Shimada strata whose generic points coincide. In other words, we want to determine all pairs $\calF'_{\tilde{\Gamma,i}}$ 
and $\calFF'_{R,G,j}$ such there intersection is non-empty and open (and hence dense) in both  $\calF'_{{\Gamma,i}}$
and $\calFF'_{R,G,j}$. This leads us to the definition 
\begin{dfn}
A positive dimensional irreducible closed subset $\calA \subset \calF'$ is called an {\em ambi-typical stratum } if there is a monodromy stratum   $\calF'_{\Gamma,i}$ and a Shimada stratum $\calFF'_{R,G,j}$ such that
\begin{equation}
\calA=\overline{\calF'_{\Gamma,i}} = \overline{\calFF'_{R,G,j}}. 
\end{equation}
(The monodromy stratum and the Shimada stratum are then uniquely defined.)
\end{dfn}

Our main result is
\begin{teo}\label{teo:mainresult}
There are $50$ ambi-typical
strata in $\calF'$. 
\end{teo}

As we have discussed, a Shimada stratum determines a pair $(R,G)$, but not necessarily the other way round. There are, however, cases where the root lattice defines a unique Shimada stratum. Indeed, this is the case for most of the 
ambi-typical strata. 
The data $(R,G)$ for the ambi-typical strata which are referred to in the following theorem are listed in Table \ref{table:list}.

\begin{teo}\label{teo:liststrata2}
In all but $3$ cases the root lattice $R$ of an ambi-typical stratum $\calA$ determines a unique Shimada stratum. The exceptions are:
\begin{itemize}
\item[{\rm{(1)}}]  The root lattice $D_4 + 2A_6 + A_1$ in 38/39 determines two Shimada strata, these are conjugate complex to each other. 
\item[{\rm{(2)}}]  The root lattice  $2A_3 + 8A_1$ in 6 determines two non-conjugate Shimada strata. Only one of these occurs as an ambi-typical stratum.\footnote{We will say more on this in Remark \ref{rem:explanationShimadastrata}.}
\end{itemize} 
\end{teo}
 
By definition a monodromy stratum determines a monodromy group $\Gamma$, but the converse will in general not be true. In fact, as we will see in the rest of the paper, monodromy strata are determined 
by some fairly involved data. This includes the  modular monodromy group $\barGamma \subset \PSL(2,\ZZ)$, but also information about the $j$-invariant $j(\calE): \PP^1 \to \PP^1$. We shall see in 
Section \ref{sec:jfactor} that this can be factored in a unique way as $j(\calE)= \jGamma \circ j_{\calE}$ where $\jGamma$ is a Belyi map.
The additional data which determine a monodromy stratum are given by 
a space of rational maps $j_\calE:\PP^1\to\PP^1$ and a rational
family of divisors on $\PP^1$ corresponding to the $*$-fibres. However, in the case of ambi-typical strata these data are fully determined by much more accessible ones:

\begin{teo}\label{teo:liststrata}
The $50$ ambi-typical monodromy strata
determine the list of invariants given in Table \ref{tab:jmaps}, consisting of
\begin{enumerate}
\item
the local monodromies of $\jGamma$ at $0,1,\infty$,
\item
the branching of $j_\calE$ for a generic element $\calE$ at the non-critical pre-images of $0,1$ 
and at poles of $\jGamma$,
\item
the number of $*$ fibres.
\end{enumerate}
Conversely, the data are pairwise distinct and each determines a unique
monodromy stratum with the following exception: 
the strata 38/39 correspond to two distinct maps
$\jGamma$ having the same combinatorial type but with monodromy
factorizations in distinct conjugation classes.
\end{teo}

The proof of these theorems is quite involved and takes up most of the remaining paper. Here we shall give a rough outline of our strategy:
\begin{itemize}
\item
In Section \ref{sec:config}
we study  configurations of singular fibres where we use the work of Kloosterman
\cite[sec.4]{Kl}.
This leads to a semi-continuous invariant which induces a stratification
which refines both the stratifications by monodromy and by
root lattice. Hence it suffices to look for ambi-typical fibre configuration strata. 
We find that to be ambi-typical poses severe restrictions
on the possible singular fibres.
\item
In Section \ref{sec:jfactor} we recall the factorisation $j(\calE)= \jGamma \circ j_{\calE}$ of the $j$-invariant.
The restrictions on fibre configurations
then translates into branching properties of $\jGamma$ and $j_\calE$
over the points $0,1,\infty$.
\item
In Section \ref{sec:Euler} we use 
an Euler number restriction to find 
a list of candidates for the modular monodromy group  $\barGamma$
and give an upper bound for the dimension of the corresponding  strata.
\item
In Section \ref{sec:Weier}
we parameterise closed subsets of $V'$ which give Weierstra{\ss}
data for all possible ambi-typical  strata with
modular monodromy group $\barGamma$ of low index, i.e.\
at most $6$.
\item
In Section \ref{sec:low}
we determine the ambi-typical strata among these
families of Weierstra{\ss} data and their invariants.
\item
In Section \ref{sec:high}
we address the classification of ambi-typical strata
with modular monodromy group $\barGamma$ of high index, i.e.\ at least $7$.
In the first step we determine the possible corresponding root
lattices and the topological types of the $j_\calE$ factor of the
$j$-function. 
In the second step we determine the topology of the maps $\jGamma$
and the corresponding monodromy groups.
\end{itemize}
This program will finally allow us to complete
the proofs of our theorems.

%%%%%%%%%%%%%%%%%%%%%%%%%%%%%%%%%%%%%%%%
\section{Singular fibre configurations and generic invariants}
\label{sec:config}

In this section we consider the configuration of singular fibres of
elliptic surfaces. The possible singular fibre types have been classified by
Kodaira and  are listed in Table \ref{table:kodaira}, see \cite[Table V.6]{BHPV}
with some of their invariants.
The relation between singular fibre types and Weierstra{\ss} data is given in Table \ref{table:tate}.
This contains all information necessary to apply the \emph{Tate algorithm} over the complex
numbers, i.e.\ to determine the fibre types from the vanishing orders of $\nu_2=\nu(g_2), \nu_3=\nu(g_3)$ 
and $\nu_{\Delta}=\nu(\Delta)$ of $g_2,g_3$ and $\Delta$ respectively.

\begin{table}[h]
\centering
\begin{tabular}{|c|c|c|c|c|}
 fibre type & 
$ADE$-type  & 
Euler number  & local monodromy & local $j$-expansion
\\ \hline
 & & & &\\[-4mm] \hline
 & & & &\\[-4mm]
 $I_0$ & -- &
$0$ & $(\begin{smallmatrix} 1 & 0 \\ 0 & 1 \end{smallmatrix})$ 
& 
\begin{tabular}{c} $j=s^{3k}$, $j=1+s^{2k}$\\ or $j\neq0,1$ \end{tabular}
\\[-4.2mm] &&&&
\\ \hline  & & & &\\[-4mm]
 $I_1$ & -- &
$1$ & $(\begin{smallmatrix} 1 & 1 \\ 0 & 1 \end{smallmatrix})$ & pole of order $1$
\\[-4.2mm] &&&&
\\ \hline & & & &\\[-4mm]
 $I_b\:\: (b\geq 2)$ &
$A_{b-1}$ & $b$ 
& $(\begin{smallmatrix} 1 & b \\ 0 & 1 \end{smallmatrix})$ & pole of order $b$
\\[-4.2mm] &&&&
\\ \hline & & & &\\[-4mm]
 $I^*_0$ &
$D_{4}$ & $6$ 
& $(\begin{smallmatrix} -1 & 0 \\ 0 & -1 \end{smallmatrix})$ & same as in first case
\\[-4.2mm] &&&&
\\ \hline & & & &\\[-4mm]
 $I^*_b \:\:  (b\geq 1)$ &
$D_{4+b}$ & $6+b$ 
& $(\begin{smallmatrix} -1 & -b \\ 0 & -1 \end{smallmatrix})$ & pole of order $b$
\\[-4.2mm] &&&&
\\ \hline & & & &\\[-4mm]
 $II$ & -- &  $2$ 
& $(\begin{smallmatrix} 1 & 1 \\ -1 & 0 \end{smallmatrix})$ & $j = s^{3k+1}$
\\[-4.2mm] &&&&
\\ \hline & & & &\\[-4mm]
 $III$ & $A\sb 1$ & $3$ 
& $(\begin{smallmatrix} 0 & 1 \\ -1 & 0 \end{smallmatrix})$ & $j = 1+s^{2k+1}$
\\[-4.2mm] &&&&
\\ \hline & & & &\\[-4mm]
 $IV$ & $A_2$ & $4$ 
& $(\begin{smallmatrix} 0 & 1 \\ -1 & -1 \end{smallmatrix})$ & $j = s^{3k+2}$
\\[-4.2mm] &&&&
\\ \hline & & & &\\[-4mm]
 $IV^*$ &
$E\sb 6$ & $8$ 
& $(\begin{smallmatrix} -1 & -1 \\ 1 & 0 \end{smallmatrix})$ & $j = s^{3k+1}$
\\[-4.2mm] &&&&
\\ \hline & & & &\\[-4mm]
 $III^*$ &
$E\sb 7$ & $9$ 
& $(\begin{smallmatrix} 0 & -1 \\ 1 & 0 \end{smallmatrix})$ & $j = 1+s^{2k+1}$
\\[-4.2mm] &&&&
\\ \hline & & & &\\[-4mm]
 $II^*$ &
$E\sb 8$ & $10$ 
& $(\begin{smallmatrix} 0 & -1 \\ 1 & 1 \end{smallmatrix})$ & $j = s^{3k+2}$
 \\[-4.2mm] &&&&
\\ \hline 
\end{tabular}
\caption{Fibre types of elliptic surfaces}
\label{table:kodaira}
\end{table}

\begin{table}[b]
\[
\begin{array}{c|c@{\hspace*{2mm}}c@{\hspace*{1mm}}c|c|c@{\hspace*{2mm}}c@{\hspace*{1mm}}c|c|c|c|c|c|c|c}
\text{type} & &I_0& & I_k , k>0 & & I_0^* & & I_k^*, k>0
& II & III & IV & IV^* & III^* & II^* \\
\hline
j & 0 & 1 & gen. & \,\infty\, & 0 & 1 & gen. & \infty & 0 & 1 & 0 & 0 & 1 & 0 \\
\nu_2 & >\!0 & 0 & 0 & 0 & >\!2 & 2 & 2 & 2 & >\!0 & 1 & >\!1 & >\!2 & 3 & >\!3 \\
\nu_3 & 0 & >\!0 & 0 & 0 & 3 & >\!3 & 3 & 3 & 1 & >\!1 & 2 & 4 & >\!4 & 5 \\
\nu_{\Delta} & 0 & 0 & 0 & k & 6 & 6 & 6 & k\!+\!6 & 2 & 3 & 4 & 8 & 9 & 10 \\[2mm]
\end{array}
\]
\caption{Tate algorithm (cf.\ \cite[Table IV.3.1]{Mi89})}
\label{table:tate}
\end{table}

In the ensuing discussion we will use notation as introduced in Kloosterman's paper  \cite{Kl}.
For a Jacobian fibration $f: S \to \PP^1$ let $C(f)$ denote the configuration of its singular fibres.
Since we are restricting ourselves to non-isotrivial fibrations, we can assume that  
$C(f)$ contains at least one fibre of type $I_\nu$ or $I_\nu^*$ with $\nu>0$. 
For an abstract configuration $C$ of singular fibres we define 
\begin{equation}\label{equ:L(C)}
L(C):= \{[f:  S \rightarrow \PP^1] \in \calF' \mid C(f)=C\}.
\end{equation}

\begin{rem}
We recall from \cite[Lemma 4.6]{Kl} that $L(C) \subset \calF'$ is constructible. 
Fibre configurations are partially ordered by degeneration of several singular fibres into fewer
more complicated ones. Degeneration occurs on closed subsets, thus semi-continuity holds
and $\calF'$ is topologically stratified by the components of the sets $L(C)$.
\end{rem}

\begin{dfn}
We call $L(C)$ a {\em configuration locus } and the components of the sets $L(C)$ the {\em configuration strata}.   
\end{dfn}

As we will need the dimension of the components of $L(C)$ we recall
\begin{lem}\cite[Lemma 4.6, Assumption 4.3]{Kl}
\label{dimLoc} 
Assume that $C$ is a configuration of singular fibres arising from a non-isotrivial Jacobian fibration.  
Then all components of $L(C)$ have the dimension
\[ 
\dim L(C)= \# \{\mbox{singular fibres}\} + 
\# \{ \mbox{fibres of type } II^*, III^*, IV^*, I_\nu^*, \nu\geq 0\} - 6. 
\]
\end{lem}

By construction, each irreducible component of $L(C)$ corresponds to Jacobian elliptic fibrations
with topologically equivalent complements of singular fibres. In particular
all elements in an irreducible component of $L(C)$ have the same monodromy group
and each irreducible component $\calF'_{\Gamma,i}$ is a finite union of irreducible components
of certain $L(C)$. The only one of these, which is open, corresponds thus to
the configuration of a generic member of the monodromy stratum.
We call this configuration $C(\calF'_{\Gamma, i})$ and
all components of $L(C(\calF'_{\Gamma, i}))$ have the same dimension 
as $\calF'_{\Gamma,i}$. 
\medskip

By definition all Jacobian fibrations $\calE \in L(C)$ have the same root configuration and hence the same rank 
$\rho_{\tr}(\calE)$ of the trivial lattice.  

From this one easily obtains 
\begin{lem}\cite[Prop. 4.7]{Kl}
\label{dimlem} 
Let $C$ be a configuration of singular fibres, containing at least one $I_\nu$ or $I_\nu^*$-fibre
($\nu>0$) with $L(C) \neq \emptyset$ and let $\calE$ be any element in $L(C)$. 
Then
\[ 
\dim L(C)= 20-\rho_{\tr}(\calE)-\# \{\mbox{fibres of type } II, III \mbox{ or } IV \}.
\]
\end{lem}
\begin{proof}
We use the above lemma and apply the fact that the rank of the trivial lattice $\rho_{\tr}(\calE)$ is given by
\[
\begin{aligned}
\rho_{\tr}(\calE) \quad = \quad
&2+\sum_{\mbox{$F$ multiplicative}} (e(F)-1)+ \sum_{\mbox{$F$ additive}} (e(F)-2)
\\
= \quad
&26 - \# \{\mbox{multiplicative fibres}\} - 2\# \{\mbox{additive fibres}\}.
\end{aligned}
\]
Here $e(F)$ denotes the Euler number of a fibre $F$.
\end{proof}

Any component of a stratum $L(C)$ defines an embedding  $U + R \hookrightarrow L_{K3}$ (up to isometries in $\O(L_{K3})$) and thus an embedding into a root stratum. 
Hence the stratification given by the configuration strata refines both the monodromy stratification and the root lattice stratification. In particular, every Shimada stratum $\calFF'_{R,G,j}$ 
and every monodromy stratum
$\calF'_{\Gamma,i}$ contains a unique configuration stratum which is open and dense in this stratum.

This also allows us to talk about the properties of a {\em generic } 
element of a Shimada stratum or a monodromy stratum. In particular, the following data, in addition to the configuration $C(\calE)$,  are invariant for all 
members $\calE$ of a configuration stratum: 
the monodromy group $\Gamma(\calE)$ (as we have already pointed out), the root lattice $R(\calE)$, the trivial lattice $\NS_{\tr}(\calE)$, the saturation $L(\calE)$ of $R(\calE)$ in $\NS(\calE)$   
and the torsion of the Mordell-Weil group $\MWt(\calE) \cong L(\calE)/R(\calE)$.

We are now ready to start the classification of ambi-typical strata. For this we first collect necessary conditions which a generic element of such a stratum must fulfill. Indeed, the first restrictions on the possible configurations
can be derived easily.  

\begin{pro}\label{prop:restricitionfibres}
\label{fibres1}
Suppose $\Gamma$ is a proper subgroup of $\slz$ of finite index and $\calE \in \calF'_{\Gamma,i}$ a generic element of a monodromy stratum with monodromy group $\Gamma$. 
Then the following properties are equivalent:
\begin{enumerate}
\item
The Jacobian fibration $\calE$ has no fibres of type $II, III \mbox{ or } IV$.
\item
The dimensions of the Shimada stratum $\calFF'_{R(\calE),G(\calE),j}$ 
which contains $\calE$ and of $\calF'_{\Gamma,i}$ coincide.
\end{enumerate}
In particular, the generic element of an ambi-typical stratum does not have any fibres of type  $II,III$ or $IV$.
\end{pro}
\begin{proof}
Since $\calE$ is a generic element of $\calF'_{\Gamma,i}$ it lies in a unique open and dense configuration stratum, namely a component of $L(C(\calE))$. Hence 
\[
\dim \calF'_{\Gamma,i}= \dim L(C(\calE)) =20-\rho_{\tr}(\calE)-\# \{\mbox{fibres of type } II, III \mbox{ or } IV \}
\]
where the last equality follows from Lemma \ref{dimlem}.
Since the Shimada stratum has dimension equal to $20-\rho_{\tr}(\calE)$ the claim follows immediately.
\end{proof}

\begin{rem}
The above lemma shows: if  the generic element of a monodromy stratum $\calF'_{\Gamma,i}$ has fibres of type $II,III$ or $IV$, then this monodromy stratum is contained in a 
Shimada stratum  of strictly bigger dimension.
\end{rem}

We can find further severe restrictions on ambi-typical strata by studying how the monodromy behaves under certain fibre degenerations.

\begin{pro}
\label{badfibres}
Assume that $\calE$ is a generic element of an ambi-typical stratum. Then 
$\calE$ has no singular fibres of type $II, III, IV, II^*$ or $III^*$.
If $-\id \in \Gamma(\calE)$ then $\calE$ also has no singular fibres of type $I_{>0}^*, IV^*$.
\end{pro}

\proof
We have already seen in Proposition \ref{fibres1} that fibres of type $II,III$ or $IV$ cannot exist.

Our strategy is the following:  suppose that $\calE$ has a singular fibre of type $II^*, III^*, I_{>0}^*$ or $IV^*$, where in the latter two cases we assume that $-\id \in \Gamma(\calE)$. We will then
construct a family containing $\calE$ where the monodromy remains constant, but  where the rank of the trivial N\'eron-Severi group drops. This contradicts the assumption that $\calE$ is a generic element of 
a Shimada stratum. 

To do this pick a Weierstra{\ss} datum $g_2,g_3$ for $\calE$. By inspection of the Tate Table \ref{table:tate},  the presence
of a *-fibre implies that $g_2,g_3$ have a common zero at this fibre. More precisely, we
can factor $g_2=\check g_2 x^2, g_3= \check g_3 x^3$ where $x$ is a linear form vanishing at this fibre. We then consider the family of Weierstra{\ss} data 
$\check g_2 (x-t)^2, \check g_3 (x-t)^3$ where $t$ varies. It has the same 
$j$-invariant as $\calE$, hence the projective monodromy group $\barGamma(\calE_t)$ is constant,
see \cite[p.211]{BHPV}.
If in addition $-\id\in \Gamma(\calE)$, then also the monodromy group $\Gamma(\calE_t)$ is constant,
since it can only get smaller on a closed subset. By Table  \ref{table:kodaira}, for fibres of type $II^*, III^*$ a power
of the local monodromy is $-\id$, so this assumption is fulfilled. For fibres of type $I_{>0}^*, IV^*$ this is part of our assumptions.
Hence the monodromy group remains constant if we vary $t$, however for 
$t\neq0$ the fibre of $*$-type is replaced by
an $I_0^*$ fibre and the fibre without $*$ that has the same local monodromy up to $-\id$, 
see Table  \ref{table:kodaira}.
But then, again by Table  \ref{table:kodaira}, this implies that the  rank of the trivial N\'eron-Severi group $\NS_{\tr}(\calE_t)$ group drops as $t\neq 0$, giving the  desired contradiction.
\qed

\begin{rem}
A typical example for the situation discussed is when a type $II^*$-fibre splits into a type $IV$- and an $I_0^*$-fibre. Here the rank of the 
trivial N\'eron-Severi drops by $2$. 
This gives examples where a Shimada stratum is contained as a proper subset in a bigger 
monodromy stratum.  
\end{rem}

%%%%%%%%%%%%%%%%%%%%%%%%%%%%%%%%%%%%
\section{Factorisations of the $j$-invariant}
\label{sec:jfactor}
\label{jcomposition}

An essential tool for our classification result is the $j$-function associated to a Jacobian fibration $f \colon \calE \to \PP^1$. As usual we denote the upper half plane by $\HH_1$ and set
$$
\overline{\HH}_1= \HH_1 \cup \QQ \cup \{\infty\}.
$$
The quotient 
$$
X(1) := \overline{\HH}_1 / \pslz  \cong \PP^1
$$
is the modular curve of level $1$, namely the compactification of the $j$-line. 
If the Jacobian fibration $f: \calE \to \PP^1$ has the monodromy group $\Gamma \subset \slz$ then we denote its image in $\pslz$ by $\barGamma$. This defines the modular curve 
$$
X({\barGamma}) := \overline{\HH}_1 / \barGamma 
$$
which in our case is again isomorphic to $\PP^1$. The $j$-function 
$$
j(\calE) \colon \PP^1 \to X(1) \cong \PP^1 
$$
has a unique  factorisation  
\begin{equation}\label{equ:jfactor}
j(\calE)= \jGamma \circ j_{\calE} 
\end{equation}
up to deck transformations of $\jGamma$,  
where 
$$
\jGamma: X(\barGamma)\cong \PP^1 \to X(1) \cong \PP^1
$$
is the natural quotient map. 
We refer the reader also to \cite[p. 1107]{BPT}.
Another characterization of $\barGamma$, using the factorisation of $j(\calE)$, is provided by 

\begin{lem}
\label{lem:BT}
The group $\barGamma$ has the following minimality property:
\begin{eqnarray}
\label{BT2new}
&& \text{if $j(\calE)$ factors through $j_{\barGamma'}$ then } 
\barGamma\subset\barGamma' \text{ up to conjugation}.
\hspace*{17mm}
\end{eqnarray}
\end{lem}

\begin{proof}
Given a  holomorphic surjective map $j:\PP^1\to X(1)$ there exists by \cite[Lemma 2.3]{BT}  
an, up to conjugation unique, subgroup
 $\barGamma_{\!j}\subseteq\pslz$ of finite index such that
\begin{eqnarray}
\label{BT1}
&& j \text{ factors through } j_{\barGamma_{\!j}},
\\
\label{BT2}
&& \text{if $j$ factors through $j_{\barGamma'}$ then } 
\barGamma_{\!j}\subseteq\barGamma' \text{ up to conjugation}.
\hspace*{17mm}
\end{eqnarray}
We apply this result to $j(\calE)$ and thus have to show that $\barGamma_{\!j(\calE)}=\barGamma$.
The inclusion $\barGamma_{\!j(\calE)}\subseteq\barGamma$ (up to conjugacy) follows from \eqref{BT2}.

For the other direction we use the discussion  in \cite[p.211]{BHPV}.
Every $j$-map gives rise to a representation class $r = r(j)$ with values
in $\pslz$ and, in particular, $r(j_{\barGamma'})$ has image $\barGamma'$.
By Kodaira, $r(j(\calE))$ is the composition of the monodromy homomorphism with
$\slz\to\pslz$ and hence has image $\barGamma$.
By property \eqref{BT1} the map $j(\calE)$ factors through $j_{\barGamma(j_\calE)}$ and hence
$\Im r(j(\calE)) \subseteq \Im r(j_{\barGamma(j_\calE)})$, which implies that 
$\barGamma \subseteq \barGamma(j_\calE)$ (up to conjugation).
\end{proof}

\begin{rem}\label{rem:equivalentdatacovering}
\label{coversubgroup}
In the sequel it will be very helpful to investigate subgroups of $\pslz$ by means
of the following three sets, where $S_d$ denotes the symmetric group of $d$ elements:
\begin{enumerate}
\item
subgroups $\barGamma$ in $\pslz$ of index $d$ up to conjugacy in $\pslz$,\\[-2mm]
\item
(connected) branched covers $j:C\to X(1)$ of degree $d$ up to equivalence of covers,
branched only over $0,1, \infty$ with multiplicities $1,3$
over $0$ and $1,2$ over $1$,\\[-2mm]
\item
homomorphisms $\mu:\pi_1(\CC\setminus\{0,1\}) \to S_d$
up to conjugacy in $S_d$, such that simple loops around $0$, resp.\ $1$
map to elements of order $1$ or $3$, resp.\ $1$ or $2$,
and the image acts transitively.
\end{enumerate}
Since
$\pi_1(\CC\setminus\{0,1\})$ is freely generated by simple loops around $0$ and $1$, while
$\pslz$ as the free product $\ZZ/2\ast\ZZ/3$ is generated by
$(\begin{smallmatrix} 0 & 1 \\ -1 & 0 \end{smallmatrix})$ and $(\begin{smallmatrix} 1 & 1 \\ -1 & 0 \end{smallmatrix})$,
these sets are in bijective correspondence with each other via the following maps:
\begin{description}
\item[(1)$\to$(2)]
Given $\barGamma$, associate the branched cover 
$\jGamma:X(\barGamma)\to X(1)$.
\item[(1)$\to$(3)]
Given $\barGamma$, associate left multiplication 
$\pslz \to S(\pslz/\barGamma)\cong S_d$ on cosets
and compose with $\pi_1(\CC\setminus\{0,1\}) \to \pslz$ to get $\mu$.
\item[(2)$\to$(3)]
Given $j:C\to \PP^1$, restrict to $\CC\setminus\{0,1\}$, which is a topological
cover of degree $d$ and associate the representation
$\mu:\pi_1(\CC\setminus\{0,1\}) \to S_d$ by permutations of a fibre.
\item[(3)$\to$(1)]
Given $\mu$, note that by our assumptions this factors through a homomorphism $\bar\mu:\pslz\to S_d$ and
associate the stabiliser subgroup $\barGamma\subset\pslz$ of $1$.
\item[(3)$\to$(2)]
Given $\mu$, associate the connected topological cover over $\CC\setminus\{0,1\}$
of degree $d$ that extends to a branched cover $j:C\to \PP^1$ by Riemann existence.
\end{description}
\end{rem}

Using the remark we obtain a characterisation of the factorisation \eqref{equ:jfactor} from Lemma \ref{lem:BT}:

\begin{lem}
\label{unique}
Given the holomorphic map $j(\calE):\PP^1\to X(1)$ a factorisation 
$j_2\circ j_1$ is equivalent to $\jGamma\circ j_\calE$ if and only if
\begin{eqnarray}
\label{belyj1}
&j_2& \text{
is branched only over $0,1, \infty$ with multiplicities $1,3$
over $0$}
\\ && \text{
and $1,2$ over $1$,}
\notag\\ 
\label{belyj2}
&j_1& \text{
has no proper left factor $j'$, such that $j_2\circ j'$ has the property above.}
\hspace*{7mm}
\label{maximality}
\end{eqnarray}
\end{lem}

In the topological analysis of more arbitrary (connected) branched covers of $\PP^1$ we first fix a base point not in the branch locus.
We then use the monodromy homomorphism taking homotopy
classes of closed paths inside the complement of the branch locus around this base point to the permutation group
of the fibre over the base point. 
This defines a local monodromy which associates to a branch point
the conjugacy class of the monodromy of the positively oriented boundary of a sufficiently small disc centred at the
branch point. 
While the monodromy depends on the choice of a fibre, the conjugacy class does not, and it is thus an
invariant, which we may as well give by the cycle type of the permutation or the corresponding partition of the fibre cardinality. 
The sizes of the parts of that partition are exactly the multiplicities of the pre-images of the branch point.
\medskip

Accordingly,
we can associate to $\jGamma$ the three conjugacy classes $C_0,C_1,C_\infty$
corresponding to the branch points $0,1,\infty$.
In practice we will give the three conjugacy classes as a $3$-tuple of partitions
of $\deg \jGamma$.
The degree of $\jGamma$ and the numbers $e_3$ and $e_2$ of fixed points of 
elements in $C_0$ and $C_1$ respectively,
are related by the following congruences
\begin{equation}
\label{jg_cong}
\deg \jGamma \equiv_2 e_2, \qquad
\deg \jGamma \equiv_3 e_3.
\end{equation}
This follows since the difference of the two numbers in question is the sum of cycle lengths of transpositions and $3$-cycles respectively.
Note that for the subgroup $\barGamma$ of $\pslz$ corresponding to $\jGamma$,
the numbers 
$e_2,e_3$ count the non-ramified pre-image of $1$ and $0$ under $\jGamma$, which
we call \emph{2-torsion} points and \emph{3-torsion} points respectively.
\medskip

For the next lemma, we only require a small part of the branching datum
of $j_\calE$, namely the partitions corresponding to the local monodromy of  $j_\calE$  around
the 2-torsion and 3-torsion points.
In particular, if $f^{-1}(p), p\in \PP^1$ is a fibre of type $IV^*$, then the map $j_\calE$ is locally unbranched and
maps the point $p$  to a 3-torsion point.

\begin{lem}
\label{rami}
The following properties hold for the invariants associated to
the factors $\jGamma$ and  $j_\calE$ of a generic element $\calE$ in an ambi-typical monodromy
stratum:
\begin{enumerate}
\item
partitions associated to $j_\calE$ at 2-torsion points have only even size parts,
\item
the parts of the partitions associated to $j_\calE$ at 3-torsion points all have size divisible by $3$ except for a
total of $\# IV^*$ parts of size equal to $1\!\!\pmod 3$,
\item
$e_2<2$,
\item
$e_3<2$, if $\# IV^*=0$,
\item
$e_3<3$, if $\# IV^*\leq2$, 
\item
$\deg j_\calE\equiv_31$ and $\# IV^*=2$, if $e_3=2$ and $\# IV^*<3$.
\end{enumerate}
\end{lem}

\begin{proof}[Proof of 1)]
Let $s$ be a local coordinate at a point corresponding to an odd part. Then the $j$-invariant
takes value $1$ and has odd multiplicity at $s=0$. Hence the local expansion is $j=1 + s^{2k+1}$.
But then the corresponding fibre of $\calE$ would be, according to Table \ref{table:kodaira}, of type $III$, or $III^*$.
This we have already excluded in Proposition \ref{badfibres}.

\emph{of 2)}
If $s$ is a local coordinate at a point corresponding to a part of size $\ell$,
then the $j$-invariant has value $0$ and multiplicity equal to $\ell$.
From the local expansion $j=s^\ell$ we can determine the corresponding fibre type again
from Table \ref{table:kodaira}.
It is $IV,II^*$ in case $\ell\equiv_32$, but these are excluded by Proposition \ref{badfibres}.
It is $II$ or $IV^*$ in case $\ell\equiv_31$, but the former is excluded again by the same reason and hence the
number of $IV^*$ fibres is equal to the number of parts of size $\ell\equiv_31$.

\emph{of 3)}
Suppose $e_2\geq2$, then there are at least two partitions with only even size parts.
By Proposition \ref{jfactor} below, $j_\calE$ then has a proper factor $j'$ which violates  condition (\ref{belyj2}) on the
factorisation of the $j$-invariant.

\emph{of 4)}
Suppose $e_3\geq2$ and $\# IV^*=0$ Then there are at least two partitions with 
parts of size divisible by $3$.
Again by Proposition \ref{jfactor} below, $j_\calE$ then has a proper factor $j'$ violating  condition (\ref{belyj2}).

\emph{of 5)}
The claim follows since the number of parts of size $\ell\equiv_31$ is bounded below by $e_3$ if
$\deg j_\calE$ is not divisible by $3$, otherwise the number of parts of size $\ell\equiv_31$ is at least
$3$ or \emph{4)} applies.

\emph{of 6)}
By \emph{4)} it is not possible to have $\# IV^*=0$.
In case $\deg j_\calE \equiv_30$ or $2$, the number of parts of size $\ell\equiv_31$ is 
therefore at least $3$, respectively \ $4$.
So $\deg j_\calE \equiv_3 1$ and the number of parts of size $\ell\equiv_31$ is $2$ and hence $\# IV^*=2$.
\end{proof}
 
In order to prove the factorisation result used above  we now study branched coverings of the Riemann sphere more systematically from a topological point of view.
Consider a finite branched covering  $\PP^1 \to \PP^1$ of degree $d$ and let $r$ 
denote the number of branch points. 
The fundamental group of the complement with respect to a base point is generated
by elements associated to a \emph{geometric} sequence of 
$r$ simple loops around these points,
i.e.\ chosen carefully such that they meet only at the base point
and their product loop is homotopic to the boundary of a disc
containing all base points.
Of course the product loop represents the trivial element and is known
to suffice as the only relation.

These elements act by permutations on the set $I=\{1,\dots,d\}$ in bijection to the elements of the fibre over the base point, giving rise to
the monodromy elements $\sigma_1,\dots,\sigma_r\in S(I)$ where $\sigma_i$ is given by monodromy along the simple loop around the $i$-th branch point.
Of course the choices of paths form an orbit under the action of the Hurwitz braid group
\cite{Hur}. The induced \emph{Hurwitz action} on $r$-tuples of monodromy elements is generated by 
transformations on adjacent pairs
\[
(\sigma_1,\dots,\sigma_i,\sigma_{i+1},\dots,\sigma_r) \mapsto 
(\sigma_1,\dots,\sigma_{i+1}, \sigma_{i+1}^{-1}\sigma_i{\sigma_{i+1}},\dots,\sigma_r).
\]

\begin{lem}\label{lem:factortwo}
\label{block}
Suppose that the covering $g: \PP^1 \to \PP^1$ 
has degree $kh$ with $k>1$,
and $I$ has a partition into parts $I_1,\dots,I_k$ each of cardinality $h$, such that
\begin{enumerate}
\item
all $\sigma_i, i>2$ preserve all parts and
\item
$\sigma_1,\sigma_2$ permute the parts. 
\end{enumerate}
Then
\begin{enumerate}
\item
the covering map $g$ is the composition of two factors $g=g_2\circ g_1$ where
\item
the second factor $g_2$ is a cyclic branched cover of degree $k$ branched at the two
branch points corresponding to $\sigma_1,\sigma_2$.
\end{enumerate}
Note that the claim and conclusion are insensitive to Hurwitz transformations, except that 
the indices may be relabelled.
\end{lem}

\begin{proof}
Let us consider the topological cover over the complement of the branch points and
an element in $I_1$. Its
stabiliser in the fundamental group determines the covering.
The setwise stabiliser of $I_1$ 
determines an intermediate cover $g_2$
which is of degree $k>1$ over the base.
By assumption $\sigma_i, i>2$ stabilises $I_1$, so $g_2$ is cyclically branched of
degree $k$ over the two points corresponding to $\sigma_1,\sigma_2$.
\end{proof}

We now prove the factorisation of $j_\calE$ into two factors, which we used in the 
proof of Lemma \ref{rami} in
a more abstract setting.

\begin{pro}\label{pro:factortwo}
\label{jfactor}
Suppose $P_1,\dots,P_r$ are the partitions associated to the branch points of a branched covering 
$g: \PP^1\to \PP^1$ of degree $hk$ with $k>1$.
If all parts of $P_1$ and $P_2$ have length divisible by $k$ then 
\begin{enumerate}
\item
the covering map $g$ is the composition of two factors $g=g_2\circ g_1$ and
\item
the second factor $g_2$ is a cyclic branched cover of degree $k$ branched at the two
branch points corresponding to $P_1,P_2$.
\end{enumerate}
\end{pro}

\begin{proof}
We choose a base point and a geometric sequence of $r$ simple loops around the $r$ points making up the branch locus. 
Let $\sigma_1,\dots,\sigma_r$ be the permutations associated to  these paths. 
We recall that the conjugacy class of $\sigma_i$ is determined by $P_i$ and does not depend on the chosen paths. The product of the $\sigma_i$ is the identity.
It suffices to show that there is a decomposition of the fibre $I$ over the base point into $k$ subsets of cardinality $h$, the \emph{blocks}, such that the hypothesis of Lemma \ref{block} is met.

We will rely on the following elementary observation.
If $\sigma$ is a permutation in $S_n$ and $\tau$ the transposition of two elements $a,b$, then
\begin{description}
\item[either (1)]
$a,b$ belong to the orbit of the same cycle $c$ of $\sigma$ of length $\ell$, and
\begin{enumerate}
\item[(a)]
there exists a minimal $\ell_1>0$ such that $\sigma^{\ell_1}(a)=b$,
\item[(b)]
$\sigma$ and $\sigma\tau$ have all cycles of $\sigma$ except $c$ in common,
\item[(c)]
the orbit of $c$ is the union of orbits of two cycles of $\sigma\tau$ which are of lengths 
$\ell_1$ and $\ell_2=\ell-\ell_1$.
\end{enumerate}
\item[or (2)]
$a,b$ belong to orbits of distinct cycles $c_1,c_2$ of $\sigma$ of lengths $\ell_1$, $\ell_2$, and
\begin{enumerate}
\item[(a)]
$\sigma$ and $\sigma\tau$ have all cycles of $\sigma$ except $c_1,c_2$ in common,
\item[(b)]
the union of the orbits  of $c_1,c_2$ is the orbit of 
one cycle of $\sigma\tau$ which has length
$\ell=\ell_1+\ell_2$.
\end{enumerate}
\end{description}

Without loss of generality we may assume $\sigma_1,\sigma_2$ to have $h$ cycles of length
$k$ each, and all other $\sigma_i,i>2$ to be transpositions. This case we call the {\em generic case}. 

In fact, in any other case there are fewer monodromy elements. We can factor each of the first two elements
into a permutation with cycles of length $k$ only and a minimal number of transpositions.
Any other element can be factored into a minimal number of transpositions.
Then the sequence of factors -- after a suitable reordering using Hurwitz transformations -- 
is just a sequence of permutations as in the generic case.
Thus is suffices to prove the claim in the generic case, since the monodromy elements meet the
hypotheses of Lemma \ref{block}, if the factors do.

Recall that the cycles of $\sigma_i$ correspond bijectively to the points in the $i$-th fibre. The Riemann Hurwitz
formula, which relates the Euler number of the domain to the Euler number of the target, the degree and the
fibre defects, reads
\[
2 = 2 hk - 2(hk-h) - (r-2) \quad \implies \quad r = 2h.
\]
Let us write our monodromy elements $\sigma_1,\sigma_2,\tau_3,\dots,\tau_{2h}$ where the $\tau_i$ are
transpositions.

Next we exploit transitivity of the group generated, which needs only the elements $\sigma_2,\tau_i,i>2$,
since the product equals identity.
The element $\sigma_2$ has $h$ orbits, so it generates a transitive group only in case $h=1$.
For $h>1$ there must be a sequence $\tau_{i_1},\dots, \tau_{i_{h-1}}$ with $i_1<\dots<i_{h-1}$ such that
$\rho:=\sigma_2 \tau_{i_1}\dots\tau_{i_{h-1}}$ is an $hk$-cycle.
Using Hurwitz transformations we may assume without loss of generality that $i_1=3,\dots, i_{h-1}=h+1$.

The element $\rho^k$ has order $h$, hence $I$ decomposes into $k$ orbits of length
$h$.
It remains to show, that this is the decomposition into blocks we need for Lemma \ref{block}.

Assume to the contrary that $\tau_i$ for some $2<i\leq h+1$ does not preserve the blocks.
Using Hurwitz transformations we may write
\begin{equation}
\label{permproduct}
\sigma_2=\rho \tau_{h+1} \dots \tau_3 = \rho \tau_i \tau'_h \dots \tau'_3.
\end{equation}
The difference in the number of cycles for $\sigma_2$ and $\rho$ is $h-1$, so in each of the $h-1$
compositions with a transposition on the right we are in case $(1)$ where the number of cycles goes up.

Then $\tau_i$ transposes two elements of the cycle of $\rho$. By observation $(1.a)$ the permutation
$\rho\tau_i$ has a cycle of length $\ell_1$ which $k$ does not divide, since $\tau_i$ is assumed \emph{not}
to preserve the blocks. We remain in the case $(1)$ for all the following $\tau'$, so also $\sigma_2$ has a
cycle of length not divisible by $k$ contrary to the hypothesis of the proposition.

Since blocks are permuted by the element $\rho$ and preserved by the $\tau_i$, $2<i\leq h+1$
the element $\sigma_2$ permutes the blocks. 

Repeating the discussion for the remaining elements, we assume to the contrary that some $\tau_i$, with $i>h+1$ does not preserve
the blocks. We may then write
\begin{equation}
\label{permproduct2}
\sigma_1^{-1} = \sigma_2 \tau_3,\cdots\tau_{2h}
= \rho \tau_{h+2}\dots\tau_{2h} = \rho \tau_i \tau'_{h+3}\dots \tau'_{2h}. 
\end{equation}
The argument applied above to \eqref{permproduct} can now be applied here
to conclude, that $\sigma_1^{-1}$ and thus $\sigma_1$ has a
cycle of length not divisible by $k$ contrary to the hypothesis of the proposition.
We conclude, too, that $\sigma_1$ permutes the blocks.

Therefore we are in a position to apply Lemma  \ref{lem:factortwo} and this concludes the proof.
\end{proof}

%%%%%%%%%%%%%%%%%%%%%%%%%%%%%%%%%%%%%%%%%
\section{Restrictions on possible modular monodromy groups}
\label{sec:Euler}

The results collected so far are sufficient to derive a first characterization of the modular monodromy groups $\barGamma$
which can occur for the elliptic surfaces $\calE$ we consider.

Since the singular fibres of a generic element are all of the form $I_k$, $I_k^*$ and $IV^*$ we can, using Table \ref{table:kodaira},  
write the Euler number formula in the following form:
\begin{equation}
\label{Eulerformel}
24  \quad = \quad \sum_{I_k,I^*_k}  k + \sum_{I^*_k} 6 + 8 \# IV^*= \quad
\deg \jGamma \cdot \deg j_\calE + 6 \# I^* + 8 \# IV^*.
\end{equation}

We will consider all numerically possible combinations of the four integers
$\deg j_\calE>0$, $\# I^*\geq 0$, $\# IV^*\geq0$ and $\deg \jGamma>0$
and we discard the trivial case $\deg \jGamma=1$ corresponding to
$\barGamma=\pslz$ and hence to $\Gamma=\slz$.
We set up the corresponding table of combinations in two parts, namely {\em low } ($\le6$) and {\em high }
($>6$) index $[\pslz:\barGamma]$,
and we add two rows giving $e_2$ and $e_3$.
Since $\# IV^*\leq2$ in every column, they are determined by
\begin{enumerate}
\item
$e_2<2$ and $e_2 \equiv_2 \deg \jGamma$, according to \eqref{jg_cong} and Lemma \ref{rami}.(3).
\item
$e_3<3$ and $e_3 \equiv_3 \deg \jGamma$, according to \eqref{jg_cong} and Lemma \ref{rami}.(5).
\end{enumerate}
In a last row we mark columns, which we \emph{discard} from further consideration
according to one of the following arguments:
\begin{enumerate}
\item[(3)]
$e_3=2$ implies $\#IV^*=2$, according to Lemma \ref{rami}.(6).
\item[(4)]
If $\deg j_\calE=1$ then the $j$-invariant of $\calE$ is rigid. Thus $\calE$ is rigid,
except when there are $I_0^*$ fibres, which is obviously excluded  for $\#I^*=0$. We recall here that we are only concerned with positive dimensional strata. 
But this is also excluded for $\#IV^*>0$ since the presence of an $I_0^*$-fibre implies that $-\id$ is in the monodromy
group, which in turn forbids the existence of a $IV^*$ fibre in $\calE$ by Proposition \ref{badfibres}.
\end{enumerate}

\begin{table}[h]
\[
\begin{array}{c|c|c|c|c|c|c|c|c|c|c|c|c|c|c|c|c|c|c|c|c|c|c|c|c|c|c|c|c|c|c|c|c|c|c|c}
\deg \jGamma & 2 & 2 & 2 & 2 & 2 & 2 & 2 & 2 & 2 & 3 & 3 & 3 & 3 & 4 & 4 & 4 & 4 & 4 & 5 & 6 & 6 & 6 & 6
\\ \deg j_\calE &1 & 2 & 3 & 4 & 5 & 6 & 8 & 9 & 12 & 2 & 4 & 6 & 8 & 1 & 2 & 3 & 4 & 6 & 2 & 1 & 2 & 3 & 4 
\\ \# I^* & 1 & 2 & 3 & 0 & 1 & 2 & 0 & 1 & 0 & 3 & 2 & 1 & 0 & 2 & 0 & 2 & 0 & 0 & 1 & 3 & 2 & 1 & 0 
\\ \# IV^* & 2 & 1 & 0 & 2 & 1 & 0 & 1 & 0 & 0 & 0 & 0 & 0 & 0 & 1 & 2 & 0 & 1 & 0 & 1 & 0 & 0 & 0 & 0
\\\hline e_2 & 0 & 0 & 0 & 0 & 0 & 0 & 0 & 0 & 0 & 1 & 1 & 1 & 1 & 0 & 0 & 0 & 0 & 0 & 2 & 0 & 0 & 0 & 0 
\\ e_3 & 2 & 2 & 2 & 2 & 2 & 2 & 2 & 2 & 2 & 0 & 0 & 0 & 0 & 1 & 1 & 1 & 1 & 1 & 2 & 0 & 0 & 0 & 0 
\\\hline \emph{discard} & {\color{red}4} & {\color{red}3} &  {\color{red}3} && {\color{red}3}&  {\color{red}3} & {\color{red}3}& 
{\color{red}3} & {\color{red}3} & &&&&{\color{red}4} & &&&& {\color{red}3}& &&&
\end{array}
\]
\caption{Combinations of numerical invariants (low index)}
\label{table:columnslow}
\end{table}

\begin{table}[h]
\[
\begin{array}{c|c|c|c|c|c|c|c|c|c|c|c|c|c|c|c|c|c|c|c|c|c|c|c|c|c|c|c|c|c|c|c|c}
\deg \jGamma & 8 & 8 & 8 & 9 & 10 & 12 & 12 & 16 & 18 & 24 
\\ \deg j_\calE & 1 & 2 & 3 & 2 & 1 & 1 & 2 & 1 & 1 & 1  
\\ \# I^* & 0 & 0 & 0 & 1 & 1 & 2 & 0 & 0 & 1 & 0  
\\ \# IV^* & 2 & 1 & 0 & 0 & 1 & 0 & 0 & 1 & 0 & 0 
\\\hline e_2 & 0 & 0 & 0 & 1 & 0 & 0 & 0 & 0 & 0 & 0  
\\ e_3 & 2 & 2 & 2 & 0 & 1 & 0 & 0 & 1 & 0 & 0  
\\\hline \emph{discard} & {\color{red}4}& {\color{red}3} & {\color{red}3} && 
{\color{red}4} & && {\color{red}4} && {\color{red}4} 
\end{array}
\]
\caption{Combinations of numerical invariants (high index)}
\label{table:columnshigh}
\end{table}

Note that $e_2=e_3=0$ is equivalent to $\barGamma$ being torsion free, 
and that $\deg \jGamma$ is equal to the index of $\barGamma$ in $\pslz$.
Furthermore, subgroups of $\pslz$ of index at most $6$ are congruence
subgroups \cite[Thm.5]{wohlfahrt}.
Accordingly, our groups of low index occur in \cite[Table 2]{CP} 
with genus
$0$, the genus of the domain of $\jGamma$, and with $(I=\deg \jGamma,e_2,e_3)$ 
as in one of the columns in our Table \ref{table:columnslow}.
We find the entries $2A^0, 2B^0, 3B^0, 2C^0$ and $4B^0$.
By 
\cite[Table 4]{CP} they usually go by the standard names
$\barGamma(1)^2, \barGamma_1(2), \barGamma_1(3), \barGamma(2)$
and $\barGamma_1(4)$ -- which we recall after the proposition --
and we get:

\begin{pro}\label{pro:indexmodular}
\label{modular_monodromy}
Let $\barGamma$ be a subgroup of $\pslz$ which appears in one of the columns of Table \ref{table:columnslow} or Table \ref{table:columnshigh} which have not been discarded. 
Then it belongs to one of the following, mutually exclusive cases:
\begin{enumerate}
\item
$\barGamma$ has index $18$ in $\pslz$ and is torsion free,
\item
$\barGamma$ has index $12$ in $\pslz$ and is torsion free,
\item
$\barGamma$ has index $9$ in $\pslz$ with $e_2=1,e_3=0$,
\item
$\barGamma$ is either $\barGamma(2)$ or $\barGamma_1(4)$, has index $6$ in $\pslz$ and is torsion free,
\item
$\barGamma_1(3)$ which has index $4$ with $e_2=0,e_3=1$,
\item
$\barGamma_1(2)$  which has index $3$ with $e_2=1,e_3=0$,
\item
$\barGamma(1)^2$ which has index $2$ with $e_2=0,e_3=2$.
\end{enumerate} 
\end{pro}

Here we recall the standard notation for certain congruence subgroups of $\slz$ which are of importance for us.
The {\em principal congruence subgroup of level $n$} is denoted by $\Gamma(n)$ and defined as 
\[
\Gamma(n)= \left\{
\begin{pmatrix} a & b \\ c & d \end{pmatrix} \Big| \quad 
a\equiv d \equiv 1, b \equiv c\equiv 0 \mod n \right\}.
\]
Its geometric relevance is the fact that the modular curve $X^0(n)=\HH_1/\Gamma(n)$ parameterises elliptic curves with a level $n$ structure, i.e. a symplectic basis with respect to the Weil form, of the group $E[n]$ of $n$-torsion points. 
Next we recall the group
\[
\Gamma_1(n)= \left\{
\begin{pmatrix} a & b \\ c & d \end{pmatrix} \Big| \quad 
a\equiv d \equiv 1, c \equiv 0 \mod n \right\}
\]
whose meaning is that the modular curve $X_1^0(n)=\HH_1/\Gamma_1(n)$ parameterises elliptic curves with a fixed point of order $n$.
We will also use the group 
\[
\Gamma_0(n)= \left\{
\begin{pmatrix} a & b \\ c & d \end{pmatrix} \Big| \quad 
c \equiv 0 \mod n \right\}.
\]
Its significance is that the modular curve $X_0^0(n)=\HH_1/\Gamma_0(n)$ is the moduli space of elliptic curves with a distinguished subgroup $\ZZ/n\ZZ \subset E[n]$ of $n$-torsion points.
By $\barGamma(n)$, $\barGamma_1(n)$ and $\barGamma_0(n)$  we denote the images of these groups in $\pslz$.
Finally we recall that $\barGamma(1)^2$ is the subgroup of $\pslz$ which is generated by all squares. 

\begin{rem}
\label{Gamma0}
We note here that $\barGamma_0(n)=\barGamma_1(n)$ for $n\leq 4$. This will be relevant for Proposition \ref{Gamma04}.
\end{rem}

The data for the groups $\barGamma$ in Proposition \ref{modular_monodromy}
can be further exploited to obtain upper bounds for the dimension of
any corresponding monodromy stratum. 
For this purpose we make the following
\begin{dfn}
Let $\Gamma$ be a subgroup of finite index of $\slz$. We define the {\em maximal dimension of a monodromy stratum associated to $\Gamma$} by:
$$
m(\Gamma)= \begin{cases}
\operatorname{max } \{\dim \calF_{\Gamma,i} \mid \calF_{\Gamma,i} \mbox{ is a monodromy stratum associated to } \Gamma \}  
\\ -\infty \mbox{ if there exists no monodromy stratum with monodromy group } \Gamma.  
\end{cases}
$$
\end{dfn}

Together with the explicit bounds in the upcoming Lemma \ref{maxdim}
this will be used later in the following way: 
\begin{align}
&\text{The closure of a Shimada stratum of dimension $d$ is ambi-typical}
\notag\\[-3mm]
&\label{dim2} \\[-3mm]
&\text{if $d\geq m(\Gamma)$ for its generic monodromy $\Gamma$.} \notag
\end{align}

\begin{lem}
\label{maxdim}
Let $\barGamma$ be a subgroup of $\pslz$ which appears in one of the columns of Table \ref{table:columnslow} or Table \ref{table:columnshigh} which have not been discarded and let $\Gamma$ be any lift of $\barGamma$ 
in $\slz$. The invariants are listed in Table \ref{table:maxdim} .

\begin{table}[htbp]\label{tab:invariants1} 
\[
\begin{array}{c|ccccccc}
\deg \jGamma & 2  & 3 & 4 & 9 & 6 & 12 & 18 \\
e_2 & 0 & 1 & 0 & 1 & 0 & 0 & 0 \\
e_3 & 2 & 0 & 1 & 0 & 0 & 0 & 0 \\
 \# \operatorname{poles } \operatorname{ of } \jGamma & 1 & 2 & 2 & 3 & 3 & 4 & 5 \\
m(\Gamma) & \leq 6 & \leq 10 & \leq 6 & \leq 2 & \leq 6 & \leq 2 & \leq 1 \\
\end{array}
\]
\caption{Maximal dimension of $\Gamma$-strata}
\label{table:maxdim}
\end{table}
\end{lem}

\begin{proof}
The first row in the table simply lists the possible degrees for $\jGamma$ which occur in Table \ref{table:columnslow} or Table \ref{table:columnshigh}. The second and third row are also copied from these 
tables. The number of poles of $\jGamma$ can be computed via Riemann-Hurwitz and equals $2+ \deg \jGamma/6 -2 e_3/3 -e_2/2$. 
Thus it remains to prove the upper bound for the dimension of the monodromy strata. 

The dimension of a stratum is equal to the open dense irreducible subset of elliptic surfaces
sharing the generic configuration of singular fibres. In Lemma \ref{dimLoc} the dimension
augmented by $6$ is given as the sum of cardinality of singular fibres and the cardinality of $*$-fibres.
Thus
\begin{equation}\label{equ:dimbound}
6 + \dim \quad \leq \quad s_\infty + s_0 + s_1 +  s^* +  s^*
\end{equation}
where $s_\infty$, $s_0,s_1$ are the number of singular fibres with $j=\infty, 0,1$  resp.  and 
$s^*$ is the number of $*$-fibres 
which is also an upper bound for the number of singular fibres with $j\not\in\{0,1,\infty\}$.

On the other hand the Euler sum gives  the bound
\begin{equation}\label{equ:ebound}
24  \geq \deg j_\calE\deg \jGamma + 4s_0 + 3 s_1 +6 s^*
\end{equation}
as follows immediately from Table \ref{table:kodaira}. 

Indeed, there are further restrictions on $s_\infty,s_0,s_1$ due to the map $\jGamma$:
\begin{enumerate}
\item
$s_\infty$ is the number of poles of the $j$-invariant and thus bounded above by $ \deg j_\calE$
times the number of poles of $\jGamma$. 
%($2+ \deg j_\Gamma/6 -2 e_3/3 -e_2/2$ according to Riemann Hurwitz).
\item
$s_0$ is the number of points with local $j$-expansion $s^k$, $k$ not a multiple of $3$,
so $s_0=0$ if $e_3=0$.
\item
$s_1$ is the number of points with local $j$-expansion $1+s^k$, $k$ odd, so $s_1=0$ if $e_2=0$.
\end{enumerate}

We shall now give the proof exemplary for the case of index $9$. The other cases can be argued similarly. From (\ref{equ:ebound}) and using $s_0 \geq 0$ we find
\[
24  \geq  \deg \jGamma \deg j_\calE + 3 s_1 + 6 s^*. 
\]
Since $\jGamma$ has degree $9$ and $3$ poles this implies
\[
24 \geq   9 \frac{s_\infty}3 + 3 s_1 + 6 s^*  
\] 
or equivalently 
\[
8 \geq   s_{\infty} +  s_1 + 2 s^*.  
\]  
Since $e_3=0$ in this case, and hence $s_0=0$, it then follows immediately from (\ref{equ:dimbound}) that $m(\Gamma) \leq 2$.
\end{proof}

In the case of {\em low index} we have ample information on the two factors in the factorisation $j(\calE)= \jGamma \circ j_{\calE} $. Firstly, since we know the groups $\barGamma$ explicitly
we can in fact also write down the maps $\jGamma$ explicitly, as listed in Table \ref{table:j-functions}. 

\begin{table}[htb]
\[
\begin{array}{rclrl}
\barGamma(2) &:& \jGamma=\frac{(z^2+3)^3}{z^2(z^2-9)^2} &=&1 +  \frac{27 (z^2 -1)^2}{z^2(z^2-9)^2}\quad  \\
\barGamma_1(4) &:& \jGamma=\frac{4(z^2-4z+1)^3}{27z(z-4)} &=&1 +  \frac{(z-2)^2 (2z^2-8z-1)^2}{27z(z-4)} \\
\barGamma_1(3) &:& \jGamma=\frac{z(z+8)^3}{64(z-1)^3} &=&1 +  \frac{(z^2 - 20z-8)^2}{64(z-1)^3} \\
\barGamma_1(2) &:& \jGamma=\frac{(z+3)^3}{27(z-1)^2} &=&1 +  \frac{z(z - 9)^2}{27(z-1)^2} \\
\barGamma(1)^2 &:& \jGamma=\frac{4z}{(z+1)^2} &=&1 -  \frac{(z-1)^2}{(z+1)^2} 
\end{array}
\]
\caption{$j$-functions}\vspace*{-.8cm}
\label{table:j-functions}
\end{table}
There are various different explicit formulae in the literature, e.g.\ \cite{FK,MS}
since coordinates on the domain can be chosen arbitrarily.
Indeed, to check our formulae it suffices to check their degrees and their branching 
over $0,1$ and $\infty$ which are determined by the multiplicity sequences
of the two numerators and the denominator respectively.

With our choice of $z$-coordinate we note that  
\begin{eqnarray}
\quad&\text{\, row }1:& \text{
$z=0, \pm 3$ are the poles of $\jGamma$, of pole order $2$ each
}\label{row1}
\\ &\text{\, row }2:& \text{
$z=0,4,\infty$ are the poles of $\jGamma$, of pole order $1,1,4$ respectively
}\label{row2}
\\ &\text{\, row }3:& \text{
$e_2=0$ and
$z=0$ is the only $3$-torsion point 
(non-critical point of value $0$)
}\label{row3}
\\ &\text{\, row }4:& \text{
$e_3=0$ and
$z=0$ is the only $2$-torsion point 
(non-critical point of value $1$)
}\label{row4}
\\ &\text{\, row }5:& \text{
$e_2=0$ and
$z=0,\infty$ are the $3$-torsion points mapping to $0$.
}\label{row5}
\end{eqnarray}

Secondly, information from Table  \ref{table:columnslow} also allows us to obtain information about the map $j_{\calE}$. In particular, we can list the possible degrees of this map as well as the branching behaviour 
over \emph{special} points, namely the $3$-torsion and $2$-torsion points, see Lemma \ref{rami}.
Using the notation of \cite[Section 1]{BPT} we will denote these by $A$ and $B$ respectively.
So we can derive the number of pre-images of special points and their multiplicities.
They are included into  Table \ref{table:jE-functions} as the tuple of multiplicities with index the point they
map to. The last column of this table gives the most general polynomial expression for $j_\calE$
fitting the branching data. Here $\alpha_i, \beta_i, \gamma_i, \delta_i$ denote coprime
homogeneous bivariate polynomials of degree $i$.

\begin{table}[htb]
\label{table:jE-functions}
\[
\begin{array}{rcccc}
\barGamma & \deg j_\calE & \text{special point(s)} & \text{branch data} & j_\calE \\[1mm]
\barGamma(2),\barGamma_1(4) & 4 & - & - & \alpha_4:\beta_4  \\
& 3 & - & - & \alpha_3:\beta_3  \\
& 2 & - & - & \alpha_2:\beta_2  \\
& 1 & - & - & \alpha_1:\beta_1  \\
\barGamma_1(3) & 6 & A=0 & (3,3)_A & \alpha_2^3:\beta_6 \\
 & 4 & A=0 & (3,1)_A & \alpha_1^3\gamma_1:\beta_4\quad \\
 & 3 & A=0 & (3)_A & \alpha_1^3:\beta_3 \\
 & 2 & A=0 & (1,1)_A & \gamma_2:\beta_2 \\
\barGamma_1(2) & 8 & B=0 & (2,2,2,2)_B & \alpha_4^2:\beta_8 \\
& 6 & B=0 & (2,2,2)_B & \alpha_3^2:\beta_6 \\
& 4 & B=0 & (2,2)_B & \alpha_2^2:\beta_4 \\
& 2 & B=0 & (2)_B & \alpha_1^2:\beta_2 \\
\barGamma(1)^2 & 4 & A_1=0, A_2=\infty & (3,1)_{A_1}, (3,1)_{A_2} & \alpha_1^3\gamma_1:\beta_1^3\delta_1\\[3mm]
\end{array}
\]
\caption{$j_\calE$-functions for low index}
\end{table}

%%%%%%%%%%%%%%%%%%%%%%%%%%%%%%%%%%%%
\section{Families of Weierstra{\ss} data for low index}
\label{sec:Weier}

In this section we will analyse the Weierstra{\ss} data of the Jacobian fibrations with modular monodromy  of low index. 
We will also discuss the Mordell-Weil group in these cases.

We start with the following table of Weierstra{\ss} data, whose relevance is that this includes all  the Weierstra{\ss} data needed to describe the families from 
Section \ref{sec:Euler} of low monodromy index.

Note that in this table we still assume the polynomials to have the degree given by their
index, but \emph{no} assumption of coprimality is imposed.

\begin{table}[htbp]
\[
\begin{array}{l|c|l|l|l}
\# 
& \barGamma & g_2 & g_3 & \Delta=g_2^3-27g_3^2\\[1mm]
%& & &  & (\text{ up to constant }) \\
\hline
i) \parbox[b][5mm][b]{0mm}{} & \barGamma(2) 
& \alpha_4^2+3\beta_4^2 & \beta_4(\alpha_4^2-\beta_4^2)
& \alpha_4^2(\alpha_4^2-9\beta_4^2)^2
\\
ii) \parbox[b][5mm][b]{0mm}{} & \barGamma_1(4) 
& 12(\alpha_4^2-4\alpha_4\beta_4+\beta_4^2) 
& 4(\alpha_4-2\beta_4)(2\alpha_4^2-8\alpha_4\beta_4-\beta_4^2)
& 2^43^6\alpha_4(\alpha_4-4\beta_4)\beta_4^4
\\
iii) \parbox[b][5mm][b]{0mm}{} & \barGamma_1(3) 
& 3\alpha_2(\alpha_2^3+8\beta_6) & \alpha_2^6-20\alpha_2^3\beta_6 - 8\beta_6^2
&  2^63^3(\alpha_2^3-\beta_6)^3\beta_6 
\\
iv) \parbox[b][5mm][b]{0mm}{} & \barGamma_1(3) 
& 3\alpha_1\gamma_2^2(\alpha_1^3 + 8\beta_3)  
& \gamma_2^3(\alpha_1^6-20\alpha_1^3\beta_3 - 8\beta_3^2)
&  2^63^3(\alpha_1^3-\beta_3)^3\beta_3\gamma_2^6
\\
v) \parbox[b][5mm][b]{0mm}{} & \barGamma_1(2) 
& 3\alpha_4^2+ 9 \beta_8 & \alpha_4(\alpha_4^2-9\beta_8)
& 3^6(\alpha_4^2-\beta_8)^2\beta_8
\\ 
vi) \parbox[b][5mm][b]{0mm}{} & \barGamma(1)^2 
& -12\alpha_1\beta_1\gamma_1^3\delta_1^3 & 
4\gamma_1^4\delta_1^4(\alpha_1^3\gamma_1-\beta_1^3\delta_1)
&  -2^43^3(\alpha_1^3\gamma_1+\beta_1^3\delta_1)^2
\gamma_1^8\delta_1^8 
\end{array}
\]
\caption{Weierstra{\ss} families}
\label{table:weierstrassfamilies}
\end{table}

\begin{pro}
\label{weier}
Every elliptic surface with data as in Table \ref{table:jE-functions} has 
Weierstra{\ss} datum in one of the families in Table \ref{table:weierstrassfamilies}.
The correspondence between the monodromy group and the Weierstra{\ss} data is given by:
\[
\begin{array}{rclrclrcl}
i) & : & \barGamma(2) &
iii) & : & \barGamma_1(3), \deg j_\calE\equiv_2 0 &
v) & : & \barGamma_1(2) 
\\[2mm]
ii) & : & \barGamma_1(4) \qquad&
iv) & : & \barGamma_1(3), \deg j_\calE = 3 \qquad&
vi) & : & \barGamma(1)^2. 
\end{array}
\]
\end{pro}

\begin{proof}
We have to show that for each row of Table \ref{table:jE-functions}
any elliptic surface $\calE$ with the data provided by the row can be 
obtained by some suitable choice of Weierstra{\ss} datum from
Table \ref{table:weierstrassfamilies}.
Here we shall give the proof for $\calE$ with modular monodromy $\barGamma_1(3)$ 
which is the most subtle case. The other groups can be treated in an analogous way.
We begin with the first corresponding row, so $\deg j_\calE=6$.  

Composing the expressions for $\jGamma$ and $j_\calE$ from
Table \ref{table:j-functions} and Table \ref{table:jE-functions} respectively,
we get
\[
j(\calE) \quad = \quad \jGamma\circ j_\calE \quad = \quad
\frac{\alpha_2^3(\alpha_2^3+8\beta_6)^3}{64(\alpha_2^3-\beta_6)^3\beta_6}  
\quad = \quad
1 +  \frac{(\alpha_2^6-20\alpha_2^3\beta_6 - 8\beta_6^2)^2}{64(\alpha_2^3-\beta_6)^3\beta_6}.
\]
We look at the general expression of the $j$-function in terms
of Weierstra{\ss} data
\[
j \quad =\quad \frac{g_2^3}{g_2^3-27g_3^2} \quad = \quad
\frac{(g_2/3)^3}{(g_2/3)^3-g_3^2} \quad = \quad
1 + \frac{g_3^2}{(g_2/3)^3-g_3^2}.
\]
If we plug in the same $\alpha_2$ and $\beta_6$ we get the identical expression
for the $j$-invariant as above. 
Moreover the analysis of common factors of $g_2,g_3$ and their multiplicities
according to the Tate algorithm, see Table \ref{table:tate}, 
gives no singular fibres except of type $I_\nu$
since $\alpha_2$ and $\beta_6$ are coprime by assumption $\deg j_\calE=6$. 
Therefore $\calE$ and the 
elliptic surface given by the Weierstra{\ss} datum share the functional and the
homological invariant and hence are isomorphic.
\medskip

Still in case $\barGamma_1(3)$ but with $\deg j_\calE=4$ we obtain from 
Table \ref{table:j-functions} and Table \ref{table:jE-functions} that
\[
j(\calE) \quad = \quad \jGamma\circ j_\calE \quad = \quad
\frac{\alpha_1^3\gamma_1(\alpha_1^3\gamma_1+8\beta_4)^3}{64(\alpha_1^3\gamma_1-\beta_4)^3\beta_4}  
\quad = \quad
1 +  \frac{(\alpha_1^6\gamma_1^2-20\alpha_1^3\gamma_1\beta_4 - 8\beta_4^2)^2}{64(\alpha_1^3\gamma_1-\beta_4)^3\beta_4}
\]
This expression can not be obtained as easily. Indeed, we have to recall 
that we are allowed to plug in polynomials into the families which are \emph{not}
necessarily coprime.
Doing this  with $\beta_6=\gamma_1^2\beta_4$ and $\alpha_2=\alpha_1\gamma_1$ 
in family $iii)$ we get
\[
j \quad = \quad 
\frac{\alpha_1^3\gamma_1^3(\alpha_1^3\gamma_1^3+8\beta_4\gamma_1^2)^3}{64(\alpha_1^3\gamma_1^3-\beta_4\gamma_1^2)^3\beta_4\gamma_1^2}  
\quad = \quad
1 +  \frac{(\alpha_1^6\gamma_1^6-20\alpha_1^3\gamma_1^5\beta_4 - 8\gamma_1^4\beta_4^2)^2}{64(\alpha_1^3\gamma_1^3-\beta_4\gamma_1^2)^3\beta_4\gamma_1^2}.
\]
which is exactly the expression for $j(\calE)$ expanded by $\gamma_1^8$.
The Weierstra{\ss} datum is thus
\[
g_2=3\alpha_1\gamma_1^3(\alpha_1^3\gamma_1+8\beta_4),\quad
g_3=\gamma_1^4(\alpha_1^6\gamma_1^2-20\alpha_1^3\gamma_1\beta_4 - 8\beta_4^2).
\]
The analysis with the Tate algorithm shows the existence of one $IV^*$ fibre and otherwise
only singular fibres of type $I_\nu$ since $\alpha_1\gamma_1$ and $\beta_4$ are coprime
and we may conclude again that $\calE$ is isomorphic to a surface given by 
Weierstra{\ss} datum from family $iii)$.

The case with $\barGamma_1(3)$ and $\deg j_\calE=2$ is very similar but with
$\beta_6=\gamma_2^2\beta_2$ and $\alpha_2=\gamma_2$ sharing even a factor
of degree two. The Weierstra{\ss} datum 
\[
g_2=3\gamma_2^3(\gamma_2+8\beta_2),\quad
g_3=\gamma_2^4(\gamma_2^2-20\gamma_2\beta_2 - 8\beta_2^2)
\]
defines then an elliptic surface sharing the functional and homological invariant with
$\calE$ again.
\medskip

This leaves us with $\barGamma_1(3)$ and $\deg j_\calE=3$. From 
Table \ref{table:j-functions} and Table \ref{table:jE-functions} we get
\[
j(\calE) \quad = \quad \jGamma\circ j_\calE \quad = \quad
\frac{\alpha_1^3(\alpha_1^3+8\beta_3)^3}{64(\alpha_1^3-\beta_3)^3\beta_3}  
\quad = \quad
1 +  \frac{(\alpha_1^6-20\alpha_1^3\beta_3 - 8\beta_3^2)^2}{64(\alpha_1^3-\beta_3)^3\beta_3}.
\]
This time we choose to plug the coprime $\alpha_1$ and $\beta_3$ from an
expression for $j(\calE)$ together with a $\gamma_2$ still to be determined
into the family $iv)$ and obtain the $j$-function of this Weierstra{\ss} datum to be
\[
j \quad = \quad
\frac{\alpha_1^3\gamma_2^6(\alpha_1^3+8\beta_3)^3}{64(\alpha_1^3-\beta_3)^3\gamma_2^6\beta_3}  
\quad = \quad
1 +  \frac{\gamma_2^6(\alpha_1^6-20\alpha_1^3\beta_3 - 8\beta_3^2)^2}{64(\alpha_1^3-\beta_3)^3\gamma_2^6\beta_3}
\]
which is exactly the expression for $j(\calE)$ expanded by $\gamma_2^6$.
Thus we get the $j$-invariant $j(\calE)$ with the Weierstra{\ss} datum
\[
g_2=3\alpha_1\gamma_2^2(\alpha_1^3+8\beta_3),\quad
g_3=\gamma_2^3(\alpha_1^6-20\alpha_1^3\beta_3 - 8\beta_3^2).
\]
Finally we choose $\gamma_2$ to vanish at the two $I_0^*$ fibres of $\calE$.
Then $\gamma_2$ is coprime to $\beta_3(\alpha_1^3-\beta_3)$ since the
$j$-invariant of an $I_0^*$ fibre is finite.

The analysis of this datum with the Tate algorithm shows the existence of 
$I_0^*$ fibres precisely at the zeros of $\gamma_2$. 
Again we may conclude, since functional and homological invariant are shown
to coincide.
\end{proof}

We next determine for the
generic members of each family whether
$-\id\in\Gamma$ or not. The following lemma
will be used as a tool to determine the Mordell-Weil torsion
from the monodromy group.

In  Section \ref{sec:Euler} we already introduced the principal congruence subgroup $\Gamma(n)$. 
As we said there, the modular curve $X^0(n)=\HH_1/\Gamma(n)$  is the classifying space
of elliptic curves with a level-$n$ structure. 
This carries a universal family if $n \geq 3$. If $n=2$, we no longer have a universal family of elliptic curves, but a universal Kummer family still exists. 
We denote by $X(n)$ the compactification of $X^0(n)$ which is obtained by adding the cusps, i.e. $X(n)=\overline\HH_1/\Gamma(n)$. The universal family over $X^0(n)$  can be extended to $X(n)$, the extension is known as Shioda modular surface.
This has $n^2$ sections which restrict to the $n$-torsion points on the smooth fibres of the universal family.
 
We had also introduced the group $\Gamma_1(n)$ and the curve   $X_1^0(n)=  \HH_1/\Gamma_1(n)$. This
is the classifying space of elliptic curves with a fixed $n$-torsion point. As above we can compactify the curve $X^0_1(n)$ by adding the cusps and obtain a curve $X_1(n)$.

If $m$ divides $n$  we consider the group
\[
\Gamma_m(n):= \Gamma(m) \cap \Gamma_1(n).
\]
Obviously, $X^0_m(n):= \HH_1 / \Gamma_m(n)$ parameterises elliptic curves with a level-$m$ structure and additionally an $n$-torsion point.  Again by adding the cusps we obtain 
the compatification  ${{X}}_m(n)$. The universal family over $X^0_m(n)$ extends to $X_m(n)$ and in addition to the sections giving the $m$-torsion points we have a distinguished section of order $n$ 
which restricts to the distinguished $n$-torsion point on the smooth fibres.  

\begin{lem}
\label{coxparry}
Let $\calE$ be a Jacobian fibration with monodromy group $\Gamma$ and assume that  $m$ and $n$ are 
positive integers with $m\mid n$. 
Then the following are equivalent:
\begin{enumerate}
\item
$\Gamma$ is, up to conjugation, contained in $\Gamma_m(n)$.
\item
The Mordell-Weil group $\MW(\calE)$ contains a subgroup 
$\ZZ/n\ZZ \times\ZZ/m\ZZ$.
\end{enumerate}
\end{lem}
\begin{proof}
We shall first prove the implication form (1) to (2). For this let $U$ be the subset of the base $\PP^1$ of $\calE$ which is given by removing the points $j(\calE)^{-1}\{0,1,\infty\}$ and the 
base points of the singular fibres. We want to construct sections which form a subgroup $\ZZ/n\ZZ \times\ZZ/m\ZZ$ of $\MW(\calE)$. It is enough to do this over $U$ as the sections can then be
extended to $\PP^1$ (since the base has dimension $1$).  We choose a base point $x_0 \in U$. For each point $x\in U$ we can choose a small disc $U(x)$ such that 
$\calE|_{U(x)} \cong \CC \times U(x)/(\ZZ + \ZZ\tau_x)$ where $\tau_x: U(x) \to \HH_1$ is a local lift of the $j$-function  $j(\calE)$ and $\ZZ + \ZZ\tau_x$ acts fibrewise on 
$\CC \times U(x)$ by translation on $\CC \times \{t\}$ with the lattice $\ZZ + \ZZ\tau_x(t)$.  

In particular, the fibre $\calE_{x_0}\cong \CC/(\ZZ + \tau_{x_0}(x_0))$ where we have chosen a fixed lift $\tau_{x_0}$. Let $a,b\in \ZZ$ and  $z_0=(a+b\tau_{x_0}(x_0))/\ell$
be an $\ell$-torsion point in $\calE_{x_0}$ (where $\ell$ will become either $m$ or $n$). Using the above local uniformisation of $\calE$ we can extend this to a local $\ell$-torsion section of $\calE|_{U(x_0)}$. We want to
extend such a section to $U$. Given any point $y\in U$ we can choose a path $s: [0,1] \to U$ connecting $x_0$ with $y$. We can cover this path with finitely many open sets $U(x_i), i=0, \ldots, N$ with $x_N=y$ 
and choose local lifts $\tau_{{x_i}}$ with $\tau_{x_i}|{U_i \cap U_{i+1}} = \tau_{x_{i+1}}|{U_i \cap U_{i+1}}$. Then we can move the point $z_0$ along the path $s$ to an $\ell$-torsion point $z_0(y) \in \calE_{y}$.    
Clearly, this will a priori depend on the chosen path $s$. The point  $z_0(y)$ will be independent of this choice if all elements in the monodromy group  $\Gamma(\calE)$ fix the point 
$(a+b)/\ell\in \ZZ/\ell + \ZZ/\ell$ (where we have chose a fixed representation $\Gamma(\calE) \to \SL(2,\ZZ)$ and hence consider $\Gamma(\calE)$ as a subgroup of $\SL(2,\ZZ)$). Given this observation, the 
claim now follows immediately from the definition of the group $\Gamma_m(n)$. The converse implication (2) to (1) follows by the same argument. 
\end{proof}

As a consequence of the above discussion we can now obtain some first results on the torsion
Mordell-Weil group.

\begin{lem}
\label{lem:G3MW}
Let $\calE$ be a member of the family $iii)$. Then the following holds:
\begin{itemize}
\item[\rm{(1)}] $-\id \notin \Gamma(\calE)$,
\item[\rm{(2)}] $\ZZ/3\ZZ \subset \MW(\calE)$. 
\end{itemize}
Moreover, for a generic element  $\calE$ the equality $\MW(\calE)= \ZZ/3\ZZ$ holds.
\end{lem}

\begin{proof}
Consider the following Weierstra{\ss} datum of a rational elliptic surface $\bar\calE$ in homogeneous
coordinates $z_1,z_0$
\[
\bar g_2=3z_1^3(z_1+8z_0),\quad
\bar g_3=z_1^4(z_1^2-20z_1z_0 - 8z_0^2).
\]
A generic member $\calE$ of family $iii)$ is given by coprime $\beta_6$ and $\alpha_2$.
To compute the datum of the pull-back of $\bar\calE$ along $j_\calE$ we plug in $\alpha_2^3$ and 
$\beta_6$ into $z_1$ and $z_0$ resp.\ resulting in 
\[
g'_2=3\alpha_2^9(\alpha_2^3+8\beta_6),\quad
g'_3=\alpha_2^{12}(\alpha_2^6-20\alpha_2^3\beta_6 - 8\beta_6^2).
\]
The proper Weierstra{\ss} datum of the normalised and possibly blown down pull-back
is then 
\[
g_2=3\alpha_2(\alpha_2^3+8\beta_6),\quad
g_3=\alpha_2^6-20\alpha_2^3\beta_6 - 8\beta_6^2
\]
which is exactly that of $\calE$. By the Tate algorithm we find that the fibre configuration  on the rational elliptic surface $\bar\calE$ is
$IV^*, I_1,I_3$.

Due to the singular fibres the trivial lattice is $E_6+A_2$, which implies that
the surface is no.69 in the list of Oguiso-Shioda \cite{os} and thus has
Mordell-Weil torsion $\ZZ/3$. In particular $-\id$ does not belong to the monodromy.

Since torsion sections are pulled back to torsion sections, the torsion
Mordell-Weil group of $\calE$ also contains $\ZZ/3$.
Moreover, the monodromy group of $\calE$ is generated by the monodromy
elements of the rational elliptic surface $\bar\calE$  associated to loops liftable along $j_\calE$.
So the monodromy of $\calE$ is contained in that of the rational elliptic surfacer $\bar\calE$ 
and hence does not contain $-\id$.

Conversely, the generic $\calE$ does not have torsion Mordell-Weil properly containing
$\ZZ/3$. Otherwise $j(\calE)$ factors through $\jGamma$ for some $\barGamma$
properly contained in $\barGamma_1(3)$. This contradicts 
$\jGamma\circ j_\calE$ being the canonical factorisation.
\medskip

In the above arguments we have used that $\calE$ is a generic element of the family. It thus remains to prove that 
$-\id$ is never contained in the monodromy group
and $\ZZ/3$ is always contained in the torsion Mordell-Weil. But this both
follows by semicontinuity, since 
the monodromy group can only get smaller under specialization and the torsion Mordell-Weil can only get bigger. 
\end{proof}

\begin{lem}
\label{mwtorsion6}
Let $\calE$ be a member of family $ii)$. Then the following holds:
\begin{itemize}
\item[\rm{(1)}] If $\alpha_4$ or $\beta_4$ is a perfect square, then $-\id \notin \Gamma(\calE)$.
\item[\rm{(2)}] In the two cases the generic element has $\MW(\calE)=\ZZ / 2\ZZ$ and $\MW(\calE)=\ZZ / 4\ZZ$ respectively.
\end{itemize}
\end{lem}

\begin{proof}
In both cases the elliptic fibration is a pull-back along $j_\calE$ 
of degree four of a rational elliptic fibration whose fibre configuration is one of the following
\[
I_4^* + 2\ I_1,\quad 
I_1^* + I_4 + I_1    
\]
%  D_8    %  D_5 + A_3
and with two simple ramification points over the $*$-fibre.
Again, it suffices to investigate the rational
elliptic fibrations.
The monodromy is freely generated by two elements, so $-\id$ does not 
belong to the monodromy, since it is torsion.

The surfaces occur as nos. 64,72 in the list of Oguiso-Shioda \cite{os} and have
Mordell-Weil torsion $\ZZ/2$
and $\ZZ/4$ respectively. So via the pull-back these groups are contained
in the Mordell-Weil torsion.
Equality holds for all $j_\calE$
such that $\jGamma\circ j_\calE$ is the canonical factorisation, hence generically.
\end{proof}

\begin{lem}
\label{minusId}
If $\calE$ is a generic member of either family $i)$ or $ii)$, then $-\id\in\Gamma(\calE)$.
\end{lem}

\begin{proof}
Here we use the semicontinuity argument for the monodromy in the other
direction. Since the group can only get smaller under specialisation, the
generic monodromy for the family contains $-\id$, if it does so for some member.

If $\alpha_4=\alpha_3\gamma_1$ and $\beta_4=\beta_3\gamma_1$ with
$\alpha_3,\beta_3,\gamma_1$ pairwise coprime, then the Tate algorithm shows
that for either family we get an $I_0^*$ fibre corresponding to $\gamma_1$.
Hence $-\id$ is in the monodromy of such a special member.
\end{proof}

%%%%%%%%%%%%%%%%%%%%%%%%%%%%%%%%%%%%
\section{Classification for $\jGamma$ of low degree}
\label{sec:low}

We are now ready to classify the ambi-typical strata with $\barGamma$ of index at most $6$.
At the same time we determine the corresponding root lattices and Mordell-Weil torsion.

\begin{pro}
\label{Gamma12}
There is a unique ambi-typical stratum with modular monodromy $\barGamma_1(2)$.
Its monodromy group is $\Gamma_1(2)$.
A generic element of this stratum has the following root lattice and Mordell-Weil torsion:
\[
8 A_1, \, \ZZ/2\ZZ.
\]
\end{pro}

\begin{proof}
Since $-\id$ is the only $2$-torsion element in $\slz$, there is no splitting map to $\slz$
from any subgroup of $\pslz$ containing $2$-torsion such as $\barGamma_1(2)$.
In particular, every $\barGamma_1(2)$
monodromy stratum is actually a $\Gamma_1(2)$ monodromy stratum.

With our explicit knowledge of $\jGamma$ and $j_{\calE}$ for the family $v)$
we can use the Tate algorithm and find that the generic fibre configuration is $8I_1+8I_2$.
By Lemma \ref{dimLoc} we find that the dimension of the configuration locus is $\dim L(8I_1+8I_2)=10$.
The corresponding generic root lattice is $8A_1$.  The Shimada stratum with this root lattice and saturation given by the generic element of family $v)$
also has dimension 10.   
This stratum and family $v)$ therefore determine the same irreducible closed subset of $\calF'$.
It is ambi-typical by the argument of \eqref{dim2}, since $\Gamma_1(2)$ belongs to the second
column in Table \ref{table:maxdim} and thus $m(\Gamma_1(2))\leq 10$ by Lemma \ref{maxdim}.

Any other Shimada stratum with $\barGamma_1(2)$ monodromy corresponds by Proposition \ref{weier} 
to a stratum in family $v)$ and must be of dimension less than $10$.
But its monodromy group must still be $\Gamma_1(2)$, so it belongs to the monodromy 
stratum above, see \eqref{closure},
and cannot be a monodromy stratum of its own. 

By Lemma \ref{coxparry} the Mordell-Weil torsion always contains $\ZZ/2\ZZ$ and for a generic element this is equal to $\ZZ/2\ZZ$. The proof is analogous to that in the proof of Lemma \ref{lem:G3MW}.
\end{proof}

\begin{pro}
\label{Gamma2}
There is a unique ambi-typical stratum with modular monodromy $\barGamma(2)$ 
and monodromy group $\Gamma(2)$.
A generic element of this stratum has the following root lattice and Mordell-Weil torsion:
\[
12 A_1, \, \ZZ/2\ZZ \times\ZZ/2\ZZ.
\]
\end{pro}

\begin{proof}
$\Gamma(2)$ is the  pre-image
of $\barGamma(2)$ in $\slz$ and contains $-\id$,
thus by Lemma \ref{minusId} it is the monodromy group of the generic surface in family $i)$.
We can now argue very much as in the proof of the previous proposition.
The generic fibre configuration is $12I_2$ which gives a generic root lattice $12A_1$,
so both the corresponding strata are of dimension $6$.

On the other hand $\Gamma(2)$ belongs to Column 5 in Table \ref{table:maxdim}.
Hence $m(\Gamma(2))\leq 6$ by Lemma \ref{maxdim}, and the irreducible
closed subset of $\calF'$ corresponding to family $i)$ is ambi-typical by \eqref{dim2}.

Any other Shimada stratum with $\Gamma(2)$ monodromy corresponds
to a stratum in family $i)$ by Proposition \ref{weier} and must be of dimension less than $6$.
Due to \eqref{closure} it belongs to the monodromy 
stratum above, and can not be a monodromy stratum of its own. 

The claim about the Mordell-Weil torsion again follows as in the previous proof.
\end{proof}

\begin{pro}
\label{Gamma04}
There is a unique ambi-typical stratum with modular monodromy $\barGamma_1(4)$ 
and monodromy group $\Gamma_0(4)$.
A generic element of this stratum has the following root lattice and Mordell-Weil torsion:
\[
4 A_3, \, \ZZ/2\ZZ.
\]
\end{pro}

\begin{proof}
$\Gamma_0(4)$ 
is the pre-image in $\slz$ of $\barGamma_1(4)=\barGamma_0(4)$
(see Remark \ref{Gamma0}), thus contains $-\id$ and
-- by Lemma \ref{minusId} -- is the monodromy group of the generic surface in family $ii)$.
The argument then is as above, only applied to the family $ii)$.
\end{proof}

To complete the classification of monodromy strata for the two torsion free subgroups
of $\pslz$, we have to exploit the fact, that for fibrations with monodromy $\Gamma$ 
not containing $-\id$ the quotient map $\slz \to \pslz$ defines an isomorphism $\Gamma\cong\barGamma$.

\begin{pro}
\label{Gamma2'}
There is a unique ambi-typical stratum with modular monodromy $\barGamma(2)$ 
and monodromy group $\Gamma$ such that $\Gamma \to \barGamma(2)$ is an isomorphism.
The generic root lattice and Mordell-Weil torsion are
\[
2 A_3 + 8A_1, \, \ZZ/2\ZZ \times\ZZ/2\ZZ.
\]
\end{pro}

\begin{rem}
The group $\Gamma$ in the proposition can be described as
\[
\Gamma= \left\{
\begin{pmatrix} a & b \\ c & d \end{pmatrix} \Big| \quad 
b,c \equiv 0 \mod 2,\quad a,d \equiv 1 \mod 4 \right\}.
\]
\end{rem}

\begin{proof}
By Proposition\ \ref{weier} the surfaces of any $\Gamma$ monodromy stratum are
contained in family $i)$ where $\jGamma$ by \eqref{row1} of table \ref{table:j-functions}
has three poles at $0$ and $\pm3$.
The splitting isomorphism gives a map from 
$\pi_1^{orb} (\HH_1/\bar{\Gamma})\cong \barGamma$
to $\Gamma\subset\slz$ and $\HH_1/\bar{\Gamma}$ is identified with the complement in $\PP^1$ 
of these poles.
The positive loop at each pole is mapped by $(\jGamma)_*$ to the conjugacy class of 
$\pm(\begin{smallmatrix} 1 & 2 \\ 0 & 1 \end{smallmatrix})$. 
Since $1$ is chosen to represent the class of $\pm(\begin{smallmatrix} 1 & 1 \\ 0 & 1 \end{smallmatrix})$  this is
further mapped to $2\in \ZZ/6\ZZ\cong\pslz_{ab}$.
Hence the splitting isomorphism associates to each
loop the conjugacy class of $(\begin{smallmatrix} 1 & 2 \\ 0 & 1 \end{smallmatrix})$ or
$(\begin{smallmatrix} -1 & -2 \\ 0 & -1 \end{smallmatrix})$ corresponding to $2$ and\ $8$
in $\ZZ/{12\ZZ}=\slz_{ab}$ respectively.
Now the sum of the three elements in $\slz_{ab}$ must be zero, since the sum in
homology of the three loops is zero.
Under the assumption of a splitting isomorphism we therefore 
can assume that the conjugacy
class associated to the loop around $0$ contains 
$(\begin{smallmatrix} -1 & -2 \\ 0 & -1 \end{smallmatrix})$
-- possibly after applying a deck transformation in $\barGamma(1)/\barGamma(2)$
of $j_{\barGamma(2)}$ acting by the full permutation group on $\{ 0, \pm 3 \}$.
 
We now consider the map $j_\calE= (\alpha_4:\beta_4)$ and try to understand the
possible fibres over pre-images of $0$.
Such a pre-image corresponds to a linear factor $\gamma_1$ of $\alpha_4$ which may
occur with multiplicity $1\leq k\leq 4$. It may occur in $\beta_4$ with multiplicity $l=0$
or $l=1$, as long as $k>l$. Indeed, in case $k\leq l$ the image would not be $0$, in case $k>l>1$
the Weierstra{\ss} coefficients $g_2,g_3$ would share a common factor of $\gamma_1^{12}$
which is not allowed.

In all possible combinations we can compute the local fibre type at the zero of $\gamma_1$
from the local monodromy.
Recall that  we made an assumption on the matrix associated to the loop around $0$ in the codomain of $j_\calE$.
Since the multiplicity of $j_\calE$ is $k-l$ at the zero of $\gamma_1$ in the domain of $j_\calE$, the matrix  associated to the loop around this
zero must be conjugate to $(\begin{smallmatrix} -1 & -2 \\ 0 & -1 \end{smallmatrix})^{k-l}$.
These monodromy matrices determine the local fibre type by Table  \ref{table:kodaira} in the way recorded in the
last row of Table \ref{table:types}.

The other row gives the local fibre types determined by the Tate algorithm:
The pole order of $j(\calE)$ is the product of the multiplicities of $j_\calE$ at the zero of $\gamma_1$
and of $\jGamma$ at $0$, hence $2(k-l)$.
The vanishing orders of $g_2,g_3$ at the zero of $\gamma_1$ are determined by row $i)$ in Table
\ref{table:weierstrassfamilies}, so for $l=0$ they do not vanish simultaneously, and $\nu_2=2,\nu_3=3$ for $l=1$.

\begin{table}[htb]
\label{table:types}
\[
\begin{array}{c|c|c|c|c|c|c|c|}
(k,l) & 1,0 & 2,0 & 3,0 & 4,0 & 2,1 & 3,1 & 4,1 \\
\hline
\text{Tate} & I_2 & I_4 & I_6 & I_8 & I_2^* & I_4^* & I_6^* \\
\text{cover} & I_2^* & I_4 & I_6^* & I_8 & I_2^* & I_4 & I_6^* 
\end{array}
\]
\caption{Comparisons of fibre type calculations}
\end{table}

Since the fibre types in both rows of Table \ref{table:types} must coincide, our assumption on the splitting implies that $k$ is even.
Hence we may conclude,
that $\alpha_4=\gamma_2^2$. 
This condition generically corresponds to a map $j_\calE$
branched outside the other poles of $\jGamma$.
Therefore $j_\calE$ induces a surjection on orbifold
fundamental groups and thus the monodromy group is $\Gamma$.

Generically there are $I_4$ fibres at the zeroes of $j_\calE$ and an $I_2$ fibre
at each of the four unramified pre-images of each of the other two poles $\pm3$ of $\jGamma$.
This yields generic fibre type $2I_4+8I_2$ and generic root lattice $2A_3+8A_1$
with corresponding strata of dimension $4$.
So we found the unique ambi-typical $\Gamma$ stratum of the claim, since every stratum
of root lattice type with constant $\Gamma$ monodromy was shown to belong
to the closed set in family $i)$ defined by $\alpha_4=\gamma_2^2$
and we can use conclusion \eqref{closure} again. 

The assertion about the Mordell-Weil torsion follows from Lemma \ref{coxparry}
since the modular monodromy is $\barGamma(2)$ and only contained in
$\barGamma_m(n)$ if $m,n$ divide $2$.
\end{proof}

\begin{rem}\label{rem:explanationShimadastrata}
Note that the generic root lattice and the Mordell-Weil torsion
do not necessarily determine a unique Shimada stratum.
In fact, Entry 95 in  Table 3 of Shimada \cite{shimada2}
shows that there are two families with root lattice $2 A_3 + 8A_1$ and Mordell-Weil torsion $ \ZZ/2\ZZ \times\ZZ/2\ZZ$, 
which can be distinguished \emph{algebraically}. That -- in the terminology of Shimada --
means, that the saturations of the root lattices are non-isomorphic for the two families.

Shimada also gives a recipe how to analyse such a situation in Section 6.1 and his Example
6.4 corresponding to number 91 in his table is similar to our case in many aspects.
By lattice theory the two possible saturations in case 95 correspond to two isotropy subgroups of the discriminant group of
$2A_3+8A_1$ up to isometry.
Shimada computes these and provides his result on his home page \cite{shimada3} . 
The subgroup belonging to our family consists of the trivial element and the isotropic elements in
$\operatorname{discr}(2 A_3 + 8A_1) =(\ZZ/4\ZZ)^2\times (\ZZ/2\ZZ)^8$ given by
\[
(0,0,1,1,1,1,1,1,1,1),\quad (2,2,0,0,0,0,1,1,1,1),\quad (2,2,1,1,1,1,0,0,0,0),
\]
where a basis for the discriminant group is chosen in the obvious way.
To prove that we are in that case requires a detailed analysis of the torsion sections and
their intersections with fibre components. We shall not give the details here.
\end{rem}

\begin{pro}
\label{Gamma04'}
There are two ambi-typical strata with modular monodromy $\barGamma_1(4)$ 
and monodromy group $\Gamma$ such that $\Gamma \to \barGamma_1(4)$ is an isomorphism.
Their generic root lattice and Mordell-Weil torsion are respectively
\begin{itemize}
\item[\rm{(1)}] $4A_3 + 2A_1, \, \ZZ/4\ZZ$, 
\item[\rm{(2)}] $2 A_7, \, \ZZ/2\ZZ$.
\end{itemize}
The two  corresponding subgroups in $\slz$ are not conjugate.
\end{pro}

\begin{rem}
The group is $\Gamma_1(4)$ in case $(1)$ and in case $(2)$ can be given by
\[
\Gamma= \Big\{ \quad M \quad \Big| \quad  M \equiv
\begin{pmatrix} -1 & 1 \\ 0 & -1 \end{pmatrix}^k \mod 4, \quad k=0,1,2,3  \Big\}.
\]
\end{rem}

\begin{proof}
We argue as in the last proof, but with  family $ii)$ and the associated map $\jGamma$
which by \eqref{row2} has poles at $0$, $4$ and $\infty$.
The analysis of the splitting isomorphism again provides information about the monodromy
at these poles.
However, this time only the poles at $0$ and $4$ are equivalent under deck transformation,
so we have to consider two cases instead of one
\begin{enumerate}
\item
the conjugacy class associated to the pole at $0$ contains 
$(\begin{smallmatrix} -1 & -1 \\ 0 & -1 \end{smallmatrix})$,
\item
the conjugacy class associated to the pole at $\infty$ contains 
$(\begin{smallmatrix} -1 & -4 \\ 0 & -1 \end{smallmatrix})$.
\end{enumerate}
In the first case, we apply the Tate algorithm to $j_\calE= (\alpha_4:\beta_4)$ and compare local monodromies
at its zeroes. We infer $\alpha_4=\alpha_2^2$, then compute the generic fibre configuration
$2I_2+ 4 I_1 + 4 I_4$, the generic root lattice $4A_3+2A_1$ and the dimension $4$ of
the corresponding strata.
In the second case we look at the poles instead and get
$ 2I_8+ 8I_1$, $2 A_7$ and again dimension 4.

We can argue as before that these yield the ambi-typical strata of the claim and that there
are no more.

Lemma \ref{mwtorsion6} provides the Mordell-Weil torsion groups, showing in particular  that  the
two subgroups are not conjugate.
\end{proof}

\begin{pro}
\label{Gamma13}
There is a unique ambi-typical stratum with modular monodromy $\barGamma_1(3)$.
Its monodromy group is $\Gamma_1(3)$.
Its generic root lattice and Mordell-Weil torsion are
\[
6 A_2, \, \ZZ/3.
\]
\end{pro}

\begin{proof}
Any ambi-typical stratum with $\barGamma_1(3)$ modular monodromy belongs
to family $iii)$ or $iv)$.

If it belongs to family $iii)$ then by Lemma \ref{lem:G3MW}, the Mordell-Weil torsion contains $\ZZ/3$ and $-\id\not\in\Gamma$.
By the former property and Lemma \ref{coxparry} family $iii)$ has monodromy contained
in $\Gamma_1(3)$.
Since $-\id\not\in\Gamma_1(3)$, no proper subgroup 
surjects onto $\barGamma_1(3)$, thus $\Gamma_1(3)$ is the monodromy group for the generic surface
in family $iii)$.

We follow the proofs of Propositions \ref{Gamma12} and \ref{Gamma2} and get generic fibre configuration $6I_3+6I_1$,
generic root lattice $6A_2$, corresponding strata dimension 6 and $m(\Gamma_1(3))\leq 6$
by Lemma \ref{maxdim}. Using \eqref{dim2} and \eqref{closure} we conclude that 
family $iii)$ constitutes a unique ambi-typical stratum, with the invariants given in the claim.

To address possible strata in family $iv)$
we look at the additional family with Weierstra{\ss} datum
\[
g_2= 3 \alpha_1\alpha_2(\alpha_1^3\alpha_2+8\beta_5),\quad
g_3=\alpha_2(\alpha_1^6\alpha_2^2-20\alpha_1^3\alpha_2\beta_5-8\beta_5^2),
\]
that has associated $j$-invariant
\[
j \quad =\quad 
\frac{(g_2/3)^3}{(g_2/3)^3-g_3^2} \quad = \quad
\frac{\alpha_1^3\alpha_2(\alpha_1^3\alpha_2+8\beta_5)^3}{64(\alpha_1^3\alpha_2-\beta_5)^3\beta_5^2}
\]
which is the composition of $\jGamma$ for $\barGamma_1(3)$ with $j_\calE=\alpha_1^3\alpha_2/\beta_5$.
It contains the family $iv)$ as a proper closed subset determined by the specialization to
$\alpha_2=\gamma_2, \beta_5=\gamma_2\beta_3$. A zero of $\gamma_2$ generically corresponds to
a fibre of type $I_0^*$ so both have generic monodromy $\Gamma_0(3)$,
the pre-image in $\slz$ of $\barGamma_1(3)=\barGamma_0(3)$.

Repeating the argument above using \eqref{closure} in particular, we deduce that no stratum
of monodromy $\Gamma_0(3)$ is ambi-typical, except possibly that corresponding to the generic fibre
configuration $2II+5I_3+5I_1$ in the additional family. However, by Proposition \ref{prop:restricitionfibres} the fibre $II$ may not occur in generic surfaces of an
ambi-typical stratum.

We are left to look for an ambi-typical stratum in family $iv)$ 
with $\Gamma\to \barGamma_1(3)$ an isomorphism.
The splitting isomorphism gives a map from 
$\pi_1^{orb}(\HH/\Gamma) \cong \barGamma$
to $\Gamma\subset\slz$.
The positive loops around the $3$-torsion point and the poles of $\jGamma$ of order $1$ and $3$ are mapped to 
$2$, $1$ and $3$ in $\ZZ/6=\pslz_{ab}$, so the splitting isomorphism associates to these loops the elements
$2$ or $8$, $1$ or $7$, and $3$ or $9$ respectively in $\ZZ/{12}=\slz_{ab}$.

In fact, we can be more precise. On the one hand $2$ can not be associated to a $3$-torsion point, since then the corresponding monodromy
is conjugate to $(\begin{smallmatrix} 1 & 1 \\ -1 & 0 \end{smallmatrix})$ contradicting $-\id\not\in\Gamma$.
On the other hand $1$ and $3$ can not be associated to the two poles. Otherwise the corresponding elements
of $\Gamma$ are conjugate to powers of  $(\begin{smallmatrix} 1 & 1 \\ 0 & 1 \end{smallmatrix})$ and so are
all monodromies at poles of $j(\calE)$. Then the $I^*$ fibres -- present according to Table \ref{table:columnslow}
-- are $I_0^*$ fibres contradicting again $-\id\not\in\Gamma$.

We use again the fact  that the sum of the three elements in $\slz_{ab}$ must be zero, since the sum in
homology of the three loops is zero. 
Thus the elements must be $8,7$ and $9$, and $\calE$ is a pullback (up to normalization and blow down) 
of an elliptic surface with singular fibre configuration $I_1^*+I_3^*+IV^*$ along $j_\calE$.

By Table \ref{table:jE-functions} the map $j_\calE$ has degree $3$ and ramification datum $(3)$ at the $3$-torsion point.
By Table \ref{table:columnslow} the surface $\calE$ has two $I^*$ fibres, so the ramification datum of $j_\calE$ at the
two poles must be $(2,1)$.
But then $j_\calE$ is unramified outside these points and 
therefore violates the maximality condition \eqref{maximality}.
\end{proof}

\begin{pro}
There is no ambi-typical stratum with $\barGamma=\barGamma(1)^2$.
\end{pro}

\begin{proof}
We first assume that there is such a stratum with $\Gamma\to\barGamma(1)^2$ an isomorphism.
By Lemma \ref{weier} we know that there is a corresponding generic surface $\calE$ in the family $vi)$.

The splitting isomorphism gives a map from 
$\pi_1^{orb}(\HH/\Gamma) \cong \barGamma$
to $\Gamma\subset\slz$.
The positive loops around both $3$-torsion points and the pole of $\jGamma$ are mapped to 
$2\in \ZZ/6=\pslz_{ab}$, so the splitting isomorphism associates to each
loop either $2$ or $8$ in $\ZZ/{12}=\slz_{ab}$.
However, $2$ can not be associated to a $3$-torsion point, since then the corresponding monodromy
is conjugate to $(\begin{smallmatrix} 1 & 1 \\ -1 & 0 \end{smallmatrix})$ contradicting $-\id\not\in\Gamma$.

Again we use the fact  that the sum of the three elements in $\slz_{ab}$ must be zero.
Thus all elements must be $8$ and $\calE$ is a pullback (up to normalization and blow down) 
of an elliptic surface with singular fibre configuration $I_2^*+2IV^*$ along $j_\calE$.

By Table \ref{table:jE-functions} the map $j_\calE$ has degree $4$ and ramification datum $(3,1)$ at the two
$3$-torsion points.
By Table \ref{table:columnslow} the surface $\calE$ has no $I^*$ fibre, so the ramification datum of $j_\calE$ at the
pole must be $(2,2)$.
But then $j_\calE$ is unramified outside these points and 
therefore violates the maximality condition \eqref{maximality}.

To show that there is neither a stratum with $-\id \in\Gamma$ we look at the new
family with Weierstra{\ss} datum given by 
\[
g_2= -12 \alpha_1\alpha_3\beta_1\beta_3,\quad
g_3=4\alpha_3\beta_3(\alpha_1^3\alpha_3-\beta_1^3\beta_3)
\]
that has associated $j$-invariant
\[
j \quad =\quad 
\frac{(g_2/3)^3}{(g_2/3)^3-g_3^2} \quad = \quad
\frac{4\alpha_3\beta_3(\alpha_1\beta_1)^3}{(\alpha_1^3\alpha_3+\beta_1^3\beta_3)^2}
\]
which is the composition of $\jGamma$ for $\barGamma(1)^2$ with $j_\calE=(\alpha_1^3\alpha_3:\beta_1^3\beta_3)$.
It contains the family $vi)$ as a proper closed subset determined by the specialization to
$\alpha_3=\gamma_1^2\delta_1, \beta_3=\gamma_1\delta_1^2$, so both have the same generic monodromy $\Gamma$.

Repeating the argument above, using in particular \eqref{closure}, we deduce that no stratum
of that monodromy is ambi-typical, except possibly that corresponding to the generic fibre
configuration $6II+6I_2$ of the new family. However, again by Proposition \ref{prop:restricitionfibres}, 
the fibre $II$ may not occur in generic surfaces of an
ambi-typical stratum.
\end{proof}

%%%%%%%%%%%%%%%%%%%%%%%%%%%%%%%%%%%%%%%%
\section{Classification in the high index cases}
\label{sec:high}

In this section we will complete the classification of ambi-typical strata of high index.  
More precisely, we will classify the subgroups $\barGamma$ of $\pslz$ of index at least $9$
for which an ambi-typical stratum exists and give a description of the possible monodromy groups $\Gamma$
in $\slz$ which cover $\barGamma$.

\subsection{Uniqueness of strata}\label{subsec:highindexuniqueness}

We shall proceed as follows.
First, given a group $\barGamma$, we shall determine the possible groups $\Gamma$  and prove  the uniqueness of the $\Gamma$ ambi-typical strata. 
Second, we will determine the root lattices, $j_\calE$ branch data and the number of $*$-fibres  of their 
generic members in terms of the ramification data of $\jGamma$. 
In fact, we use the notation as in  \cite{BPT} for the three cases that $j_\calE$ can belong to:
\begin{itemize}
\item[$(2)_B,2$]: \hspace{.6cm}
$j_\calE$ is of degree $2$ with branching in a $2$-torsion point and a non-torsion point, 
\item [$2^2$]:\hspace{.8cm}
$j_\calE$ is of degree $2$ with branching in two non-torsion points, 
\item[$1$]:\hspace{.8cm}
$j_\calE$ is of degree $1$ with no branching.
\end{itemize}

The actual list of ramification data of $\jGamma$, and thus the list of possible groups $\barGamma$
is postponed to the next subsection, as is the computation of the corresponding
Mordell-Weil torsion.

\begin{pro}
\label{index9root}
Suppose $\barGamma$ has index $9$ in $\pslz$. Any such subgroup defines at most one ambi-typical stratum and in this case 
the monodromy group $\Gamma$ is the pre-image of $\barGamma$ in $\slz$. 

The generic invariants, root lattice, 
$j_\calE$ ramification and
number of $*$-fibres are
\[
D_4 + 2 A_{i_1} + \dots + 2 A_{i_k}, \qquad (2)_B,2,\qquad 1,
\]
with $k$ the number of poles of $\jGamma$ of order at least two,
$i_1+1,\dots, i_k+1$ their orders
and root lattice rank $4+2i_1+ \dots + 2i_k = 16$.
\end{pro}

\begin{proof}
By Column $4$ of Table \ref{table:columnshigh} the map $j_\calE$ is a double cover and we have $e_2=1$ and $e_3=0$.
So $\barGamma$ contains $2$-torsion and has no splitting map to $\slz$,
since $-\id$ is the only $2$-torsion element in $\slz$.
In particular, every $\barGamma$
monodromy stratum is actually a $\Gamma$ monodromy stratum with $\Gamma$ the pre-image.

The map $j_\calE$ is branched over the unique $2$-torsion point
according to  Lemma \ref{rami}(1). 
The family of such coverings is $1$-dimensional and irreducible.

Moreover, it follows from  Proposition \ref{badfibres} that the $I^*$-fibre in the generic member of the stratum must be $I_0^*$.
Therefore we get a $2$-dimensional irreducible family of elliptic surfaces
with monodromy $\Gamma$, where the second parameter is the position of
the $I_0^*$-fibre.
If we have an ambi-typical stratum the rank of the root lattice must be $16$.

The generic fibre type consists of the $I_0^*$-fibre and the $I_\nu$-fibres
corresponding to pole orders of the generic $j$-map $\jGamma\circ j_\calE$,
which coincides with the lengths of the ramification partition at infinity.
Since $\jGamma\circ j_\calE$ is the canonical factorisation,
the map $j_\calE$ is a double cover not branched over $\jGamma^{-1}(\infty)$, 
and the ramification partition is twice that of $\jGamma$.

We can now use the argument in \eqref{dim2} to conclude that $\Gamma$ defines a monodromy stratum. This is the case since our family
has varying moduli and is of maximal dimension by Column 4 of Table \ref{table:maxdim}.
\end{proof}

\begin{pro}
\label{index12root}
Suppose $\barGamma$ is of index $12$ in $\pslz$. Let $k$ be the number of poles of $\jGamma$ of order at least two
and $i_1+1,\dots, i_k+1$ their orders. If $\barGamma$ gives rise to an ambi-typical stratum then it is torsion free and there are two possibilities:
\begin{enumerate}
\item
there is a unique ambi-typical $\Gamma$ stratum
with $\Gamma \to \barGamma$ an isomorphism.
The generic invariants, root lattice of rank $16$, $j_\calE$ ramification and
number of $*$-fibres are
\[
2 A_{i_1} + \dots + 2 A_{i_k}, \qquad 2^2,\qquad 0,
\]
\item
there is a unique ambi-typical $\Gamma$ stratum
with $\Gamma\subset \slz$ the pre-image of $\barGamma$.
The generic invariants, root lattice of rank $16$, $j_\calE$ ramification and
number of $*$-fibres are
\[
2D_4 +  A_{i_1} + \dots +  A_{i_k}, \qquad 1,\qquad 2.
\]
\end{enumerate}
\end{pro}

\begin{proof}
By the hypothesis on $\barGamma$ we are either in either the cases of Column $6$ or 
Column $7$ of Table \ref{table:columnshigh} and we already know that $\barGamma$ is torsion free by Proposition \ref{modular_monodromy}(2).

In case of Column $6$ of Table \ref{table:columnshigh} the map $j_\calE$ is an isomorphism.
Moreover the two $I^*$-fibres in the general member of the stratum must be of type $I_0^*$
according to Proposition \ref{badfibres}, since otherwise the corresponding surface were rigid.
Therefore we get a $2$-dimensional irreducible family of elliptic surfaces
with monodromy $\Gamma$ containing $-\id$, where the parameters are the positions of
the $I^*$-fibres, given by a reduced divisor of degree two on $\PP^1$.

The generic fibre type consists of the $I_0^*$-fibres and the $I_\nu$-fibres
corresponding to the poles of the generic $j$-map $\jGamma\circ j_\calE$.
Since $j_\calE$ is a isomorphism, the pole orders are those of $\jGamma$.

In case of Column $7$ 
of Table \ref{table:columnshigh} the map $j_\calE$ is a
double cover.
It is branched generically at two points, since Lemma \ref{rami} imposes no
additional conditions.
The family of such coverings is $2$-dimensional and irreducible.
The generic fibre type consists of the $I_\nu$-fibres
corresponding to the pole orders of the generic $j$-map $\jGamma\circ j_\calE$,
which coincide with the lengths of the ramification partition at infinity.
Since $\jGamma\circ j_\calE$ is the canonical factorisation,
the map $j_\calE$ is a double cover not branched over $\jGamma^{-1}(\infty)$, 
and the ramification partition is twice that of $\jGamma$.

By Column 6 of Table \ref{table:maxdim} our family is not contained in a monodromy stratum of
larger dimension.

Since the dimension of both families is two, the rank of the root lattice must be $16$.
\end{proof}

\begin{pro}
\label{index18root}
Suppose $\barGamma$ is of index $18$ in $\pslz$. If $\barGamma$ is associated  to an  ambi-typical stratum, then it is torsion free and the monodromy group  $\Gamma$ is
the pre-image of $\barGamma$ in $\slz$.

The generic invariants, root lattice of rank $17$, $j_\calE$ ramification and
number of $*$-fibres are
\[
D_4 + A_{i_1} + \dots + A_{i_k}, \qquad 1, \qquad 1,
\]
with $k$ the number of poles of $\jGamma$ of order at least two, and
$i_1+1,\dots, i_k+1$ their orders.
\end{pro}

\begin{proof}
By Proposition \ref{modular_monodromy}(1) we know that $\barGamma$ is torsion free. 
We also note that the $I^*$-fibre in the general member of the stratum must 
be of type $I_0^*$, since $j_\calE$ is of degree $1$ by Column $9$ of Table 
\ref{table:columnshigh}, thus positive dimensionality of the stratum must be due
to moduli of a movable $I_0^*$-fibre. Accordingly $-\id\in\Gamma$ and
$\Gamma$ is the pre-image in $\slz$ of $\barGamma$.

Therefore we get a $1$-dimensional irreducible family of elliptic surfaces
with monodromy $\Gamma$, where the parameter is the position of
the $I^*_0$-fibre, consequently the rank of the root lattice must be $17$.

The generic fibre type consists of the $I_0^*$-fibre and the $I_\nu$-fibres
corresponding to the poles of the generic $j$-map $\jGamma\circ j_\calE$.
Since $j_\calE$ is an isomorphism, the pole orders
are those of $\jGamma$.

The property to be a monodromy stratum follows as before by \eqref{dim2}
and Table \ref{table:maxdim}.
\end{proof}

%%%%%%
\subsection{Classification of relevant subgroups in the high index cases}

In the previous subsection we have related the monodromy strata uniquely, resp.\
in a two to one way, with subgroups of $\pslz$ in the high index case. So
we have to classify these subgroups next.

To this end we note that subgroups $\barGamma$ are in bijective 
correspondence with maps $\jGamma$ and such maps are in turn determined by
triples of permutations $\mu_0,\mu_1,\mu_\infty$ whose product is the identity and which
generate a group acting transitively on the set of $9,12$ and $18$ elements respectively. 
Since Condition \eqref{belyj1} and the value of $e_2,e_3$ impose restrictions
on $\mu_0$ and $\mu_1$, we deduce:
\begin{enumerate}
\item
In the index $9$ case, $\mu_0$ has only $3$-cycles and $\mu_1$ only $2$-cycles
except for one fixed point.
\item
In the torsion free cases, where the index is $12$ or $18$, 
$\mu_0$ has only $3$-cycles and $\mu_1$ only $2$-cycles.
\end{enumerate}
Moreover, concerning the fibres of type $I_\nu,\nu>0$: 
\begin{enumerate}
\item[(3)]
The number of parts of $\mu_\infty$ is the number of poles and the pole orders
are the lengths of the parts and determine the corresponding fibre.
\end{enumerate}

\begin{pro}
\label{index9groups}
In the index $9$ subcase there are $4$ relevant subgroups $\barGamma$ of $\pslz$. 
These are in bijection with the following $4$ partitions of $\mu_\infty$ of the corresponding map $\jGamma$:
\[
(7,1,1), (6,2,1), (5,3,1), (4,3,2).
\]
This translates into the following root lattices:
\[
D_4+2A_6, D_4+2A_5+2A_1, D_4+2A_4+2A_2, D_4+2A_3+2A_2+2A_1.
\]
with corresponding Mordell-Weil torsion
\[
trivial,\quad \ZZ/2,\quad trivial, \quad \ZZ/2.
\]
\end{pro}

\begin{proof}
Since each relevant subgroup gives an irreducible stratum their number is
bounded above by the number of Shimada components which have
root lattice as described in Proposition \ref{index9root} above. The list of Shimada 
\cite[http://arxiv.org/pdf/math/0505140.pdf]{shimada} published in the arXiv version
shows that there are four such root lattices, 
entries $2171, 2179, 2190$ and $2198$, which uniquely determine the Mordell-Weil torsion
as claimed.
They do not appear in the list of root
lattices corresponding to multiple components, see
Shimada \cite[Table II, p.38]{shimada2}, therefore each of these contributes one component.

To see that these Shimada strata are indeed ambi-typical
it suffices to exhibit the four groups, which we do  in terms of four tuples
of permutations
\begin{eqnarray}
\label{cycle711}
 (123)(456)(789), & (14)(27)(56)(89), &(1643297)(5)(8)\\
\label{cycle621}
 (123)(456)(789), & (14)(26)(57)(89), &(16)(259743)(8)\\
\label{cycle531}
 (123)(456)(789), & (14)(25)(67)(89), &(16975)(243)(8)\\
\label{cycle432}
 (123)(456)(789), & (14)(27)(59)(68), &(167)(2943)(58)
\end{eqnarray}
where we use the correspondence from $(3)$ to $(1)$ in Remark \ref{coversubgroup}.
\end{proof}

\begin{rem}
One can give an alternative proof which avoids the use of Shimada's list by performing an exhaustive search for all homomorphisms
as in $(3)$ of Remark \ref{coversubgroup}. This is straightforward though cumbersome and we omit the details.
\end{rem}

In the index $12$ subcase, we get, by Proposition \ref{index12root} two monodromy strata for each relevant subgroup
of $\pslz$.
As it turns out Beauville \cite{bea} has classified exactly these torsion free subgroups
in his study of semi-stable rational elliptic fibrations. Note that Beauville uses the notation
$\Gamma_0^0(n)$ for our groups $\Gamma_1(n)$.

\begin{pro}
\label{index12groups}
In the index $12$ subcase there are $6$ relevant subgroups 
of $\pslz$ with $\jGamma$ in bijection with the partition
corresponding to $\mu_\infty$ being one of 
\[
(9,1,1,1), (8,2,1,1), (6,3,2,1), (5,5,1,1), (4,4,2,2), (3,3,3,3).
\]
\end{pro}

This claim is proved in \cite[p.658f]{bea}, which is used also in
the proof of the following proposition addressing case $(1)$ of Proposition \ref{index12root}.

\begin{pro}
\label{beauville}
If $j_\calE$ is generic of degree $2$, so $\calE$ is the pull-back along $j_\calE$ of a
rational modular elliptic surface without *-fibre,
then the monodromy group $\Gamma$ is one of the following 
\[
\Gamma_0(9)\cap\Gamma_1(3),\quad\Gamma_0(8)\cap\Gamma_1(4),\quad \Gamma_1(6),\quad \Gamma_1(5),
\quad \Gamma_1(4)\cap\Gamma(2),\quad \Gamma(3)
\]
with corresponding Mordell-Weil torsion
\[
\ZZ/3,\quad \ZZ/4,\quad \ZZ/6,\quad \ZZ/5,
\quad\ZZ/4\times\ZZ/2,\quad \ZZ/3\times\ZZ/3.
\]
\end{pro}

\begin{proof}
The list for the monodromy groups of the rational modular
elliptic surfaces is taken from \cite{bea}.
Indeed, they do not change under generic pull-back, since
${(j_{\calE}})_*$ induces a surjection on fundamental groups of the complements of
singular values.
The claim of the Mordell-Weil torsion is then 
obtained using Lemma \ref{coxparry}.
It can also be verified by a check of the corresponding
lines 2242, 2262, 2322, 2345, 2368, 2373 in \cite{shimada}.
\end{proof}

The other families of surfaces with torsion free modular monodromy of index $12$ are obtained
from the Beauville elliptic surfaces, which are rational and rigid
for deformations preserving the monodromy group, by replacing  two smooth fibres by
fibres of type $I_0^*$. This \emph{generic quadratic twisting} corresponds to 
the transformation of the Weierstra{\ss} datum $(g_2',g_3')$ of a rational Beauville elliptic surface
into Weierstra{\ss} data $(\gamma_2^2 g_2',\gamma_2^3 g_3')$, where
the zeroes of $\gamma_2$ avoid the singular fibres.
Geometrically this means the following. We first take the double cover branched at the smooth fibres over the zeroes
of $\gamma_2$. This double cover is acted on by the 
deck-transformation and the involution which is $-\id$ on each fibre.
The quotient by the product of these two involutions gives the twisted surface after resolution of its singularities.
The choice of $\gamma_2$, or equivalently the position of the $I_0^*$ fibres,  gives the two moduli of these families.

\begin{pro}
\label{twistedbeauville}
If $\calE$ is a rational modular elliptic surface without *-fibre \emph{twisted} at two smooth fibres,
then the monodromy group $\Gamma$ is one of the following 
\[
\Gamma_0(9),\quad\Gamma_0(8),\quad \Gamma_0(6),\quad \Gamma_1(5)\{\pm\id\},
\quad \Gamma_0(4)\cap\Gamma(2),\quad \Gamma(3)\{\pm\id\}
\]
with corresponding Mordell-Weil torsion
\[
trivial,\quad \ZZ/2,\quad \ZZ/2,\quad trivial,
\quad\ZZ/2\times\ZZ/2,\quad trivial.
\]
\end{pro}

\begin{proof}
The monodromy groups are the groups $\Gamma\{\pm\id\}$ generated by 
the groups from \cite{bea} and $-\id$ due to the presence of $I_0^*$ fibres.
The claim of the Mordell-Weil torsion is then 
obtained with Lemma \ref{coxparry}.
It is also immediate by a check of the corresponding
lines $2148 - 2153$ in \cite{shimada}
\end{proof}

\begin{pro}
\label{index18groups}
In the index $18$ subcase there are $26$ relevant subgroups of $\pslz$
corresponding to root lattices and Mordell-Weil torsion
in the list of Shimada given in lines
\[
2762 -  2786.
\]
The lattice appearing in line $2776$ gives rise to two different ambi-typical strata.
\end{pro}

\begin{rem}
We remark without proof that the two components corresponding to line $2776$ are complex conjugate, as are the maps $\jGamma$
for the corresponding modular monodromy groups. More precisely, this means that the corresponding maps 
$\mu:\pi_1(\CC\setminus\{0,1\}) \to S_{18}$ modulo conjugation in $S_{18}$ differ only by the automorphism of
$\pi_1(\CC\setminus\{0,1\})$ induced by complex conjugation.
\end{rem}

\begin{proof}
Since each relevant subgroup gives an irreducible stratum, their number is
bounded above by the number of Shimada strata
which have
root lattice of rank $17$ of the form $D_4+A_{i_1}+\dots+A_{i_k}$ 
as in Proposition \ref{index18root} above.
The list of Shimada\cite{shimada}
shows that there are $25$ such root lattices and the component count
of Shimada\cite{shimada2} shows that each of these contributes one component except for that
of line 2776 where two components exist.

To see that all these components are related to torsion free subgroups we
provide the corresponding list of $26$ triples of monodromy permutations together with
a proof, that the two triples corresponding to line $2776$ are not equal under
conjugation.
We give representatives for all orbits, with 
$\mu_1=(12)(34)(56)(78)(9\,10)(11\,12)(13\,14)(15\,16)(17\,18)$ and
$\mu_0,\mu_\infty$ in the following Table \ref{tab:jmonodromy}.

\begin{table}[htbp]
\[
{\scriptsize
\begin{array}{rlll}
1 & (123)(567)(9\:\!1\!\!\:0\!\:\:\!1\!\!\:1\!\:)(\:\!1\!\!\:3\!\:\:\!1\!\!\:4\!\:\:\!1\!\!\:5\!\:)(48\:\!1\!\!\:7\!\:)(\:\!1\!\!\:2\!\:\:\!1\!\!\:6\!\:\:\!1\!\!\:8\!\:),
(1)(5)(9)(\:\!1\!\!\:3\!\:)(23\:\!1\!\!\:7\!\:\:\!1\!\!\:6\!\:\:\!1\!\!\:4\!\:\:\!1\!\!\:5\!\:\:\!1\!\!\:2\!\:\:\!1\!\!\:0\!\:\:\!1\!\!\:1\!\:\:\!1\!\!\:8\!\: 8674) &1\!\!\:4\!\: 1 1 1 1
\\ 2 & (123)(567)(9\:\!1\!\!\:0\!\:\:\!1\!\!\:1\!\:)(48\:\!1\!\!\:7\!\:)(\:\!1\!\!\:2\!\:\:\!1\!\!\:5\!\:\:\!1\!\!\:7\!\:)(\:\!1\!\!\:4\!\:\:\!1\!\!\:8\!\:\:\!1\!\!\:6\!\:),
(1)(5)(9)(\:\!1\!\!\:5\!\:\:\!1\!\!\:8\!\:)(23\:\!1\!\!\:3\!\:\:\!1\!\!\:6\!\:\:\!1\!\!\:2\!\:\:\!1\!\!\:0\!\:\:\!1\!\!\:1\!\:\:\!1\!\!\:7\!\:\:\!1\!\!\:4\!\: 8674) & 1\!\!\:3\!\: 2 1 1 1
\\ 3 & (123)(567)(9\:\!1\!\!\:0\!\:\:\!1\!\!\:1\!\:)(4\:\!1\!\!\:3\!\:\:\!1\!\!\:8\!\:)(8\:\!1\!\!\:5\!\:\:\!1\!\!\:4\!\:)(\:\!1\!\!\:2\!\:\:\!1\!\!\:7\!\:\:\!1\!\!\:6\!\:),
(1)(5)(9)(\:\!1\!\!\:3\!\:\:\!1\!\!\:5\!\:\:\!1\!\!\:7\!\:)(23\:\!1\!\!\:8\!\:\:\!1\!\!\:2\!\:\:\!1\!\!\:0\!\:\:\!1\!\!\:1\!\:\:\!1\!\!\:6\!\: 867\:\!1\!\!\:4\!\: 4) & 1\!\!\:2\!\: 3 1 1 1
\\ 4 & (123)(458)(697)(\:\!1\!\!\:0\!\:\:\!1\!\!\:1\!\:\:\!1\!\!\:4\!\:)(\:\!1\!\!\:2\!\:\:\!1\!\!\:5\!\:\:\!1\!\!\:3\!\:)(\:\!1\!\!\:6\!\:\:\!1\!\!\:7\!\:\:\!1\!\!\:8\!\:),
(1)(\:\!1\!\!\:7\!\:)(57)(\:\!1\!\!\:1\!\:\:\!1\!\!\:3\!\:)(2389\:\!1\!\!\:4\!\:\:\!1\!\!\:5\!\:\:\!1\!\!\:8\!\:\:\!1\!\!\:6\!\:\:\!1\!\!\:2\!\:\:\!1\!\!\:0\!\: 64) & 1\!\!\:2\!\: 2 2 1 1
\\ 5 & (123)(567)(489)(\:\!1\!\!\:0\!\:\:\!1\!\!\:1\!\:\:\!1\!\!\:3\!\:)(\:\!1\!\!\:2\!\:\:\!1\!\!\:7\!\:\:\!1\!\!\:5\!\:)(\:\!1\!\!\:4\!\:\:\!1\!\!\:6\!\:\:\!1\!\!\:8\!\:),
(1)(5)(\:\!1\!\!\:6\!\:\:\!1\!\!\:7\!\:)(\:\!1\!\!\:1\!\:\:\!1\!\!\:5\!\:\:\!1\!\!\:4\!\:)(239\:\!1\!\!\:3\!\:\:\!1\!\!\:8\!\:\:\!1\!\!\:2\!\:\:\!1\!\!\:0\!\: 8674) & 1\!\!\:1\!\: 3 2 1 1
\\ 6 & (123)(567)(9\:\!1\!\!\:0\!\:\:\!1\!\!\:1\!\:)(48\:\!1\!\!\:4\!\:)(\:\!1\!\!\:3\!\:\:\!1\!\!\:5\!\:\:\!1\!\!\:7\!\:)(\:\!1\!\!\:2\!\:\:\!1\!\!\:6\!\:\:\!1\!\!\:8\!\:),
(1)(5)(9)(10\:\!1\!\!\:1\!\:\:\!1\!\!\:8\!\:\:\!1\!\!\:5\!\:\:\!1\!\!\:2\!\:)(23\:\!1\!\!\:4\!\:\:\!1\!\!\:7\!\:\:\!1\!\!\:6\!\:\:\!1\!\!\:3\!\: 8674) & 1\!\!\:0\!\: 5 1 1 1
\\ 7 & (123)(567)(49\:\!1\!\!\:1\!\:)(8\:\!1\!\!\:3\!\:\:\!1\!\!\:0\!\:)(\:\!1\!\!\:2\!\:\:\!1\!\!\:5\!\:\:\!1\!\!\:7\!\:)(\:\!1\!\!\:4\!\:\:\!1\!\!\:8\!\:\:\!1\!\!\:6\!\:),
(1)(5)(\:\!1\!\!\:5\!\:\:\!1\!\!\:8\!\:)(9\:\!1\!\!\:3\!\:\:\!1\!\!\:6\!\:\:\!1\!\!\:2\!\:)(23\:\!1\!\!\:1\!\:\:\!1\!\!\:7\!\:\:\!1\!\!\:4\!\: 867\:\!1\!\!\:0\!\: 9) & 1\!\!\:0\!\: 4 2 1 1
\\ 8 & (123)(567)(49\:\!1\!\!\:1\!\:)(8\:\!1\!\!\:3\!\:\:\!1\!\!\:5\!\:)(\:\!1\!\!\:0\!\:\:\!1\!\!\:6\!\:\:\!1\!\!\:7\!\:)(\:\!1\!\!\:2\!\:\:\!1\!\!\:8\!\:\:\!1\!\!\:4\!\:),
(1)(5)(9\:\!1\!\!\:7\!\:\:\!1\!\!\:2\!\:)(\:\!1\!\!\:3\!\:\:\!1\!\!\:8\!\:\:\!1\!\!\:6\!\:)(23\:\!1\!\!\:1\!\:\:\!1\!\!\:4\!\: 867\:\!1\!\!\:5\!\:\:\!1\!\!\:0\!\: 4) & 1\!\!\:0\!\: 3 3 1 1
\\ 9 & (123)(457)(698)(\:\!1\!\!\:0\!\:\:\!1\!\!\:1\!\:\:\!1\!\!\:3\!\:)(\:\!1\!\!\:2\!\:\:\!1\!\!\:7\!\:\:\!1\!\!\:5\!\:)(\:\!1\!\!\:4\!\:\:\!1\!\!\:6\!\:\:\!1\!\!\:8\!\:),
(1)(58)(\:\!1\!\!\:6\!\:\:\!1\!\!\:7\!\:)(11\:\!1\!\!\:5\!\:\:\!1\!\!\:4\!\:)(2379\:\!1\!\!\:3\!\:\:\!1\!\!\:8\!\:\:\!1\!\!\:2\!\:\:\!1\!\!\:0\!\: 64) & 1\!\!\:0\!\: 3 2 2 1
\\1\!\!\:0 & (123)(567)(9\:\!1\!\!\:0\!\:\:\!1\!\!\:1\!\:)(4\:\!1\!\!\:3\!\:\:\!1\!\!\:8\!\:)(8\:\!1\!\!\:4\!\:\:\!1\!\!\:5\!\:)(\:\!1\!\!\:2\!\:\:\!1\!\!\:7\!\:\:\!1\!\!\:6\!\:),
(1)(5)(9)(67\:\!1\!\!\:5\!\:\:\!1\!\!\:7\!\:\:\!1\!\!\:3\!\: 8)(23\:\!1\!\!\:8\!\:\:\!1\!\!\:2\!\:\:\!1\!\!\:0\!\:\:\!1\!\!\:1\!\:\:\!1\!\!\:6\!\:\:\!1\!\!\:4\!\: 4) & 9 6 1 1 1
\\1\!\!\:1 & (123)(567)(49\:\!1\!\!\:1\!\:)(\:\!1\!\!\:0\!\:\:\!1\!\!\:3\!\:\:\!1\!\!\:2\!\:)(\:\!1\!\!\:4\!\:\:\!1\!\!\:5\!\:\:\!1\!\!\:7\!\:)(8\:\!1\!\!\:6\!\:\:\!1\!\!\:8\!\:),
(1)(5)(9\:\!1\!\!\:2\!\:)(67\:\!1\!\!\:8\!\:\:\!1\!\!\:5\!\: 8)(23\:\!1\!\!\:1\!\:\:\!1\!\!\:3\!\:\:\!1\!\!\:7\!\:\:\!1\!\!\:6\!\:\:\!1\!\!\:4\!\:\:\!1\!\!\:0\!\: 4) & 9 5 2 1 1
\\1\!\!\:2  & (123)(457)(69\:\!1\!\!\:1\!\:)(8\:\!1\!\!\:3\!\:\:\!1\!\!\:5\!\:)(\:\!1\!\!\:0\!\:\:\!1\!\!\:7\!\:\:\!1\!\!\:2\!\:)(\:\!1\!\!\:4\!\:\:\!1\!\!\:8\!\:\:\!1\!\!\:6\!\:),
(1)(9\:\!1\!\!\:2\!\:)(\:\!1\!\!\:3\!\:\:\!1\!\!\:6\!\:)(5\:\!1\!\!\:1\!\:\:\!1\!\!\:7\!\:\:\!1\!\!\:4\!\: 8)(237\:\!1\!\!\:5\!\:\:\!1\!\!\:8\!\:\:\!1\!\!\:0\!\: 64) & 8 5 2 2 1
\\1\!\!\:3 & (123)(457)(69\:\!1\!\!\:1\!\:)(8\:\!1\!\!\:2\!\:\:\!1\!\!\:3\!\:)(\:\!1\!\!\:0\!\:\:\!1\!\!\:5\!\:\:\!1\!\!\:7\!\:)(\:\!1\!\!\:6\!\:\:\!1\!\!\:7\!\:\:\!1\!\!\:8\!\:),
(1)(\:\!1\!\!\:5\!\:\:\!1\!\!\:8\!\:)(5\:\!1\!\!\:1\!\: 8)(9\:\!1\!\!\:7\!\:\:\!1\!\!\:4\!\:\:\!1\!\!\:2\!\:)(237\:\!1\!\!\:3\!\:\:\!1\!\!\:6\!\:\:\!1\!\!\:0\!\: 64) & 8 4 3 2 1
\\1\!\!\:4 & (135)(274)(68\:\!1\!\!\:0\!\:)(9\:\!1\!\!\:1\!\:\:\!1\!\!\:3\!\:)(\:\!1\!\!\:4\!\:\:\!1\!\!\:5\!\:\:\!1\!\!\:7\!\:)(\:\!1\!\!\:2\!\:\:\!1\!\!\:8\!\:\:\!1\!\!\:6\!\:),
(14)(\:\!1\!\!\:5\!\:\:\!1\!\!\:8\!\:)(376)(11\:\!1\!\!\:6\!\:\:\!1\!\!\:4\!\:)(25\:\!1\!\!\:0\!\:\:\!1\!\!\:3\!\:\:\!1\!\!\:7\!\:\:\!1\!\!\:2\!\: 98) & 8 3 3 2 2 
\\1\!\!\:5 & (123)(567)(49\:\!1\!\!\:1\!\:)(8\:\!1\!\!\:0\!\:\:\!1\!\!\:4\!\:)(\:\!1\!\!\:2\!\:\:\!1\!\!\:6\!\:\:\!1\!\!\:8\!\:)(\:\!1\!\!\:3\!\:\:\!1\!\!\:7\!\:\:\!1\!\!\:5\!\:),
(1)(5)(16\:\!1\!\!\:7\!\:)(23\:\!1\!\!\:1\!\:\:\!1\!\!\:8\!\:\:\!1\!\!\:3\!\:\:\!1\!\!\:0\!\: 4)(67\:\!1\!\!\:4\!\:\:\!1\!\!\:5\!\:\:\!1\!\!\:2\!\: 98) & 7 7 2 1 1
\\1\!\!\:6 & (123)(567)(4\:\!1\!\!\:1\!\: 9)(8\:\!1\!\!\:4\!\:\:\!1\!\!\:0\!\:)(\:\!1\!\!\:2\!\:\:\!1\!\!\:6\!\:\:\!1\!\!\:8\!\:)(\:\!1\!\!\:3\!\:\:\!1\!\!\:7\!\:\:\!1\!\!\:5\!\:),
(1)(5)(16\:\!1\!\!\:7\!\:)(239\:\!1\!\!\:4\!\:\:\!1\!\!\:5\!\:\:\!1\!\!\:2\!\: 4)(67\:\!1\!\!\:0\!\:\:\!1\!\!\:1\!\:\:\!1\!\!\:8\!\:\:\!1\!\!\:3\!\: 8) & 7 7 2 1 1
\\1\!\!\:7 & (123)(567)(49\:\!1\!\!\:1\!\:)(8\:\!1\!\!\:3\!\:\:\!1\!\!\:5\!\:)(\:\!1\!\!\:0\!\:\:\!1\!\!\:4\!\:\:\!1\!\!\:7\!\:)(\:\!1\!\!\:2\:\!\:\!1\!\!\:8\!\:\:\!1\!\!\:6\!\:),
(1)(5)(9\:\!1\!\!\:7\!\:\:\!1\!\!\:2\!\:)(67\:\!1\!\!\:5\!\:\:\!1\!\!\:8\!\:\:\!1\!\!\:4\!\: 8)(23\:\!1\!\!\:1\!\:\:\!1\!\!\:6\!\:\:\!1\!\!\:3\!\:\:\!1\!\!\:0\!\: 4) & 7 6 3 1 1
\\1\!\!\:8 & (123)(457)(689)(\:\!1\!\!\:0\!\:\:\!1\!\!\:1\!\:\:\!1\!\!\:3\!\:)(\:\!1\!\!\:2\!\:\:\!1\!\!\:5\!\:\:\!1\!\!\:7\!\:)(\:\!1\!\!\:4\!\:\:\!1\!\!\:8\!\:\:\!1\!\!\:6\!\:),
(1)(\:\!1\!\!\:5\!\:\:\!1\!\!\:8\!\:)(11\:\!1\!\!\:7\!\:\:\!1\!\!\:4\!\:)(23764)(59\:\!1\!\!\:3\!\:\:\!1\!\!\:6\!\:\:\!1\!\!\:2\!\:\:\!1\!\!\:0\!\: 8) & 7 5 3 2 1
\\1\!\!\:9 & (123)(457)(69\:\!1\!\!\:1\!\:)(8\:\!1\!\!\:3\!\:\:\!1\!\!\:5\!\:)(\:\!1\!\!\:2\!\:\:\!1\!\!\:7\!\:\:\!1\!\!\:4\!\:)(\:\!1\!\!\:0\!\:\:\!1\!\!\:6\!\:\:\!1\!\!\:8\!\:),
(1)(9\:\!1\!\!\:8\!\:\:\!1\!\!\:2\!\:)(\:\!1\!\!\:3\!\:\:\!1\!\!\:7\!\:\:\!1\!\!\:6\!\:)(5\:\!1\!\!\:1\!\:\:\!1\!\!\:4\!\: 8)(237\:\!1\!\!\:5\!\:\:\!1\!\!\:0\!\: 64) & 7 4 3 3 1
\\ 20 & (123)(567)(49\:\!1\!\!\:1\!\:)(8\:\!1\!\!\:3\!\:\:\!1\!\!\:5\!\:)(\:\!1\!\!\:0\!\:\:\!1\!\!\:7\!\:\:\!1\!\!\:6\!\:)(\:\!1\!\!\:2\!\:\:\!1\!\!\:4\!\:\:\!1\!\!\:8\!\:),
(1)(5)(9\:\!1\!\!\:6\!\:\:\!1\!\!\:3\!\:\:\!1\!\!\:2\!\:)(23\:\!1\!\!\:1\!\:\:\!1\!\!\:8\!\:\:\!1\!\!\:0\!\: 4)(67\:\!1\!\!\:5\!\:\:\!1\!\!\:7\!\:\:\!1\!\!\:4\!\: 8) & 6 6 4 1 1
\\ 21 & (135)(486)(79\:\!1\!\!\:1\!\:)(\:\!1\!\!\:0\!\:\:\!1\!\!\:4\!\:\:\!1\!\!\:2\!\:)(\:\!1\!\!\:3\!\:\:\!1\!\!\:5\!\:\:\!1\!\!\:7\!\:)(2\:\!1\!\!\:8\!\:\:\!1\!\!\:6\!\:),
(36)(9\:\!1\!\!\:2\!\:)(15\:\!1\!\!\:8\!\:)(1\:\!1\!\!\:6\!\:\:\!1\!\!\:3\!\:\:\!1\!\!\:0\!\: 74)(258\:\!1\!\!\:1\!\:\:\!1\!\!\:4\!\:\:\!1\!\!\:7\!\:) & 6 6 2 2 2
\\ 22 & (123)(567)(49\:\!1\!\!\:1\!\:)(8\:\!1\!\!\:3\!\:\:\!1\!\!\:5\!\:)(\:\!1\!\!\:0\!\:\:\!1\!\!\:2\!\:\:\!1\!\!\:7\!\:)(\:\!1\!\!\:4\!\:\:\!1\!\!\:6\!\:\:\!1\!\!\:8\!\:),
(1)(5)(23\:\!1\!\!\:1\!\:\:\!1\!\!\:0\!\: 4)(67\:\!1\!\!\:5\!\:\:\!1\!\!\:4\!\: 8)(9\:\!1\!\!\:7\!\:\:\!1\!\!\:6\!\:\:\!1\!\!\:3\!\:\:\!1\!\!\:8\!\:\:\!1\!\!\:2\!\:) & 6 5 5 1 1
\\ 23 & (123)(457)(69\:\!1\!\!\:1\!\:)(8\:\!1\!\!\:3\!\:\:\!1\!\!\:0\!\:)(\:\!1\!\!\:2\!\:\:\!1\!\!\:5\!\:\:\!1\!\!\:7\!\:)(\:\!1\!\!\:4\!\:\:\!1\!\!\:8\!\:\:\!1\!\!\:6\!\:),
(1)(\:\!1\!\!\:5\!\:\:\!1\!\!\:8\!\:)(9\:\!1\!\!\:3\!\:\:\!1\!\!\:2\!\:\:\!1\!\!\:6\!\:)(237\:\!1\!\!\:0\!\: 64)(5\:\!1\!\!\:8\!\:\:\!1\!\!\:7\!\:\:\!1\!\!\:4\!\: 8) & 6 5 4 2 1
\\ 24 & (135)(28\:\!1\!\!\:0\!\:)(7\:\!1\!\!\:1\!\: 9)(4\:\!1\!\!\:2\!\:\:\!1\!\!\:4\!\:)(\:\!1\!\!\:3\!\:\:\!1\!\!\:5\!\:\:\!1\!\!\:7\!\:)(6\:\!1\!\!\:8\!\:\:\!1\!\!\:6\!\:),
(89)(\:\!1\!\!\:5\!\:\:\!1\!\!\:8\!\:)(1\:\!1\!\!\:0\:\!\:\!1\!\!\:1\!\: 4)(3\:\!1\!\!\:4\!\:\:\!1\!\!\:7\!\: 6)(25\:\!1\!\!\:6\!\:\:\!1\!\!\:3\!\:\:\!1\!\!\:2\!\: 7) & 6 4 4 2 2
\\ 25 & (135)(274)(69\:\!1\!\!\:1\!\:)(8\:\!1\!\!\:3\!\:\:\!1\!\!\:5\!\:)(\:\!1\!\!\:0\!\:\:\!1\!\!\:6\!\:\:\!1\!\!\:7\!\:)(\:\!1\!\!\:2\!\:\:\!1\!\!\:8\!\:\:\!1\!\!\:4\!\:),
(14)(9\:\!1\!\!\:7\!\:\:\!1\!\!\:2\!\:)(\:\!1\!\!\:3\!\:\:\!1\!\!\:8\!\:\:\!1\!\!\:6\!\:)(25\:\!1\!\!\:1\!\:\:\!1\!\!\:4\!\: 8)(37\:\!1\!\!\:5\!\:\:\!1\!\!\:0\!\: 6) & 5 5 3 3 2
\\ 26 & (135)(279)(4\:\!1\!\!\:1\!\:\:\!1\!\!\:3\!\:)(6\:\!1\!\!\:4\!\:\:\!1\!\!\:5\!\:)(8\:\!1\!\!\:6\!\:\:\!1\!\!\:7\!\:)(\:\!1\!\!\:0\!\:\:\!1\!\!\:8\!\:\:\!1\!\!\:2\!\:),
(3\:\!1\!\!\:3\!\: 6)(7\:\!1\!\!\:7\!\:\:\!1\!\!\:0\!\:)(19\:\!1\!\!\:2\!\: 4)(25\:\!1\!\!\:5\!\: 8)(\:\!1\!\!\:1\!\:\:\!1\!\!\:8\!\:\:\!1\!\!\:6\!\:\:\!1\!\!\:4\!\:) & 4 4 4 3 3
\end{array}
}
\]
\caption{Monodromy data of $\jGamma$}\label{tab:jmonodromy}
\end{table}

Supposing the factorizations in line 15 and 16 are conjugate, 
there is a permutation $\sigma$ which commutes 
with $\mu_1$ and which
conjugates $\mu_0,\mu_\infty$ of line 15 to those of line 16, which for this
proof we denote by $\rho_0$, $\rho_\infty$. From this we deduce the conditions
\[
\sigma( \mu^k_*( n)) \, = \, \rho^k_* ( \sigma (n)) \qquad
\text{for all }n,k \text{ and for } *=0,1,\infty.
\]
In particular we get for $1=\mu_\infty(1)$
\[
\sigma(1) = \sigma(\mu_\infty(1)) \implies \sigma(1) = \rho_\infty(\sigma(1)). 
\]
Thus $\sigma(1)$ is a fixed point of $\rho_\infty$ and thus either $1$ or $5$.
Then for $2=\mu_1(1)$ we get
\[
\sigma(2) = \rho_1(\sigma(1)) \in \{\rho_1(1),\rho_1(5)\} = \{2,6\}.
\]
On the other hand the $2$-cycle $(16\,17)$ is unique in both $\mu_\infty,\rho_\infty$
so
\[
(\sigma(16), \sigma(17)) = (16, 17) \implies \sigma(17) \in \{16,17\}.
\]
Consequently we get for $18=\mu_1(17)$ that
\[
\sigma(18) = \rho_1(\sigma(17)) \in \{\rho_1(16),\rho_1(17)\} = \{15,18\}.
\]
But in this way we arrive at a contradiction since 
\[
\sigma (\mu^4_\infty(2)) = \sigma(18) \in\{15,18\}
\]
while
\[
\rho^4_\infty(\sigma(2)) \in \{\rho^4_\infty(2),\rho^4_\infty(6)\} = \{14,11\}.
\]
\end{proof}

This finally concludes the proof of Theorems \ref{teo:mainresult},  \ref{teo:liststrata2} and  \ref{teo:liststrata}.

%%%
\subsection{Conclusion}

We obtain a list of 48 root lattices with uniquely associated torsion Mordell-Weil groups.
The latter appear in Table \ref{tab:lattices} using the notation of Shimada \cite{shimada} with $[n]$ for $\ZZ/n$
and $[n,m]$ for $\ZZ/n\times \ZZ/m$.
In addition the Table \ref{tab:lattices} 
gives the dimensions of the $50$ ambi-typical strata, the index of $\barGamma$ in $\pslz$,
the cardinality of the kernel of $\Gamma\to \barGamma$
and the corresponding number in
the list of Shimada.       
Although it is not needed in this paper we also state for completeness which of the monodromy groups are congruence  subgroups. This turns out to be the case if and only if the index is not divisible by $9$.  
For index $\leq 6$ this follows from \cite[Theorem 5]{wohlfahrt}. For index $9$ this follows from \cite[Table 2]{CP}, alternatively one can give an independent argument using the amplitudes of the cusps. Finally, for indices $12$ and 
$18$ the claim can be deduced from Sebbar's classification \cite[Table 1]{Se}. 

In Table \ref{tab:jmaps} we list the branch behaviour of the maps 
$\jGamma$ and  $j_\calE$.

\newpage

\begin{table}[htbp]
{\scriptsize
\begin{tabular}{|c|c|c|c|c|c|cH|c|} \hline
\# & root lattice & MW & $\dim$ & ind & $|\ker|$ & [Shi] \# & [Shi] MW & Remarks \\
\hline
0 & & $[1]$ & 18 & 1 & 2 & & & $\Gamma=\slz$ \\ \hline
1&  $8A_1$ & $[2]$ & 10 & 3 & 2 & 99 & [1],[2] & $\Gamma=\Gamma_1(2)$  \\ \hline
2&  $6A_2$ & $[3]$ & 6 &  4 & 1 & 559 & [1],[3] & $\Gamma=\Gamma_1(3)$ \\ \hline
3&  $4\,A_{3}$ & $[2] $ & 6 & 6 & 2 & 547 & [1],[2] & $\Gamma=\Gamma_0(4)$  \\  \hline 
4&  $12\,A_{1}$ & $[2, 2]$ & 6 & 6 & 2 & 565 & [2,2] & $\Gamma=\Gamma(2)$  \\  \hline 
5& $2\,A_7$ & $[2]$ & 4 & 6 & 1 & 1134 & [1],[2] &
\raisebox{-1pt}{$\barGamma=\barGamma_0$}$(4)$,
%\footnotemark\addtocounter{footnote}{-1}
{$-\id\not\in \Gamma$}
\\  \hline 
6& $2\,A_3\,+\,8\,A_{1}$ & $[2, 2]$ & 4 & 6 & 1 & 1223 & [2,2] &
$\raisebox{-1pt}{$\barGamma=\barGamma$}(2)$,
%\footnotemark
{$-\id\not\in \Gamma$}
\\  \hline 
7& $4\,A_3\,+\,2\,A_{1}$ & $[4]$ & 4 & 6 & 1 & 1215 & [4],[2],[1] & $\Gamma=\Gamma_1(4)$  \\  \hline 

8&  $D_4+2A_6$ & $[1]$ & 2 &$9$ & 2 & 2171 & [1] & \text{non-congruence}\\ \hline
9&  $D_4+2A_5+2A_1$ & $[2]$ & 2 &$9$ & 2 & 2179 & [2]  &  " \\ \hline
10&  $D_4+2A_4+2A_2$ & $[1]$ & 2 &$9$ & 2 & 2190 & [1]  & " \\ \hline
11&  $D_4+2A_3+2A_2+2A_1$ & $[2]$ & 2 &$9$ & 2 & 2198 & [2] &  " \\ \hline

12&  $2\,D_{4}\,+\,A_{8}$ & $[1]$ & 2 &$12$ & 2 & 2148 & $[1]$  & \text{twisted Beauville} \\  \hline 
13&  $2\,D_{4}\,+\,A_{7}\,+\,A_{1}$ & $[2]$ & 2 &$12$ & 2 & 2149 & $[2]$ &  " \\  \hline 
14&  $2\,D_{4}\,+\,A_{5}\,+\,A_{2}\,+\,A_{1}$ & $[2]$ & 2 &$12$ & 2 & 2150 & $[2]$ & " \\  \hline 
15&  $2\,D_{4}\,+\,2\,A_{4}$ & $[1]$ & 2 &$12$ & 2 & 2151 & $[1]$ & "  \\  \hline 
16&  $2\,D_{4}\,+\,2\,A_{3}\,+\,2\,A_{1}$ & $[2, 2]$ & 2 &$12$ & 2 & 2152 & $[2, 2]$  & " \\  \hline 
17&  $2\,D_{4}\,+\,4\,A_{2}$ & $[1]$ & 2 &$12$ & 2 & 2153 & $[1]$ & "\\  \hline 

18&  $2\,A_{8}$ & $[3]$ & 2 &$12$ & 1 & 2242 & $[1],[3]$  & 2:1\text{ Beauville} \\  \hline 
19&  $2\,A_{7}\,+\,2\,A_{1}$ & $[4]$ & 2  &$12$ & 1 & 2262 & $[1],[2],[4]$  & " \\  \hline 
20&  $2\,A_{5}\,+\,2\,A_{2}\,+\,2\,A_{1}$ & $[6]$ & 2  &$12$ & 1 & 2322 & $[1],[2],[3],[6]$  & "\\  \hline 
21&  $4\,A_{4}$ & $[5]$ & 2  &$12$ & 1 & 2345 & $[1],[5]$  & " \\  \hline 
22&  $4\,A_{3}\,+\,4\,A_{1}$ & $[4, 2]$ & 2  &$12$ & 1 & 2368 & $[2,2],[4,2]$  & " \\  \hline 
23&  $8\,A_{2}$ & $[3, 3]$ & 2  &$12$ & 1 & 2373 & $[3,3]$  & " \\  \hline 

24&  $D_{4}\,+\,A_{13}$ & $[1]$ & 1 &$18$ & 2 & 2762 & $[1]$  & \text{non-congruence} \\  \hline 
25&  $D_{4}\,+\,A_{12}\,+\,A_{1}$ & $[1]$ & 1 &$18$ & 2 & 2763 & $[1]$  & " \\  \hline 
26&  $D_{4}\,+\,A_{11}\,+\,A_{2}$ & $[2]$ & 1 &$18$ & 2 & 2764 & $[2]$  & " \\  \hline 
27&  $D_{4}\,+\,A_{11}\,+\,2\,A_{1}$ & $[2]$ & 1 & $18$ & 2 & 2765& $[2]$  & "\\  \hline 
28&  $D_{4}\,+\,A_{10}\,+\,A_{2}\,+\,A_{1}$ & $[1]$ & 1 &$18$ & 2 & 2766 & $[1]$  & " \\  \hline 
29&  $D_{4}\,+\,A_{9}\,+\,A_{4}$ & $[1]$ & 1 &$18$ & 2 & 2767 & $[1]$  & " \\  \hline 
30&  $D_{4}\,+\,A_{9}\,+\,A_{3}\,+\,A_{1}$ & $[2]$ & 1 &$18$ & 2 & 2768 & $[2]$  & " \\  \hline 
31&  $D_{4}\,+\,A_{9}\,+\,2\,A_{2}$ & $[1]$ & 1 &$18$ & 2 & 2769 & $[1]$  & " \\  \hline 
32&  $D_{4}\,+\,A_{9}\,+\,A_{2}\,+\,2\,A_{1}$ & $[2]$ & 1 &$18$ & 2 & 2770 & $[2]$  & " \\  \hline 
33&  $D_{4}\,+\,A_{8}\,+\,A_{5}$ & $[1]$ & 1 &$18$ & 2 & 2771 & $[1]$  & " \\  \hline 
34&  $D_{4}\,+\,A_{8}\,+\,A_{4}\,+\,A_{1}$ & $[1]$ & 1 &$18$ & 2 & 2772 & $[1]$  & " \\  \hline 
35&  $D_{4}\,+\,A_{7}\,+\,A_{4}\,+\,2\,A_{1}$ & $[2]$ & 1 &$18$ & 2 & 2773 & $[2]$  & " \\  \hline 
36&  $D_{4}\,+\,A_{7}\,+\,A_{3}\,+\,A_{2}\,+\,A_{1}$ & $[2]$ & 1 & $18$ & 2 & 2774 & $[2]$  & " \\  \hline 
37&  $D_{4}\,+\,A_{7}\,+\,2\,A_{2}\,+\,2\,A_{1}$ & $[2]$ & 1 &$18$ & 2 & 2775 & $[2]$  & " \\  \hline 
38&  $D_{4}\,+\,2\,A_{6}\,+\,A_{1}$ & $[1]$ & 1 &$18$ & 2 & 2776 & $[1]$  & "  \\  \hline 
39&  $D_{4}\,+\,2\,A_{6}\,+\,A_{1}$ & $[1]$ & 1 &$18$ & 2 & 2776 & $[1]$  & "  \\  \hline 
40&  $D_{4}\,+\,A_{6}\,+\,A_{5}\,+\,A_{2}$ & $[1]$ & 1 &$18$ & 2 & 2777 & $[1]$  & " \\  \hline 
41&  $D_{4}\,+\,A_{6}\,+\,A_{4}\,+\,A_{2}\,+\,A_{1}$ & $[1]$ & 1 &$18$ & 2 & 2778 & $[1]$  & " \\  \hline 
42&  $D_{4}\,+\,A_{6}\,+\,A_{3}\,+\,2\,A_{2}$ & $[1]$ & 1 &$18$ & 2 & 2779 & $[1]$  & " \\  \hline 
43&  $D_{4}\,+\,2\,A_{5}\,+\,A_{3}$ & $[2]$ & 1 &$18$ & 2 & 2780 & $[2]$  & " \\  \hline 
44&  $D_{4}\,+\,2\,A_{5}\,+\,3\,A_{1}$ & $[2, 2]$ & 1 &$18$ & 2 & 2781 & $[2,2]$  & " \\  \hline 
45&  $D_{4}\,+\,A_{5}\,+\,2\,A_{4}$ & $[1]$ & 1 &$18$ & 2 & 2782 & $[1]$  & " \\  \hline 
46&  $D_{4}\,+\,A_{5}\,+\,A_{4}\,+\,A_{3}\,+\,A_{1}$ & $[2]$ & 1 &$18$ & 2 & 2783 & $[2]$  & " \\  \hline 
47&  $D_{4}\,+\,A_{5}\,+\,2\,A_{3}\,+\,2\,A_{1}$ & $[2, 2]$ & 1 &$18$ & 2 & 2784 & $[2,2]$  & " \\  \hline 
48&  $D_{4}\,+\,2\,A_{4}\,+\,2\,A_{2}\,+\,A_{1}$ & $[1]$ & 1 &$18$ & 2 & 2785 & $[1]$  & " \\  \hline 
49&  $D_{4}\,+\,3\,A_{3}\,+\,2\,A_{2}$ & $[2]$ & 1 &$18$ & 2 & 2786 & $[2]$  & " \\  \hline 
\end{tabular}
}
\caption{Data of ambi-typical strata}\label{tab:lattices}\label{table:list}
\end{table}
%\footnotetext{$-\id\not\in \Gamma$}

\begin{table}[htbp]
{\scriptsize
\begin{tabular}{|c|l|c|c|c|c|c|c|}
\hline
\# & \hspace*{3cm}$\jGamma$ & \raisebox{-1pt}{$\barGamma$} & $j_\calE$ & $I_0^*$ fibres \\ \hline
0 & $1$ & $\pslz$ & $(3^8)_A$, $(2^{12})_B$, $2^{18}$ & - \\ \hline
1 & $(3), (2,1), (2,1)$ & \raisebox{-1pt}{$\barGamma_1$}$(2)$ & $(2,2,2,2)_B$, $2^{10}$ & - \\ \hline
2 & $(3,1), (2,2), (3,1)$ & \raisebox{-1pt}{$\barGamma_1$}$(3)$ & $(3,3)_A$, $2^6$ & - \\ \hline
3 & $(3,3), (2,2,2), (4,1,1)$ & \raisebox{-1pt}{$\barGamma_1$}$(4)$ & $2^6$ & - \\ \hline
4 & $(3,3), (2,2,2), (2,2,2)$ & \raisebox{-1pt}{$\barGamma$}$(2)$ & $2^6$ & - \\ \hline
5 & $(3,3), (2,2,2), (4,1,1)$ & \raisebox{-1pt}{$\barGamma_1$}$(4)$ & $(2,2)_{4_\infty}$, $2^4$ & - \\ \hline
6 & $(3,3), (2,2,2), (2,2,2)$ & \raisebox{-1pt}{$\barGamma$}$(2)$ & $(2,2)_{2_\infty}$, $2^6$ & - \\ \hline
7 & $(3,3), (2,2,2), (4,1,1)$ & \raisebox{-1pt}{$\barGamma_1$}$(4)$ & $(2,2)_{1_\infty}$, $2^6$ & - \\ \hline
8 & $(3,3,3), (2,2,2,2,1), (7,1,1)$ & & $(2)_B$, $2$ & 1 \\ \hline
9 & $(3,3,3), (2,2,2,2,1), (6,2,1)$ & & $(2)_B$, $2$ & 1 \\ \hline
10 & $(3,3,3), (2,2,2,2,1), (5,3,1)$ & & $(2)_B$, $2$ & 1 \\ \hline
11 & $(3,3,3), (2,2,2,2,1), (4,3,2)$ & & $(2)_B$, $2$ & 1 \\ \hline
12 & $(3,3,3,3), (2,2,2,2,2,2), (9,1,1,1)$ & & $1$ & 2 \\ \hline
13 & $(3,3,3,3), (2,2,2,2,2,2), (8,2,1,1)$ & & $1$ & 2 \\ \hline
14 & $(3,3,3,3), (2,2,2,2,2,2), (6,3,2,1)$ & & $1$ & 2 \\ \hline
15 & $(3,3,3,3), (2,2,2,2,2,2), (5,5,1,1)$ & & $1$ & 2 \\ \hline
16 & $(3,3,3,3), (2,2,2,2,2,2), (4,4,2,2)$ & & $1$ & 2 \\ \hline
17 & $(3,3,3,3), (2,2,2,2,2,2), (3,3,3,3)$ & & $1$ & 2 \\ \hline
18 & $(3,3,3,3), (2,2,2,2,2,2), (9,1,1,1)$ & & $2^2$ & - \\ \hline
19 & $(3,3,3,3), (2,2,2,2,2,2), (8,2,1,1)$ & & $2^2$ & - \\ \hline
20 & $(3,3,3,3), (2,2,2,2,2,2), (6,3,2,1)$ & & $2^2$ & - \\ \hline
21 & $(3,3,3,3), (2,2,2,2,2,2), (5,5,1,1)$ & & $2^2$ & - \\ \hline
22 & $(3,3,3,3), (2,2,2,2,2,2), (4,4,2,2)$ & & $2^2$ & - \\ \hline
23 & $(3,3,3,3), (2,2,2,2,2,2), (3,3,3,3)$ & & $2^2$ & - \\ \hline
24 & $(3,3,3,3,3,3), (2,2,2,2,2,2,2,2,2), (14,1,1,1,1)$ & & $1$ & 1  \\ \hline
25 & $(3,3,3,3,3,3), (2,2,2,2,2,2,2,2,2), (13,2,1,1,1)$ & & $1$ & 1  \\ \hline
26 & $(3,3,3,3,3,3), (2,2,2,2,2,2,2,2,2), (12,3,1,1,1)$ & & $1$ & 1  \\ \hline
27 & $(3,3,3,3,3,3), (2,2,2,2,2,2,2,2,2), (12,2,2,1,1)$ & & $1$ & 1  \\ \hline
28 & $(3,3,3,3,3,3), (2,2,2,2,2,2,2,2,2), (11,3,2,1,1)$ & & $1$ & 1  \\ \hline
29 & $(3,3,3,3,3,3), (2,2,2,2,2,2,2,2,2), (10,5,1,1,1)$ & & $1$ & 1  \\ \hline
30 & $(3,3,3,3,3,3), (2,2,2,2,2,2,2,2,2), (10,4,2,1,1)$ & & $1$ & 1  \\ \hline
31 & $(3,3,3,3,3,3), (2,2,2,2,2,2,2,2,2), (10,3,3,1,1)$ & & $1$ & 1  \\ \hline
32 & $(3,3,3,3,3,3), (2,2,2,2,2,2,2,2,2), (10,3,2,2,1)$ & & $1$ & 1  \\ \hline
33 & $(3,3,3,3,3,3), (2,2,2,2,2,2,2,2,2), (9,6,1,1,1)$ & & $1$ & 1  \\ \hline
34 & $(3,3,3,3,3,3), (2,2,2,2,2,2,2,2,2), (9,5,2,1,1)$ & & $1$ & 1  \\ \hline
35 & $(3,3,3,3,3,3), (2,2,2,2,2,2,2,2,2), (8,5,2,2,1)$ & & $1$ & 1  \\ \hline
36 & $(3,3,3,3,3,3), (2,2,2,2,2,2,2,2,2), (8,4,3,2,1)$ & & $1$ & 1  \\ \hline
37 & $(3,3,3,3,3,3), (2,2,2,2,2,2,2,2,2), (8,3,3,2,2)$ & & $1$ & 1  \\ \hline
38 & $(3,3,3,3,3,3), (2,2,2,2,2,2,2,2,2), (7,7,2,1,1)$ & & $1$ & 1  \\ \hline
39 & $(3,3,3,3,3,3), (2,2,2,2,2,2,2,2,2), (7,7,2,1,1)$ & & $1$ & 1  \\ \hline
40 & $(3,3,3,3,3,3), (2,2,2,2,2,2,2,2,2), (7,6,3,1,1)$ & & $1$ & 1  \\ \hline
41 & $(3,3,3,3,3,3), (2,2,2,2,2,2,2,2,2), (7,5,3,2,1)$ & & $1$ & 1  \\ \hline
42 & $(3,3,3,3,3,3), (2,2,2,2,2,2,2,2,2), (7,4,3,3,1)$ & & $1$ & 1  \\ \hline
43 & $(3,3,3,3,3,3), (2,2,2,2,2,2,2,2,2), (6,6,4,1,1)$ & & $1$ & 1  \\ \hline
44 & $(3,3,3,3,3,3), (2,2,2,2,2,2,2,2,2), (6,6,2,2,2)$ & & $1$ & 1  \\ \hline
45 & $(3,3,3,3,3,3), (2,2,2,2,2,2,2,2,2), (6,5,5,1,1)$ & & $1$ & 1  \\ \hline
46 & $(3,3,3,3,3,3), (2,2,2,2,2,2,2,2,2), (6,5,4,2,1)$ & & $1$ & 1  \\ \hline
47 & $(3,3,3,3,3,3), (2,2,2,2,2,2,2,2,2), (6,4,4,2,2)$ & & $1$ & 1  \\ \hline
48 & $(3,3,3,3,3,3), (2,2,2,2,2,2,2,2,2), (5,5,3,3,2)$ & & $1$ & 1  \\ \hline
49 & $(3,3,3,3,3,3), (2,2,2,2,2,2,2,2,2), (4,4,4,3,3)$ & & $1$ & 1  \\ \hline
\end{tabular}
}\caption{$j$-factorisation data of ambi-typical strata}\label{tab:jmaps}
\end{table}

%%%%%%%%%%%%%%%%%%%%%%%%%%%%%%%%%%%%%%%%%
%%%%%%%%%%%%%%%%%%%%%%%%%%%%%%%%%%%%%%%%%

\section{Moduli spaces of lattice polarised $K3$ surfaces and Shimada strata}\label{sec:mirrorpol}

In this section we want to start our discussion about the precise relationship between moduli spaces of lattice polarised $K3$ surfaces
and ambi-typical strata.
For this  we first recall some basic facts about lattice polarised $K3$ surfaces and their  moduli spaces. 
Our aim is to understand how many moduli spaces of lattice polarised $K3$ surfaces dominate a given ambi-typical stratum and to compute the degree of the finite map
from a component of such a moduli space to the ambi-typical stratum it dominates. 

In Section \ref{sec:ellfibered} we encountered the following situation: we started with an elliptically fibred $K3$ surface $f: \calE \to \PP^1$ and root lattice $R(\calE)$ whose saturation in the $K3$ lattice 
was denoted by $L(\calE)$. Then $M(\calE)=U+ L(\calE)$ is a hyperbolic lattice contained in the N\'eron-Severi group of $\calE$. This gives rise to a lattice polarization.

To explain this in more detail, let   $M$ be an even lattice of signature $(1,t)$ which admits a primitive embedding 
$\iota: M \to L_{K3}$ into the $K3$ lattice $L_{K3}$.  
By $T=T(\iota)=\iota(M)^{\perp}_{L_{K3}} \subset L_{K3}$ we denote the orthogonal complement of the image of $\iota$.
The lattice $T$ has signature $(2,19- t)$.
We shall for the rest of this section assume that the rank of $T$ is at least $3$ (which is the case in our situation).
Then $T$  defines a type IV homogeneous domain 
\begin{equation*} 
\Omega_{T}=\{[x] \in \PP(T \otimes \CC) \mid (x,x)=0, (x,\bar x) >0\} = \calD_{T} \cup  \calD'_{T}
\end{equation*}
of dimension $19 - t$ consisting 
of two connected components, namely $ \calD_{T}$ and  $ \calD'_{T}$. 

We also consider the real cone
\begin{equation*}
C_M=\{x \in M_{\RR} \mid (x,x)>0 \}=C_M^+ \cup C_M^-
\end{equation*}
which again consists of two connected components, of which we choose one, say $C_M^+$. Removing from $C_M^+ $ all hyperplanes orthogonal to
roots $\Delta_M=\{d \in M \mid d^2=-2\}$ subdivides $C_M^+$ into different connected components, the so called {\em Weyl chambers}, of which we choose one and call it  $C_M^{\operatorname {pol}}$.
An {\em $(M,\iota)$-polarised} $K3$ surface is a pair $(S,\tilde{\iota})$ where $\tilde{\iota}:M \to \NS(S) \subset H^2(S,\ZZ)$ is a primitive embedding which is isomorphic to $\iota$ with respect to a suitable marking $\varphi: H^2(S,\ZZ) \to L_{K3}$ and 
such that $\tilde{\iota}(C_M^{\operatorname {pol}})$ contains an ample class. 
In this case we call $\varphi: H^2(S,\ZZ) \to L_{K3}$  an {\em $M$-polarised marking}. 
We call two $M$-polarised $K3$ surfaces $(S_1,\tilde{\iota}_1)$ and $(S_2,\tilde{\iota}_2)$ isomorphic if there is an isomorphims $f: S_1 \to S_2$ with $f^* \circ \tilde{\iota}_2=\tilde{\iota}_1$.  

In order to describe the moduli space of $M$-polarised $K3$ surfaces we consider the group
\begin{equation*}
\O(L_{K3},M,\iota):=\{g \in O(L_{K3}) \mid g|_{\iota(M)}= id_{\iota(M)}\}
\end{equation*}
and recall the definition of the {\em { stable }} orthogonal group
\begin{equation*}
\tO(T):=\{g \in O(T) \mid g|_{D(T)}= id_{D(T)}\}
\end{equation*}
consisting of all orthogonal transformations of $T$ which act trivially on the discriminant$D(T)$. 
It is well known, see \cite[Corollary 1.5.2]{Nik}, that there is an isomorphism 
\begin{equation*}
\O(L_{K3},M,\iota) \cong \tO(T).
\end{equation*}
The group $O(T)$, and hence also $\tO(T)$,  acts properly discontinuously on $\Omega_{T}$ as well as on the open subset
\begin{equation*}
\Omega_{T} ^{\operatorname {pol}}=  \Omega_{T} \setminus \bigcup_{d \in T, d^2=-2} (H_d \cap \Omega_{T})
\end{equation*}
where $H_d=\langle d \rangle ^{\perp} \subset \PP(T\otimes {\CC})$ is the hyperplane orthogonal to $d$.
\begin{pro}\label{prop:latticepol}
The quotient 
\begin{equation*}
\calN^a_{M,\iota}= \Omega_{T} ^{\operatorname {pol}} / \tO(T)
\end{equation*}
is the moduli space of $(M,\iota)$-polarised $K3$ surfaces.
\end{pro} 
\begin{proof}
See \cite[Section 3]{Do} or \cite[p. 360]{BHPV}
\end{proof}

It is sometimes also useful to consider a weakening of $(M,\iota)$-polarizations. Recall that a line bundle $\calL$ is said to be a {\em quasi-polarization} if it is nef and big. 
We say that $(S,\tilde \iota)$ is an $(M,\iota)$-quasi-polarised $K3$ surface if $\tilde{\iota}(C_M^{\operatorname {pol}})$ contains a big and nef class.
By
\cite[Section 3]{Do} the quotient
\begin{equation*}
\calN_{M,\iota}= \Omega_{T} / \tO(T). 
\end{equation*}
is in $1:1$ correspondence with the set of isomorphism classes of $M$-quasi-polarised $K3$ surfaces. 
It can be viewed as the moduli space of $M$-polarized $K3$ surfaces with ADE singularities.
This contains  $\calN^a_{M,\iota}$ as an open subset. 

At this point some remarks are in order. 
In the literature it is often tacitly assumed that the lattice $M$ has a unique primitive embedding into the $K3$ lattice. This assumption then justifies to talk about {\em the} moduli space of $M$-polarised $K3$ surfaces.
For us it will be important  to also allow the possibility that $M$ possesses different primitive embeddings into the $K3$ lattice (the number of such embeddings modulo $\O(L_{K3})$ is always finite). We will then consider the union 
\begin{equation*}
\calN_{M} = \cup_{\iota} \calN_{M,\iota}
\end{equation*}
 where $\iota$ runs over all classes of different primitive embeddings of $M$ into $L_{K3}$. 
Moreover, Dolgachev has formulated arithmetic conditions which lead to the notion of {\it $m$-admissible} lattices. This means in particular that the transcendental lattice $T$ splits off a summand $U(m)$, i.e. a multiple of a hyperbolic plane.
This is relevant for mirror symmetry and the discussion of the Yukawa-coupling, but it plays no role for our purposes since
the construction of the moduli space of lattice (quasi-)polarised $K3$ surfaces does not require this condition. In fact, a number of the lattices which we consider, in particular when $M$ has rank $19$, do not fulfill this condition,
see the Appendix. In these cases the lattice $T$ does not split off a summand $U$ over $\mathbb Q$ resulting in compact moduli spaces of $M$-polarised $K3$ surfaces. 
Finally, we recall that the case $M=U$ gives us another construction for the
moduli space $\calF$ of elliptically fibred $K3$ sufaces with a section or Jacobian fibrations.  
 
Since $\O(T)$ acts properly discontinuously on $\Omega_T$ the quotient  $\calN_{M,\iota}$ has at most finite quotient singularities. 
By a well know result of Baily-Borel this is a quasi-projective variety.
We also note that for small rank of the transcendental lattice $T$ there is a relation with Siegel spaces:
if $t=18$ then  $\calD_{T}\cong \HH_1$ is the upper half plane, if $t=17$, then   $\calD_{T}\cong \HH_1 \times \HH_1 $, and if
$t=16$, then $\calD_{T}\cong \HH_2$, the Siegel upper half plane of genus $2$. We will discuss this in more detail in the case of $t=18$ in  the Appendix.

The  quotient $\Omega_T / \tO(T) $ can have one or two components. If it has two components then these are complex conjugate to each other.
An element $g \in \O(T)$ can either fix  the two components $\calD_{T}$ and $\calD'_{T}$ or interchange them, depending on its spinor norm. We recall that the {\em real spinor norm} is a homomorphism
\begin{equation*}
\operatorname{sn}_{\RR}: \O(T) \to \RR^*/(\RR^*)^2 = \{\pm 1\}.
\end{equation*}
For a precise definition we refer the reader to \cite[Section 1]{GHS}. 
Our normalization of the spinor norm  is such that in the case of signature $(2,n)$, in which we here are,  the transformation $g$ fixes the  two components of 
$\Omega_T$ if and only if $\operatorname{sn}_{\RR}(g)=1$ and it interchanges them if  and only if $\operatorname{sn}_{\RR}(g)=-1$.
We define the groups
\begin{equation*}
\O^+(T)= \{g \in \O(T) \mid \operatorname{sn}_{\RR}(g)=1 \}
\end{equation*}
and 
\begin{equation*}
\tO^+(T)= \O^+(T) \cap \tO(T).
\end{equation*}
The quotient $\Omega_T /\tO(T)$ then has two components if any only if $\tO(T)=\tO^+(T)$.

We first want to discuss the number of moduli spaces of $M$-polarised $K3$ surfaces and the number of connected components. 
Clearly, there is only one such moduli space if there is a unique primitive embedding of $\iota: M \to L_{K3}$ (up to $\O(L_{K3})$). In general, however, such an embedding need not be unique. 
Assume that there is at least one primitive embedding $\iota: M \to L_{K3}$ and let $T_{\iota}=\iota(M)^{\perp}_{L_{K3}}$. Then $T_{\iota}$ may depend on $\iota$, but its genus does not. It is determined by
\[
\sign(T_{\iota})=(2,18 - \rank(M)) \mbox{ and } (D(T_i),  q_{T_i}) \cong (D(M), - q_M).  
\]    
We call this the {\em genus orthogonal to $M$} and denote it by $\calG_T$.   
\begin{pro}
Let $T\in \calG_T$ and let $\overline{{\O}}(T)$ be the image of $\O(T)$ in $\O(D(T))$. Then the index $[\O(D(T)): \overline{\O}(T)]$ depends only on the genus $\calG_T$ of $T$ and not on $T$ itself.
\end{pro}
\begin{proof}
This follows from the theory of Miranda and Morrison, in particular \cite[Chapter VIII, Proposition 6.1. (2)]{MiMo}. Here we recall that we are always in the situation that $T$ is indefinite since we assume in this section 
that it  has rank at least $3$. 
\end{proof}

It follows from the theory of Miranda-Morrison \cite[Theorem VIII.7.2]{MiMo}, see  also \cite[p. 514]{shimada2}, 
that there is an exact sequence, which is completely determined by $M$, of the form 
\begin{equation} \label{seq:MirandaMorrison}
0 \to \coker\big(\!\O(T) \to \O(D(T))\big) \to \calM_T \to \calG_T \to 0
\end{equation}
where $\calM_T$ is a finite group which is in $1:1$ correspondence with the primitive embeddings $\iota: M \to L_{K3}$ and thus with the moduli spaces of lattice polarised $K3$ surfaces with lattice polarization $M$. 

The following result is far less obvious. We will not need it for the lattices we are concerned with, as in our cases the genus always consists of one element only, but state it to complete the picture.
The proof of this was communicated to us by Simon Brandhorst, here we only give a sketch.

\begin{pro}\label{pro:independenceindex}
The index $[\tO(T): \tO^+(T)]$ depends only on the genus $\calG_T$ of $T$ and not on $T$ itself.
\end{pro}
\begin{proof} (Sketch) This is a consequence of the strong approximation theorem. It  can be deduced from \cite[Proposition VIII.6.1 (2)]{MiMo} 
with extra bookkeeping of the real spinor norm 
using the fact that (in the terminology of 
\cite{MiMo}) the objects $\Gamma_S$, $\Sigma(T)$ and $\Sigma^{\sharp}(T)$ all depend only on the genus $\calG_T$ and not on the lattice $T$ itself.  
\end{proof}

As a corollary of the theory of Miranda and Morrison and in particular Sequence \eqref{seq:MirandaMorrison} together with Proposition \ref{pro:independenceindex} we thus obtain
\begin{cor}\label{cor:countingcomponets} 
Let $M$ be a hyperbolic lattice which admits a primitive embedding into the $K3$-lattice $L_{K3}$ and let $\calG_T$ be the genus of 
the orthogonal complement of one, and hence any such embedding of $M$ into $L_{K3}$.
\begin{itemize}
\item[\rm{(1)}]The number of moduli spaces of $M$-polarised $K3$ surfaces is given by 
\[
| \calM_T |=  [\O(D(T)):\overline{\O}(T)] \cdot |\calG_T|.
\]
\item[\rm{(2)}]The number of connected components of these moduli spaces is given by 
\[
| \calM_T |^c=  [\O(D(T)):\overline{\O}(T)] \cdot |\calG_T|\cdot \frac{2}{[\tO(T):\tO^+(T)]} .
\]
\end{itemize}
\end{cor}

We shall now discuss which values these numbers can have in our cases.
The relevant data for the ambi-typical strata are collected in Table \ref{table:list}. The lattice data consist first of a root lattice $R$ and an isotropic subgroup 
$G\subset D(R)$ in the discriminant group of $R$. This determines an overlattice $L$ of $R$ and the hyperbolic lattice $M=U+L$.

We shall first prove that in our situation the genus $\calG_T$ always consists of a single element.

\begin{pro} \label{prop:uniquegenus}
Let $M$ be a hyperbolic lattice associated to one of the families listed in Table \ref{table:list}.
Then the genus given by the signature $(2,20 - \rank(M))$ and discriminant group $(D(M),-q_M)$
consists of  one element only.  
\end{pro}
\begin{proof}
The claim for the entries (1),(2) and (3) of Table  \ref{table:list} follows immediately from the numerical conditions of Nikulin's theorem \cite[Theorem 1.14.2]{Nik}.

For the other lattices we shall use \cite[Chapter 15]{CS}. Since we have indefinite lattices, it suffices, according to \cite[Section 15.9.7]{CS}, to show that there are no non-tractable primes
(for a definition of non-tractable primes see \cite[Chapter 15.9.6]{CS}). 
Let $d$ be the discriminant of $M$ and let $n=22 - \operatorname{rank}(M)$ be the rank of the orthogonal complement. 
According to \cite[Theorem 15.20]{CS} a necessary condition for an odd prime $p$ to be non-tractable is  
$$
d \mbox{ is divisible by } p^{\binom{n}{2}}.
$$
Note that $n=8$ in the case of entry (4), i.e. Shimada's case 565, and $n=4$ or $n=3$ in all other cases.
One can now check by hand that this condition is never fulfilled for the lattices coming from Table  \ref{table:list}. In most cases this  follows already from the discriminants of the root lattices
$R$  in the
second column of this table. However, in some cases one has to be more careful. An example is entry (23) which is Shimada's family 2373. Here $n=4$ and we must not have divisors of $d$ of the form $p^6$.
Now the lattice $8A_2$ has discriminant $3^8$. However, the lattice $L$ is an overlattice of $8A_2(-1) + U$ with torsion group $[3,3]$. But  this means that the order
of the discriminant group of $M$ is $3^8/3^4=3^4$ and hence $3$ is not a non-tractable prime. The other cases can be treated in the same way.

Thus the only possible non-tractable prime is $p=2$.  By \cite[Theorem 15.20]{CS} this implies that  
$$
4^{[\frac{n}{2}]}d \mbox{ is divisible by } 8^{\binom{n}{2}}.
$$
Again, this can be checked by hand. For example in case 2784 one has 
$4^{[\frac{n}{2}]}d(D_4+A_5+2A_3+2A_1 + U)=2^{11}\cdot 3$. However, taking the torsion into account we obtain that $4^{[\frac{n}{2}]}d=2^7\cdot 3$ which is not 
divisible by $8^3=2^9$.  The other cases are similar. 
\end{proof}

\begin{cor}
Let $M$ be a hyperbolic lattice associated to one of the families listed in Table \ref{table:list}. Then the orthogonal complement of a primitive embedding 
$\iota: M \to L_{K3}$ depends only on $M$ and not on the chosen embedding $\iota$. 
\end{cor}

This corollary allows us to speak of {\em the} orthogonal complement $T$ of the lattice $M$  in the $K3$ lattice $L_{K3}$, even if the primitive embedding $\iota: M \to L_{K3}$ is not uniquely defined (which can occur). 

\begin{pro} \label{prop:embeddings}
Let $M$ be a hyperbolic lattice associated to one of the families listed in Table \ref{table:list} and let  $T$ be the unique element in the genus orthogonal to $M$.
Then the map $\O(T) \to \O(D(T))$ is always surjective with the exception of the following five  root lattices $R$:
\begin{itemize}
\item[(17)] $2D_4 + 4A_2$,
\item[(21)] $4A_4$,
\item[(37)] $D_4 + A_7 + 2A_2 + 2A_1$,
\item[(48)] $D_4 + 2A_4 + 2A_2 + A_1$,
\item[(49)] $D_4 + 3A_3 + 2A_2$.
\end{itemize}
In these cases the index $[\O(D(T)): \overline{\O}(T)]=2$, in particular, there are exactly two non-isomorphic primitive embeddings of $M$ into  $L_{K3}$.
\end{pro}

\begin{proof}
For  the cases (1),(2) and (3) from Table \ref{table:list} the surjectivity of $\O(T) \to \O(D(T))$ can be seen directly by Nikulin's ciriterion \cite[Theorem 1.16.10]{Nik}.
In general this follows from computations of Shimada, which are available from his website, see \cite{shimada4}. 
For the rank $17$ cases this follows independently from Kirschmer's computations, see Table \ref{table:OPlus}, Column 5 in the Appendix. 
\end{proof}

It now follows that the number of connected components of $M$-polarised $K3$ surfaces, where $M$ is a lattice corresponding to an ambi-typical stratum, is either $1,2$ or $4$. 
It is at least $2$ in the cases listed in Proposition \ref{prop:embeddings}.
To determine the exact number one has to compute the index $[\tO(T): \tO^+(T)]$.
Indeed, all three possible cases occur as the 
following example shows. 
\begin{ex}\label{ex:components}
Assume that the root lattice $R$ has rank $17$.
Then the following holds:
out of the $25$ rank $17$ lattices we have $|\calM_T|^c=1$ in $18$ cases. 
In the four cases $D_4+A_7+A_4+2A_1$, $D_4+2A_6+A_1$, $D_4+A_6+A_3+2A_2$ and $D_4 + 2A_5 +3A_1$ we have $| \calM_T |=1$ and $[\tO(T):\tO^+(T)]=1$ and thus $|\calM_T|^c=2$. Finally, in the three cases 
$D_4 +A_7 +2A_2 +2A_1$, $D_4 +2A_4 +2A_2 +A_1$ or  $D_4 +3A_3 +2A_2$ we have $| \calM_T |=2$ and  $[\tO(T):\tO^+(T)]=1$
and hence $|\calM_T|^c=4$. This follows from Kirschmer's computations, see Table \ref{table:OPlus}  in the Appendix. Note that we have $2$ components if either $[\tO(T):\tO^+(T)]=1$ or 
$[\O(D(T)):\overline{\O(T)]}=2$ and $4$ components if both of these indices are $1$ and $2$ respectively. In all other cases we have $1$ component. 
\end{ex}

\begin{rem}\label{rem:groupscomponents}
It is interesting to compare this with the Shimada's list of non-connected moduli of elliptic $K3$ surfaces, see \cite[Corollary 1.5 and Table II]{shimada2}. We first observe that the lattices which appear in Proposition \ref{prop:embeddings} do not appear 
in Shimada's list. The reason is that different moduli spaces of lattice-polarised $K3$ surfaces can lead to the same Shimada stratum. This is due to symmetries of the root lattice $R$ and we shall discuss this in the next section. 
On the other hand, there are three root lattices which appear in both Table \ref{table:list} of our paper and Table 3 in \cite[p.555-557]{shimada2}. 
These are $D_4+2A_6+A_1$, which are our cases (38/39), as well as $2A_3+8A_1$ and $4A_3+2A_1$ which are our cases (6) and
(7) respectively. In the case of $D_4+2A_6+A_1$ we have two complex conjugate components. The lattice $4A_3+2A_1$ appears in  \cite[Table 3]{shimada2}
in connection with the 
Mordell-Weil torsion $\ZZ/ 2\ZZ$. In this case one has two components.
However, in our situation we have Mordell-Weil torsion $\ZZ/4\ZZ$ and this leads to one component only. Finally, in the case of $2A_3 + 8A_1$ the two components in Shimada's list come from inequivalent isotropic 
subgroups $\ZZ/2\ZZ \times \ZZ/2\ZZ$ which correspond to different intersection behaviour of the torsion sections. Only one of these cases leads to an ambi-typical stratum, see the discussion 
in Remark \ref{rem:explanationShimadastrata}.    
\end{rem}

%%%%%%%%%%%%%%%%%%%%%%%%%%%%%%%%%%%%%%%%%
%%%%%%%%%%%%%%%%%%%%%%%%%%%%%%%%%%%%%%%%%

\section{The moduli maps}\label{sec:modulimaps}

The aim of this section is to investigate the precise relationship between moduli spaces of lattice-polarised $K3$ surfaces and ambi-typical strata.   In particular we show that there is a finite map from certain moduli 
spaces of lattice polarised $K3$ surfaces to ambi-typical strata and compute the degree of this map.

For this we consider an ambi-typical stratum
\begin{equation}\label{equ:ambitypical}
\calA=\overline{\calF'_{\Gamma,i}} = \overline{\calFF'_{R,G,{\iota}}}. 
\end{equation}
As before we denote by $L$ the lattice defined by the isotropic group $G \subset D(R)$ and set $M=U+L$.
Our first goal is to associate to such a given ambi-typical stratum certain moduli spaces of lattice-polarised $K3$ surfaces. First of all, given any surface $S$ in such a stratum, we identify   
the sublattice spanned by the section and the fibre with a fixed summand  $U$ of the $K3$ lattice $L_{K3}$. More precisely, we  first choose the standard basis $e,f$ of $U$ with $e^2=f^2=0$ and $e.f=1$ and
identify the fibre class with $f$ and the section $s$ with $e-f$. This can be done once and for all simultaneously for all surfaces in this stratum. 
Now choose a generic surface $S$ in $\overline{\calFF'_{R,G,{\iota}}}$ and a marking $\varphi: H^2(S,\ZZ) \to L_{K3}$ where the sublattice spanned by the fibre and the section are identified with a given summand $U$ as described above. 
The components of the singular fibres not meeting the section define a sublattice isomorphic to $R$ whose saturation in the $K3$ lattice is isomorphic to $L$. We note that the isomorphism with $R$ is not canonical, but depends on an
identification of these components with a set of simple roots of $R$. There are finitely many ways of choosing such an identification.   
Adding the sublattice spanned  by the fibre and the section one obtains an embedding 
\[
\iota: M \to \NS(S) \subset H^2(S,\ZZ)\cong L_{K3}.
\]
and thus an element $(S,\iota)$ in a connected component  $\calN_{M,\ell}$ of the  moduli space  $\calN_{M,\iota}$  of lattice polarised $K3$ surfaces. 
We shall refer to such a component, as well as to the moduli space $\calN_{M,\iota}$ itself, as   {\em associated} to $\calA$.
Note that we do not claim that  $\calN_{M,\ell}$  or $\calN_{M,\iota}$ are uniquely determined by the ambi-typical stratum since there is no canonical way of identifying the fibre components with a 
root system.  Nor do we claim that we obtain a morphism from (an open part of) the stratum  $\calA$ to such a moduli space $\calN_{M,\iota}$. 
Indeed, we always have (at least) the ambiguity given by the symmetries $S_R$ 
of the Dynkin diagram associated to $R$. Changing the isomorphism by such an element can have two effects. One is that it defines different points in
the same moduli space of $M$-lattice polarised $K3$ surfaces. The other is that
it leads to  different moduli spaces of lattice polarised $K3$ surfaces. Note however, that two such moduli spaces of $M$-polarised $K3$ surfaces define the same Shimada stratum as the combinatorial type of the singular fibres is the same.   
As we shall see, both cases  occur. We shall now discuss this in detail. 
Before doing this we establish the following convention. Given such a moduli space $\calN_{M,\iota}$ of lattice polarised $K3$ surfaces and a $K3$ surface $S$ which appears in this moduli space,
we always have a distinguished copy $U$. In each component $\calN_{M,\ell}$ we consider the non-empty open part of the moduli space where  this copy of $U$ contains a nef isotropic class and 
an irreducible $(-2)$-curve. 
This determines a Jacobian fibration $S \to \PP^1$ whose fibre and section are contained in $U$. Unless stated otherwise we will always work with this elliptic fibration.
By  $\calN_{M,\ell}^0$ we further denote the open subset where the configuration of the singular fibres  is generic. 
 
\begin{pro}\label{pro:constructionmap}
The following holds:
\begin{itemize}
\item[{\rm{(1)}}] Given an ambi-typical stratum $\calA$ there are only finitely many components of moduli spaces of lattice-polarised $K3$ surfaces associated to $\calA$.
\item[{\rm{(2)}}] Let  $\calN_{M,\ell}$ be a component of a moduli space  $\calN_{M,\iota}$ of lattice polarised $K3$ surfaces associated to $\calA$ and
let  $\calN_{M,\ell}^0$ be the open subset where the configuration of the singular fibres is generic.    
Then there is a natural finite to one dominant morphism $\calN^0_{M,\ell} \to \calA$.
\end{itemize}
\end{pro} 
\begin{proof}
Claim $(1)$ follows from Nikulin's theory since there are only finitely many inequivalent primitive embeddings $\iota: L \to L_{K3}$.  

To prove (2) we recall that 
we have for every  element in $\calN_{M,\iota}$ a well defined Jacobian fibration $S \to \PP^1$ which can be written in Weierstra{\ss} form. The fact that $\calF$ is a coarse moduli space of $K3$ surfaces with  a Jacobian fibration defines 
a morphism $\calN_{M,\ell} \to \calF$. On the open subset  $\calN_{M,\ell}^0 $ the monodromy is constant with monodromy group $\Gamma$  and hence we obtain a morphism $\calN_{M,\ell}^0 \to \calA$. This map has finite fibres since there are
only finitely many ways of identifying the components of the singular fibres not intersection the $0$-section with a basis of the root lattice $R$. The map is dominant by the definition of an ambi-typical stratum. 
\end{proof}
  
We now want to understand these maps better. The next two statements will be useful for this.

\begin{lem} \label{lem:roots}
Let $(R,G)$ be a pair consisting of a root lattice and an isotropic subgroup $G \subset D(R)$ with associated overlattice $L$, arising from one of the families listed in Table \ref{table:list}.
Then the roots of $R$ and $L$ coincide. 
\end{lem}
\begin{proof}
A proof can be found in \cite[Proposition 3.2]{shimada2}. Indeed, this is a simple geometric argument. We can use the fact that $M=U+L$ is isomorphic to the N\'eron-Severi group $\NS(S)$ of some (sufficiently general) $K3$ surface $S$ and that $R$ is the
subgroup of $\NS(S)$ generated by all fibre components which do not meet the $0$-section.   
Then the claim is geometrically clear: assume that $r$ is a root of $L$, which is not a root of $R$. Then $\pm r$ is effective and meets neither the $0$-section nor a general fibre, since it is orthogonal to the summand $U$ which contains the classes 
of the $0$-section and a general fibre. Hence $\pm r$ defines 
a union of rational curves on the associated elliptic $K3$ surfaces consisting of  components of singular fibres, which do not intersect the $0$-section. But  these roots are already contained in $R$. 
\end{proof}
\begin{rem}
Here we use the specific geometric situation. In general such a statement is false. Indeed, let $R=4A_1$ and let $H$ be the subgroup generated by the diagonal  $(1,1,1,1)$ of $D(R)= 4 \ZZ/2\ZZ$.  
Then the overlattice is $D_4$ and one obtains a new root, namely $r= \frac12 \sum_{i=1}^4 r_i$. 
\end{rem}

Using the above lemma we can now prove
\begin{pro} \label{prop:automorphisms}
Let $(R,G)$ be a pair of a root lattice and an isotropic subgroup arising from one of the families listed in Table \ref{table:list} with associated overlattice $L$. Then 
$$
\O(L) \cong \{g \in \O(R) \mid \overline{g}(G)=G \}.
$$
\end{pro}
\begin{proof}
Clearly, the elements of $\O(R)$, which leave $G$ invariant (as a subgroup), define isometries of $L$, by construction of this lattice. Conversely, let
$g \in \O(L)$. Then $g$ maps roots of $L$ to roots of $L$. By the above Lemma \ref{lem:roots} these are exactly the roots of $R$ and hence $R$ is mapped to itself. 
Hence $g$ is an isometry of the lattice inclusion $R \subset L$ and thus $\overline{g}$ maps the isotropic subgroup $G$ to itself.   
\end{proof}
  
Let $R$ be  a root lattice. We recall that the Weyl group $W_R \subset \O(R)$ is the group generated  
by the reflections with respect to the roots $r \in R$. 
We denote by $S_R$ the subgroup of $\O(R)$ which is induced by symmetries of the Dynkin diagram.
We also
recall from \cite[Theorem 12.2]{Hum} that the isometry group of $R$  is the semi-direct 
product of the Weyl group $W_R$ and the group $S_R$:
$$
\O(R)= W_R \rtimes S_R.
$$

Elements in the Weyl group $W_R$ act trivially on the discrimant $D(R)$ and  hence, in our situation, lift to isometries of $L$. In particular, we can consider $W_R \subset \O(L)$.
We further denote the subgroup of $S_R$ which leaves $G$ invariant (as a subgroup) by $S_R^G$.
It now follows from Proposition \ref{prop:automorphisms} that $\O(L)$ is generated by the Weyl group $W_R$ together with the 
group $S_R^G$ of diagram isometries which fix the subgroup $G$, i.e.
\begin{equation}\label{equ:groupgenerated} 
\O(L)=W_R \rtimes S_R^G.
\end{equation}  
We can extend all elements in $\O(L)$ to isometries of $M=U + L$ by  taking the identity on the first factor $U$. 
By doing this we can consider $S_R^G$ as a subgroup of $\O(M)$ and to simplify the notation we shall denote the image of  $S_R^G$ in $O(M)$ by $S_M$. 
This notation is unambiguous in our situation as 
there is no lattice $R$ in Table  \ref{table:list} which comes with more than one subgroup $G \subset D(R)$.

The lattice $M$ has signature $(1,\rank(M)-1)=(1,\rank(R)+1)$. As before we fix a connected component $C^+(M)$ of the positive cone in $M_{\RR}$. We have already mentioned that the  hyperplanes orthogonal to the roots $r \in M$ subdivide 
$C^+(M)$ into connected components, the Weyl chambers of $M$. It is well known, see \cite[Proposition 8.2.6]{Hu} that the Weyl group $W_M$ of $M$ acts simply transitively on the Weyl chambers in $C^+(M)$.
By Lemma \ref{lem:roots} we have  
\begin{equation} \label{equ:Weylgroups}
W_R=W_L \subset  W_M \subset \tO^+(M).
\end{equation} 
In particular, $W_R$ also acts faithfully on the set of Weyl chambers of $M$. 

For our applications it will be essential that the group $S_M$ maps Weyl chambers to themselves:

\begin{pro}\label{pro:Dynkinauto}
Let $(R,G)$ be a pair of a root lattice and an isotropic subgroup arising from one of the families listed in Table \ref{table:list}.
The elements of  $S_M$, i.e. the symmetries of the Dynkin diagram of $R$ fixing $G$ as a group, map all Weyl chambers of $C^+(M)$ to themselves. 
\end{pro}
\begin{proof}
Here we make again use of the special situation, namely the fact that there exists an $M$-polarised $K3$ surface $S$ with $\NS(S)\cong M$ and such that the components of the singular fibres which do not intersect the $0$-section 
generate the root lattice $R$.  In fact these define a set of simple roots which gives rise to the Dynkin diagram associated to $R$ and $S_R$ acts on these. 
More precisely, we can choose a primitive embedding $\iota: M \to L_{K3}$ and  a marking $\varphi: H^2(S,\ZZ) \to L_{K3}$  such that $\varphi(\NS(S))=\iota(M)$.
Let $\Delta^+$ be the set of positive roots (i.e. effective $(-2)$-classes in $\NS(S)$). By \cite[Section 8.2.3]{Hu} it is enough to show that every element of $S_M$ maps positive roots to positive roots. Let $g\in S_M$ and let $f$ be the class of a fibre.
Then $g(f)=f$. If $s$ is a positive root with $(f,s) \neq 0$, then $(f,s)>0$. 
Since $(f,g(s))=(g(f),g(s))=(f,s)>0$ it follows that $g(s)$ is again positive. By definition of the group $S_M$ we have  $g(s_0)=s_0$ where $s_0$ is the class of the $0$-section.
Then the same argument also shows that $g(s)$ is positive for every positive root $s$ with $(s,s_0)\neq 0$. It remains to consider the positive roots which are orthogonal to $f$ and $s_0$, but these are exactly the positive roots of 
$R$, which are given by non-negative combinations of components of singular fibres which do not intersect the $0$-section. 
Since $g$ permutes these, the claim follows. 
\end{proof}

We shall now discuss the action of the group of symmetries of the Dynkin diagram on the moduli spaces of $M$-lattice polarised $K3$ surfaces in more detail.
Let $(S,\tilde{\iota})$ be a general element in $\calN_{M,\iota}$, more precisely an element in  $\calN_{M,\ell}^0$, where $\calN_{M,\ell}$ is a connected component of $\calN_{M,\iota}$ 
and $\calN_{M,\ell}^0$ denotes 
the open set where the fibre configuration is constant. Then  $\tilde{\iota} : M \to \NS(S) \subset H^2(S,\ZZ)$
is a primitive embedding and 
 there exists a marking $\varphi: H^2(S,\ZZ) \to L_{K3}$ such
 that $\tilde{\iota}= \varphi^{-1} \circ \iota$ and 
 such that $\tilde{\iota}(C_M^{\operatorname {pol}})$ contains an ample class on $S$ where $C_M^{\operatorname {pol}}$ is the fixed Weyl chamber in $C^+(M)$ which we have chosen once and for all. 
 Every element $g_M \in S_M \subset \O(M)$ has, by Proposition \ref{pro:Dynkinauto}, the property that it fixes the Weyl chambers in $C^+(M)$. 
  Hence $\tilde{\iota} \circ g_M: M \to \NS(S)$ defines 
 again  an $M$-polarization on $S$. Now two cases can occur. The first is that  $\tilde{\iota}$ and $\tilde{\iota} \circ g_M$ define isomorphic embeddings of $M$ into the $K3$ lattice $L_{K3}$. In this case  $(S,\tilde{\iota})$ and   $(S,\tilde{\iota}\circ g_M)$
 define elements in the same moduli space $\calN_{M,\iota}$ and $g_M$ induces a map from  $\calN_{M,\iota}$ to itself identifying $(S,\tilde{\iota})$ and   $(S,\tilde{\iota}\circ g_M)$. 
 Note that the moduli space $\calN_{M,\iota}$ can have one or two components. If it has two components, then $g_M: \calN_{M,\iota} \to \calN_{M,\iota}$ can either fix the components or interchange them.
 By the discussion in Remark \ref{ex:components} all of these possibilities discussed actually occur in our situation.
 In the second case $\iota$ and $\iota \circ g_M$ define different embeddings and $g$ induces an isomorphism of moduli spaces
$\calN_{M,\iota} \to\calN_{M,\iota\circ g_M} $. 
The lattices where more than one embedding exists are listed in Proposition \ref{prop:embeddings}. In these cases we have two different embeddings, but there exist symmetries of the Dynkin diagram which defines 
isomorphisms $\calN_{M,\iota} \cong \calN_{M,\iota\circ g_M}$ (by Remark \ref{rem:groupscomponents}).
 Again, note that $\calN_{M,\iota}$ can have one or two components, but the latter does not occur among our cases. 

We can also formulate the above discussion in more group theoretic terms. 
The choice of a primitive embedding $\iota: M \to L_{K3}$ with orthogonal complement  $T$ (up to isomorphism of embeddings) is 
equivalent to the choice of  an isomorphism $\alpha_{\iota}: (D(M),q_M) \cong (D(T),-q_T)$ (modulo $\bO(T)$).
Recall the definitions of the groups $S^G_R$ and $S_M$ from (\ref{equ:groupgenerated}) and the subsequent paragraph.  
An element $g_M \in S_M$ defines an isometry $\overline{g}_M \in   \O(D(M))$ and, via $\alpha_{\iota}$, an isometry $\alpha_{\iota}(\overline{g}_M) \in   \O(D(T))$. 
The morphism which maps $(S,\tilde{\iota})$ to  $(S,\tilde{\iota}\circ g_M)$ maps $\calN_{M,\iota}$ to itself if and only if $\alpha_{\iota}(\overline{g}_M) \in \bO(T)$.

If $g_M$ induces a morphism from $\calN_{M,\iota}$ to itself, then we can describe this map explicitly.  In this case  $\alpha_{\iota}(\overline{g}_M)\in \bO(T)$ and we can lift this to an element $g_T \in  \O(T)$ such that the pair $(g_M,g_T)$ 
extends to an isometry of $L_{K3}$. The lift $g_T$ is uniquely determined up to $\tO(T)$. The action of $g_T$ on $\calN_{M,\iota}= \Omega_T/ \tO(T)$ then induces the map on $\calN_{M,\iota}$ in question.
Now  $\calN_{N,\iota}$ has two components if and only if $\tO^+(T)=\tO(T)$ and in this case $g_T$ interchanges the two components if and only if $g_T$ has real spinor norm $-1$, i.e. if and only if
$g_T \notin \tO^+(T)$ (which in this case is independent of the chosen lift). 

Let  $\pi_M: \O(M) \to \O(D(M))$ and $\pi_T: \O(T) \to \O(D(T))$ be the canonical projections. We define
\begin{equation}
\overline{S}_M=\pi_M(S_M)
\end{equation} 
which we will, via $\alpha_{\iota} : D(M) \to D(T)$, also consider as a subgroup $\overline{S}_M \subset \O(D(T))$. As a subgroup this depends on the embedding $\iota$ as it is defined via the 
isomorphism $\alpha_{\iota}$. This becomes important when we define
\begin{equation}\label{equ:SMiota}
 \overline{S}_{M,\iota} =  \overline{S}_M \cap \bO(T), \,    \overline{S}^+_{M,\iota} = \overline{S}_M \cap \bO^+(T)
\end{equation} 
and their pre-images
\begin{equation}\label{eq:lifts}
\Gamma_{M,\iota}=\pi_T^{-1} (\overline{S}_{M,\iota}) \subset \O(T), \, \Gamma^+_{M,\iota}=\pi_T^{-1} (\overline{S}^+_{M,\iota}) \subset \O^+(T).
\end{equation}
The elements in $\Gamma_{M,\iota}$ are those isometries of $T$ which can be extended to the overlattice $L_{K3}$ of $T \oplus \iota(M)$ (where we do not ask that these isometries act trivially on $\iota(M)$). 
The group $\Gamma_{M,\iota}$ acts on the period domain $\Omega_T$ and induces an action on the moduli space $\calN_{N,\iota}$. 
An element in $\Gamma_{M,\iota}$ fixes the components of $\calN_{N,\iota}$ if and only if it is in  $\Gamma^+_{M,\iota}$. 

By the above discussion the group $\overline{S}_M$ operates on $\O(D(T))/\bO(T)$ via $h \mapsto h \circ \alpha_{\iota}(\overline{g})$.
It will be important for us to know whether this action is transitive. 
The following is essentially a reformulation of Remark \ref{rem:groupscomponents}.

\begin{pro}\label{prop:trans}
If $M$ is a hyperbolic lattice arising from one of the families listed in Table \ref{table:list}, then $\overline{S}_M$ acts transitively on $\O(D(T))/\bO(T)$ unless we are in the case where $R=2A_3 + 8A_1$.
\end{pro}
\begin{proof}
As we have said before (see Remark \ref{rem:groupscomponents}) comparing the lists in \cite[Corollary 1.5]{shimada2}
and \cite[Table 3]{shimada2} with our Table \ref{table:list} we find three lattices, namely  $4A_3 + 2A_1$, $2A_3 + 8A_1$ and $D_4+2A_6 +A_1$. 
The first lattice is irrelevant for us as it appears with Mordell-Weil torsion $\ZZ/2\ZZ$ in Shimada's lists, whereas we have torsion $\ZZ/4\ZZ$. The reason that $D_4+2A_6 +A_1$ appears in Shimada's lists is that 
there are two connected component, but they belong to the same primitive lattice embedding. Finally, for  $2A_3 + 8A_1$ there exist two combinatorially different components of the moduli space coming from 
different lattice embeddings (but only one of them appears in our classification). 
\end{proof}  

Our previous discussion can now be summarised as follows. Let $\calA$ be an ambi-typical stratum, resp. let $\calA \cup \overline{\calA}$ be the union of the two complex-conjugated components 38/39. Then there are finitely many components of
moduli spaces $\calN_{N,\iota}$  of lattice-polarised $K3$ surfaces which are associated to $\calA$ or $\overline{\calA}$. The group $\overline{S}_M$ acts transitively on the set of all  moduli spaces $\calN_{M,\iota}$ whose components are
associated to $\calA$ or $\calA \cup \overline{\calA}$ respectively. 
The groups $\Gamma_{M,\iota}$ act on the moduli spaces $\calN_{M,\iota}$ (and their elements may interchange connected components of these moduli spaces).

\begin{rem}\label{rem:groupSMiota}
By inspection one sees that for all our lattices $-1 \in \overline{S}_{M,\iota}$ and hence also $-1 \in \Gamma_{M,\iota}$ and hence  $\Gamma_{M,\iota}/(\pm \tO(T)) \cong \overline{S}_{M,\iota}/(\pm 1)$. 
\end{rem}

\begin{teo}\label{pro:descriptionmap}
Let $\calA$ be an ambi-typical stratum and let $\calN_{M,\iota}$ be a moduli space of lattice polarised $K3$ surfaces associated to $\calA$. Then the following holds:
\begin{itemize}
\item[(1)] If $\calA$ is one of the ambi-typical strata different from 38/39, then the dominant map
$\calN_{M,\iota}^0 \to \calA$ is given by the action of the finite group $\Gamma_{M,\iota}/(\pm\tO(T))$ which acts faithfully.
\item[(2)] Let  $\calA \cup \overline{\calA} $ be the union of the two complex conjugated strata 38/39. Then $\calN_{M,\iota}$ has two connected components $\calN_{M,\ell}$ and $\overline{\calN}_{M,\ell}$ and the dominant map
$\calN_{M,\ell}^0 \cup \overline{\calN}_{M,\ell}^0 \to \calA \cup \overline{\calA}$ is given by 
the group   $\Gamma_{M,\iota}/(\pm\tO(T))$ which acts faithfully.
\end{itemize}
\end{teo}
\begin{proof}
We start with a surface $S\in \calA$ (or $S \in  \overline{\calA} $). 
As we have explained before, identifying the fibre components not meeting the $0$-section of the singular fibres with simple roots of $R$ and the lattice spanned by the fibre and section with a copy 
of $U$ we obtain a lattice polarization $\tilde \iota: M \to \NS(S) \subset H^2(S,\ZZ)$. Let $\varphi: H^2(S,\ZZ) \to L_{K3}$ be a marking and set $\iota=\varphi \circ \tilde \iota$.
We have to determine when two lattice-polarised 
$K3$ surfaces  in  $\calN_{M,\iota}$ are mapped to the same point in $\calA$ or $\overline{\calA} $ respectively.   

Assume that two surfaces $(S_1,\tilde{\iota}_1)$ and $(S_2,\tilde{\iota}_2)$ in  $\calN_{N,\iota}$ 
define the same point in $\calA$  or  $\calA \cup \overline{\calA}$ respectively. Then there is an isomorphism $f:S_2 \to S_1$ which respects the elliptic fibration. 
This induces a map $f^*: \NS(S_1) \to \NS(S_2)$. Let $\varphi_i : H^2(S_i;\ZZ) \to L_{K3}$ be markings with $\iota= \tilde{\iota}_i \circ \varphi_i$ for $i=1,2$. 
Then $\varphi_2 \circ f^* \circ \varphi_1$ defines an isometry of $M=U + L$. This is the identity on $U$ as the 
$0$-section and the general fibre are mapped to the $0$-section and the general fibre respectively. By restriction this then
defines  an isometry of $L$. Recall  from  (\ref{equ:groupgenerated}) that  
$\O(L)=W_R \rtimes S_R^G$ and from (\ref{equ:Weylgroups}) that $W_R=W_L \subset W_M   \subset \tO^+(L)$.  
Since the isometry $\varphi_2 \circ f^* \circ \varphi_1|_M$ maps $C_M^{\operatorname {pol}}$ to itself and since the Weyl group acts faithfully on the Weyl chambers it follows that   
$\varphi_2 \circ f^* \circ \varphi_1|_M$ defines an element in $\Gamma_{M,\iota}$.
Conversely, if two lattice polarised $K3$ surfaces in $\calN_{N,\iota}^0$ are conjugate under $\Gamma_{M,\iota}$ the underlying elliptic $K3$ surfaces are isomorphic and hence define the same point in $\calA$
or  $\calA \cup \overline{\calA}$ respectively.  
Finally, since $\pm \id_T$ are the only elements in $\O(T)$ which act trivially on $\Omega_M$ it follows that the group $\Gamma_{M,\iota}/(\pm\tO(T))$ acts faithfully on $\calN_{M,\iota}$. 
\end{proof}

\begin{cor}\label{cor:countingdegree}
The following holds:
\begin{itemize}
\item[(1)] If $\calA$ is an ambi-typical stratum different from 38/39, then the degree of the dominant 
map $\calN_{M,\iota}^0 \to \calA$ is given by 
$$
|\overline{S}_M/(\pm1)|/|\calM_T|^c=|\overline{S}^+_{M,\iota}/(\pm 1)|.
$$
\item[(2)] If $\calA$ (or $\overline{\calA}$ respectively)  is one of the strata 38/39, then the degree of the dominant  
map $\calN_{M,\ell}^0 \cup \overline{\calN}_{M,\ell}^0  \to \calA \cup \overline{\calA}$ is given by 
$$
2|\overline{S}_M/(\pm1)|/|\calM_T|^c = |\overline{S}^+_{M,\iota}/(\pm 1)|=24.
$$
\end{itemize}
\end{cor}
\begin{proof}
The first equality follows from the fact that $\overline{S}_M$ acts transitively on the connected components of $M$-polarised $K3$ surfaces with the exception of the cases 38/39 where we have two orbits. The second
equality follows from the definition of the group $\overline{S}^+_{M,\iota}$ and the constructin of the covering map. 
\end{proof}  

In the case of $1$-dimensional strata one can compute all the data of the covering map from the moduli space of lattice polarised $K3$ surfaces to the ambi-typical stratum. In Table \ref{table:genera} we give 
these data for all strata of dimension $1$ with more than one connected component. We list the case number, the root lattice, the number of connected components, the order of the group $\overline{S}_M$ coming from the symmetries of the 
Dynkin diagram, the degree of the covering map, the genus of the modular curve parameterising the lattice polarised $K3$ surfaces and finally the genus $g_{\operatorname{BPT}}$ of the ambi-typical 
stratum. The later is always $0$ in accordance  with  \cite[Theorem]{BPT} which says 
that all monodromy strata are rational. We find it remarkable that in contrast the genus of the components of the associated moduli spaces of lattice-polarised $K3$ surfaces can be as high as $13$.    
Note that the group $\overline{S}_M$ is a subgroup of the symmetry group of the Dynkin diagram of the root lattice. If the isotropic group $G$ is trivial, then the two groups coincide. 
Otherwise the condition that $G$ must be fixed can impose nontrivial extra conditions.
This is the case for numbers 
$37, 48, 49$ where $\overline{S}_{M,\iota}$ is a proper subgroup of index $2$ of $\overline{S}_M$.  Furthermore, in these cases $\overline{S}^+_{M,\iota}$ is also an index $2$ subgroup 
of $\overline{S}_{M,\iota}$. Altogether $\overline{S}^+_{M,\iota}$  can have index $1$, $2$ or $4$ in $\overline{S}_{M}$  
The relevant computations, in particular of the genera $g$, were performed by Markus Kirschmer and are presented in the appendix. 

\begin{table}\label{table:genera}
\[
\begin{array}
{c|c|c|c|c|c|c}
\text{Number}& \text{root lattice} &|\calM_T|^c & |\overline{S}_{M}/(\pm 1)| & d= |\overline{S}^+_{M,\iota}/(\pm 1)|  & g & g_{\operatorname{BPT}}\\
\hline
35&D_4+A_7+A_4+2A_1&  2 & 8 & 4 & 0 &0\\
37&D_4+A_7+2A_2+2A_1&  4 & 32 & 8 & 0 &0\\
38/39&D_4 +2A_6+ A_1&  2 & 24 & 24 & 1 &0\\
42&D_4+A_6+A_3+2A_2&  2 & 96 & 48 & 13 &0\\
44&D_4+2A_5+3A_1&  2 & 24 & 12 & 0 &0\\
48&D_4+2A_4+2A_2+A_1& 4 & 192 & 48   & 1 &0\\
49&D_4+3A_3+2A_2& 4 & 384 & 96  & 5 &0\\
\end{array}
\]
\caption{Covering data for all $1$-dimensional ambi-typical strata with more than one component}
\label{table:covering}
\end{table}

%%%%%%%%%%%%%%%%%%%%%%%%%%%%%%%%%%%%%%%%%
%%%%%%%%%%%%%%%%%%%%%%%%%%%%%%%%%%%%%%%%%

\section*{Appendix. Numerical calculations \\ Markus Kirschmer}
\label{sec:appendix}
\appendix
\stepcounter{section}

In this appendix 
we summarise some explicit computations concerning the 25 rank 17 
lattices in Table \ref{table:list}.
All computations were done in \textsc{Magma} \cite{Magma} and the complete code is available from
\url{www.math.rwth-aachen.de/~Markus.Kirschmer/magma/K3.html}.

Let $L$ be one of the rank 17 lattices in Table \ref{table:list} and set $M := U \oplus L$.
Further let $\iota \colon M \hookrightarrow L_{K3}$ be a primitive embedding.

\subsection*{Constructing $T$}
We first need to construct an integral lattice $T$ isometric to $\iota(M)^\perp$.
This can be done without constructing an embedding $\iota$ as follows:

Let $(V, q)$ be the ambient quadratic space of $T$. 
The fact that $M \oplus T$ and $L_{K3}$ lie in isometric quadratic spaces shows that $(V, q)$ has signature $(2,1)$ and determinant $\det(M)$.
It also uniquely determines the Hasse-Witt-invariants of $(V, q)$.
Using \cite[Alg. 3.4.3]{Kir} we can construct a rational quadratic space isometric to $(V, q)$.
Let $X$ be a maximal even lattice in $(V, q)$, ie. $q(X) \subseteq 2\ZZ$ and no lattice properly containing $X$ has that property, cf. \cite[Alg. 3.5.5]{Kir}.
By \cite[Theorem 91:2]{OMeara}, the genus of $X$ is unique, thus we may assume that $T \subseteq X$.

For any prime $p$ dividing $\# D(L)$, let $S_p$ be the $p$-Sylow subgroup of $D(L)$. Then
\[ \{ Y \subseteq X \mid D(Y) \cong S_p \mbox{ and } \# q_Y^{-1}(\{a\}) = \# q_L^{-1}(\{-a\}) \mbox{ for all } a \in \QQ/2\ZZ \} \]
consists of a single genus. Let $Y^{(p)}$ be any representative. By Proposition \ref{prop:uniquegenus}, the lattice
\[ T:= \cap_{p} Y^{(p)} \]
is isometric to $\iota(M)^\perp$.

\subsection*{Computing $\O(T)$ and its subgroups}

A finite generating set of $\O(T)$ can be constructed using a variation of Voronoi's algorithm by M.H.~Mertens \cite{Mertens}.
A slight modification of Mertens' algorithms also yields a finite presentation of $\O(T)$ using Bass-Serre theory \cite{Voronoi}.
This modification was provided to us by S. Sch\"onnenbeck.

The group $\tO(T)$ is the kernel of the homomorphism $\pi_T \colon \O(T) \to \O(D(T), q_T)$. 
Since the group $\O(D(T), q_T)$ is finite, we can construct a finite generating set of $\tO(T)$ 
using the standard orbit stabiliser algorithm.
Similarly, the spinor norm map $\textnormal{sn}_\RR \colon \O(T) \to \{\pm 1\}$ yields finite generating sets for $\O^+(T)$ and $\tO^+(T)$.
Note that spinor norms can be computed using Zassenhaus' trick \cite{Zassenhaus}.
But one has to keep in mind that our normalization of the spinor norm on $(T, q)$ corresponds to Zassenhaus' spinor norm on the lattice $(T, -q)$.

Next we want to compute generators for the group $\Gamma_{M, \iota}$ cf. equation (\ref{eq:lifts}).
The group $D(M)$ is finite and so is its automorphism group $\Aut(D(M))$.
Thus the subgroup
\[  \O(D(M), q_M) = \{ f \in \Aut(D(M)) \mid q_M(f(x)) = q_M(x) \mbox{ for all } x \in D(M) \} \]
can be enumerated by brute force. 
Similarily, we enumerate the subgroup $\overline{S}_M \subseteq \O(D(M), q_M)$ induced by the automorphism group of the Dynkin diagram of the root lattice $R$.
If $\overline{S}_M = \O(D(M), q_M)$, then $\Gamma_{M,\iota} = \O(T)$.
Suppose now $\overline{S}_M \subsetneq \O(D(M), q_M)$.
In these cases, it just happens that $\pi_T \colon \O(T) \to \O(D(T), q_T)$ is onto.
We start by computing any isometry $\alpha \colon (D(M), q_M) \to (D(T), -q_T)$ using a backtrack approach.
This gives us an isomorphism $\O(D(T)) \cong \O(D(M))$.
The fact that $\O(T) \to \O(D(T), q_T)$ is onto shows that there exists some $f \in \O(T)$ such that $\pi_T(f) \circ \alpha = \alpha_\iota$.
In particular, the pre-image of $\overline{S}_{M,\iota}$ under $\O(T) \to \O(D(T)) \cong \O(D(M))$ must be conjugate to $\Gamma_{M, \iota}$ 
and we find generators for this group using the orbit stabiliser algorithm.

\subsection*{Fuchsian groups}
Let $\SO^+(T) = \{ \varphi \in \O^+(T) \mid \det(\varphi) = 1\}$.
We denote by $\SO^+(2,1)$ the connected component of the special orthogonal group of a real quadratic space of signature $(2,1)$.
The sporadic isomorphism between $\SO^+(2,1)$ and $\PSL(2,\RR)$ implies that the type IV homogeneous domain associated to $T$ is isomorphic to 
the upper half plane $\HH_1$. Moreover, it induces  
an injection $\SO^+(T) \hookrightarrow \PSL(2,\RR)$ and thus an action of $\SO^+(T)$ on the upper half plane $\HH_1$.
This action is properly discontinuous.
Hence any finite index subgroup $G$ of $\SO^+(T)$ is a Fuchsian group, cf. \cite[Theorem 2.2.6]{Katok}.

Suppose $G$ is a finite index subgroup of $\SO^+(T)$ and let $g:= g(\HH_1 / G)$ be the genus of the (compactified) curve $\HH_1 /G$.
By \cite[Section~4.3]{Katok}, the group $G$ admits a presentation
\begin{equation}\label{eq:pres}
 G \cong \left\langle a_1,b_1,\ldots,a_g,b_g,x_1,\ldots,x_d,y_1,\dots,y_t \middle\vert
 \begin{array}{@{}l@{}}
 x_1^{m_1} = \ldots = x_d^{m_d} = 1 \mbox{ and } \\
 x_1\cdots x_d y_1 \cdots y_t [a_1,b_1] \cdots [a_g,b_g] = 1 
 \end{array}
 \right\rangle
\end{equation}
with $t=0$ if and only if $\HH_1/G$ is compact.

Since the index of $G$ in $\O(T)$ is finite, we can obtain a finite presentation of $G$ from the presentation of $\O(T)$ using the Reidemeister-Schreier method \cite[Section~2.3]{CGT}.
Even though this presentation might not be in the form of eq.~\eqref{eq:pres}, it is good enough to determine the genus of $\HH_1/G$:

\begin{lem}\label{lem:appendixnoncongruence}
Let $L$ be one of the rank $17$ lattices of Table \ref{table:list} and let $\iota \colon L \oplus U \to L_{K3}$ be a primitive embedding.
Further, let $T:= \iota(L \oplus U)^\perp$ and let $G$ be a finite index subgroup of $\SO^+(T)$.
\begin{enumerate}
\item The space $\HH_1/G$ is compact. 
\item The torsion free part of $G/G'$ has rank $2 g(\HH_1/G)$.
\end{enumerate}
\end{lem}
\begin{proof}
It suffices to prove the first statement for $G = \SO^+(T)$.
An explicit computation shows that the abelian group $\SO^+(T) / \SO^+(T)' \cong (\ZZ/2\ZZ)^{r}$ is an elementary abelian $2$-group.
Suppose $\HH_1/\SO^+(T) $ is not compact, i.e. the parameter $t$ in eq.~\eqref{eq:pres} is non-zero.
The isomorphism type of  $\SO^+(T) / \SO^+(T)'$ implies that $g(\HH_1/ \SO^+(T) ) = 0$ and 
\[ \SO^+(T) \cong \ZZ/2\ZZ * \ldots * \ZZ/2\ZZ  \]
is a free product of $r$ copies of $\ZZ/2\ZZ$.
An explicit computation shows that for all lattices $T$, the groups $\SO^+(T)$ and $\ZZ/2\ZZ * \ldots * \ZZ/2\ZZ$ have different numbers of subgroups of small index.
This proves the first assertion.\\
The second assertion follows immediately from the fact that the parameter $t$ in eq.~\eqref{eq:pres} is zero as $\HH_1/G$ is compact.
\end{proof}

Suppose now $G$ is a subgroup of $\O^+(T)$.
We denote by $\HH_1/G$ the space $(\HH_1/ \pm G \cap \SO^+(T))$ where $\pm G$ is the subgroup of $\O(T)$ generated by $G$ and $-I_3$.

\subsection*{Results}

For each rank 17 lattice in Table \ref{table:list}, we constructed a lattice $T$ as well as generating sets for $\O(T)$, $\O^+(T)$, $\tO(T)$ and $\tO^+(T)$.
It turns out that in all cases $[\O(T) : \O^+(T)] = 2$ and $g(\HH_1/\O^+(T)) = 0$.
Table~\ref{table:OPlus} lists the indices  $I_1:= [\O^+(T) : \tO^+(T)]$, $I_2:=[\tO(T) : \tO^+(T)] $ and $ I_3:= [\O(D(T)) : \overline{\O(T)}]$,
the order of $\overline{S}_{M,\iota}/(\pm 1)$ and the genus of the moduli space of lattice polarised $K3$ surfaces. We note that in all but three cases  $\overline{S}_{M}/(\pm 1)=\overline{S}_{M,\iota}/(\pm 1)$.
The only cases where this is not the case are  37, 48 and 49 where  $\overline{S}_{M,\iota}/(\pm 1)$ has index $2$ in $\overline{S}_{M}/(\pm 1$). Furthermore, in these cases, as well as in 
35, 42 and 44 the group  $\overline{S}^+_{M,\iota}/(\pm 1)$ has  
index $2$ in $\overline{S}_{M,\iota}/(\pm 1)$.

\begin{table}[h]
\[
\begin{array}{c|c|c|c|c|c|c}
\# & \text{root lattice} & I_1 & I_2 & I_3 & |\overline{S}_{M}/(\pm 1)| & g(\HH_1/\tO^+(T) ) \\ \hline
24 & D_{4}+A_{13} & 12 & 2 & 1 & 6 & 1\\
25 & D_{4}+A_{12}+A_{1} & 12 & 2 & 1 & 6 & 1\\
26 & D_{4}+A_{11}+A_{2} & 16 & 2 & 1 & 4 & 0\\
27 & D_{4}+A_{11}+2A_{1} & 8 & 2 & 1 & 4 & 0\\
28 & D_{4}+A_{10}+A_{2}+A_{1} & 24 & 2 & 1 & 12 & 1\\
29 & D_{4}+A_{9}+A_{4} & 72 & 2 & 1 & 12 & 4\\
30 & D_{4}+A_{9}+A_{3}+A_{1} & 8 & 2 & 1 & 4 & 1\\
31 & D_{4}+A_{9}+2A_{2} & 96 & 2 & 1 & 48 & 7\\
32 & D_{4}+A_{9}+A_{2}+2A_{1} & 8 & 2 & 1 & 4 & 0\\
33 & D_{4}+A_{8}+A_{5} & 72 & 2 & 1 & 12 & 4\\
34 & D_{4}+A_{8}+A_{4}+A_{1} & 24 & 2 & 1 & 12 & 1\\
35 & D_{4}+A_{7}+A_{4}+2A_{1} & 8 & 1 & 1 & 8 & 0\\
36 & D_{4}+A_{7}+A_{3}+A_{2}+A_{1} & 16 & 2 & 1 & 8 & 1\\
37 & D_{4}+A_{7}+2A_{2}+2A_{1} & 32 & 1 & 2 & 32 & 0\\
38/39 & D_{4}+2A_{6}+A_{1} & 48 & 1 & 1 & 24 & 1\\
40 & D_{4}+A_{6}+A_{5}+A_{2} & 48 & 2 & 1 & 24 & 7\\
41 & D_{4}+A_{6}+A_{4}+A_{2}+A_{1} & 48 & 2 & 1 & 24 & 1\\
42 & D_{4}+A_{6}+A_{3}+2A_{2} & 96 & 1 & 1 & 96 & 13\\
43 & D_{4}+2A_{5}+A_{3} & 32 & 2 & 1 & 16 & 3\\
44 & D_{4}+2A_{5}+3A_{1} & 24 & 1 & 1 & 24 & 0\\
45 & D_{4}+A_{5}+2A_{4} & 96 & 2 & 1 & 48 & 7\\
46 & D_{4}+A_{5}+A_{4}+A_{3}+A_{1} & 16 & 2 & 1 & 8 & 1\\
47 & D_{4}+A_{5}+2A_{3}+2A_{1} & 16 & 2 & 1 & 8 & 0\\
48 & D_{4}+2A_{4}+2A_{2}+A_{1} & 192 & 1 & 2 & 192 & 1\\
49 & D_{4}+3A_{3}+2A_{2} & 384 & 1 & 2 & 384 & 5\\

\end{array}
\]
\caption{Indices $I_1= [\O^+(T) : \tO^+(T)]$, $I_2=[\tO(T) : \tO^+(T)] $, $ I_3= [\O(D(T)) : \overline{\O(T)}]$,
order of $\overline{S}_{M,\iota}/(\pm 1)$ and genus of the moduli space of lattice polarised $K3$ surfaces.}
\label{table:OPlus}
\end{table}

There are only four lattices $T$ for which $\Gamma_{M, \iota} \ne \O(T)$.
In these cases, the genus of $\hat{C}_T:= \HH_1/ \Gamma^+_{M,\iota} $, which is birational to the ambitypical stratum, is again $0$, once more in concordance with \cite[Theorem]{BPT}.
Table~\ref{table:Gamma} lists the indices $[\O^+(T) : \Gamma^+_{M,\iota} ]$ and $[\Gamma_{M,\iota} : \Gamma^+_{M,\iota}]$ in these cases.

\begin{table}[h]
\[
\begin{array}{c|c|c|cc}
\# & \text{root lattice} & [\O^+(T) : \Gamma^+_{M,\iota} ] & [\Gamma_{M,\iota} : \Gamma^+_{M,\iota}] \\ \hline
26 & D_{4}+A_{11}+A_{2} & 2 & 2 \\
29 & D_{4}+A_{9}+A_{4} & 3 & 2 \\
33 & D_{4}+A_{8}+A_{5} & 3 & 2 \\
38/39 & D_{4}+2A_{6}+A_{1} & 1 & 1 
\end{array}
\]
\caption{Indices of $\Gamma_{M,\iota}^+$ in $\O^+(T)$ and $\Gamma_{M,\iota}$.}
\label{table:Gamma}
\end{table}

\newpage
%
%
%
%.      Table placement
%
%
%

%%%%%%%%%%%%%%%%%%%%%%%%%%%%%%%%%%%%%%%%%
%%%%%%%%%%%%%%%%%%%%%%%%%%%%%%%%%%%%%%%%%


\begin{thebibliography}{Magma}

\bibitem[BHPV]{BHPV} W.~Barth, K.~Hulek, C.~Peters, A.~Van de Ven, 
  {\it Compact complex surfaces.} 2nd Enlarged Edition, Ergebnisse der Mathematik 3. Folge, 4, Springer Verlag 2004.
  
\bibitem[Be]{bea} A.~Beauville,
{\it Les familles stables de courbes elliptiques sur $\PP^1$ admettant quatre fibres singuli\`eres.}
C. R. Acad. Sci. Paris Ser. I Math. 294, 657--660 (1982).
  
\bibitem[BeMo]{bemo} F.~Beukers, H.~Montanus,
{\it Explicit calculations of elliptic fibrations of K3-surfaces and their Belyi-maps.}
In: Number theory and polynomials, Proceedings of the workshop, Bristol, UK, April 3--7, 2006.
LMS Lecture Note Ser.\ 352, Cambridge Univ. Press, Cambridge, 2008, 33--51.
  
\bibitem[BPT]{BPT}
F.~Bogomolov, T.~Petrov, Y.~Tschinkel,
{\it Rationality of moduli of elliptic fibrations with fixed monodromy.}
Geom. Funct. Anal. 12:6, 1105--1160 (2002).  

\bibitem[BT]{BT}
F.~Bogomolov, Y.~Tschinkel,
{\it Monodromy of elliptic surfaces.} 
In: Galois groups and fundamental groups, Math. Sci. Res. Inst. Publ., 41, Cambridge Univ. Press, Cambridge, 2003, 167--181.

\bibitem[BCNS]{Voronoi}
 O.~Braun, R.~Coulagon, G.~Nebe. S.~Sch\"onnenbeck,
 {\it Computing in arithmetic groups with Voronoi's algorithm},
 Journal of Algebra, 435, 263--285 (2015).

\bibitem[CS]{CS}
J.~H.~Conway, N.~J.~A.~Sloane
{\it Sphere packings, lattices and groups}.
3rd edition, Grundlehren der Mathematischen Wissenschaften 290, Springer Verlag 1999.

\bibitem[CoPa]{CoxP}
D.A.~Cox, W.R.~Parry,
{\it Torsion in elliptic curves over $k(t)$.}
Comp. Math. 41:3, 337--354 (1980).

\bibitem[CuPa]{CP}
C.J.~Cummins, S.~Pauli,
{\it Congruence subgroups of $\PSL(2,\mathbb{Z})$ of genus less than or equal to 24.}
Experiment. Math. 12:2,  243--255 (2003).  

\bibitem[Do]{Do} 
I.~Dolgachev,
{\it Mirror symmetry for lattice polarized $K3$ surfaces.}
J. Math. Sci. 81:3, 2599--2630 (1996).

\bibitem[FK]{FK} K.~Filom, A.~Kamalinejad,
{\it Dessins on Modular Curves.}
\url{http://arxiv.org/pdf/math/1603.01693.pdf} (2016).

\bibitem[FM]{FM}
R.~Friedman, J.~Morgan,
{\it Smooth four-manifolds and complex surfaces.}
Ergebnisse der Mathematik, 3.Folge, 27,  Springer Verlag (1994).


\bibitem[GHS1]{GHS}
V.~Gritsenko, K.~Hulek, G.~K.~Sankaran,
{\it Abelianisation of orthogonal groups and the fundamental group of modular varieties}.
J. Algebra, 322:2, 463--478 (2009).

\bibitem[GHS2]{GHS2}
V.~Gritsenko, K.~Hulek, G.~K.~Sankaran,
{\it Moduli of $K3$ surfaces and irreducible symplectic manifolds}.
Handbook of moduli. Volume I, 459--526, International Press (2015).


\bibitem[He]{Hem}
J.~Hempel,
{\it Existence conditions for a class of modular subgroups of genus zero.}
Bull. Austral. Math. Soc. 66:3, 517--525  (2002).

\bibitem[Hum]{Hum}
J.~Humphreys,
 {\it Reflection groups and Coxeter groups}.
 Cambridge University Press, Camb. Stud. Adv. Math. {29} (1992).

\bibitem[Hur]{Hur}
A.~Hurwitz,
{\it Ueber Riemann'sche Fl\"achen mit gegebenen Verzweigungspunkten.}
Math. Ann. 39, 1--61 (1891).

\bibitem[Huy]{Hu} D.~Huybrechts,
{\it Lectures on $K3$ surfaces}.  Cambridge University Press, 
Camb. Stud. Adv. Math. 158 (2016). 

\bibitem[Ka]{Katok}
S.~Katok,
{\it Fuchsian groups},
Chicago Lectures in Mathematics,
University of Chicago Press, Chicago (1992).
      
\bibitem[Kir]{Kir}
M.~Kirschmer,
{\it Definite quadratic and hermitian forms with small class number}.
Habilitation, RWTH Aachen University, 2016,

\bibitem[Kl]{Kl} R.~Kloosterman,
{\it Higher Noether-Lefschetz loci of elliptic surfaces.}
J. Diff. Geom. 76:2, 293--316 (2007).

\bibitem[Le]{Le}
P.~Lejarraga, 
{\it  The moduli of Weierstrass fibrations over $\mathbb P^1$: Rationality.} 
Rocky Mt. J. Math., 23:2, 649 -- 650 (1993).

\bibitem[Magma]{Magma}
W.~Bosma, W. J.~Cannon, C.~Playoust,
{\it The {M}agma algebra system. {I}. {T}he user language}.
J. Symbolic Comput., 24:3-4, 235--265 (1997).
    
\bibitem[MS]{MS}
J.~McKay, A.~Sebbar. 
{\it {J-invariants of arithmetic semistable elliptic surfaces and graphs}}. Proceedings
on Moonshine and related topics, Montr\'eal, QC, 119-130, 1999.  
CRM Proceedings and Lecture Notes 30, American
Mathematical Society, Providence, RI, 2001.

\bibitem[OM]{OMeara}
O'Meara, O.T.,
{\it Introduction to {Q}uadratic {F}orms.}
Grundlehren der Mathematischen Wissenschaften, 117, 
Springer Verlag (1973).

\bibitem[Me]{Mertens}
M.~H.~Mertens,
{\it Automorphism groups of hyperbolic lattices}.
J. Algebra, 408, 147--165 (2014). 

\bibitem[Mi81]{Mir}
R.~Miranda,
{\it The moduli of Weierstrass fibrations over $\PP^1$.}
Math. Ann. 255:3, 379--394 (1981).
 
\bibitem[Mi89]{Mi89}
R.~Miranda,
{\it The basic theory of elliptic surfaces.} 
Dottorato di Ricerca in Matematica, ETS Editrice, Pisa, 1989. 

\bibitem[Mi90]{Mi90}
R.~Miranda,
{\it Persson's list of singular fibres for a rational elliptic surface.} 
Math. Z. 205, 191--211 (1990).

\bibitem[MM]{MiMo}
R.~Miranda, D.~Morrison
{\it Embeddings of integral quadratic forms (electronic)}.\url{http://www.math.ucsb.edu/drm/manuscripts/eiqf.pdf} (2009).

\bibitem[MKS]{CGT}
W.~Magnus, A.~Karrass, D.~Solitar,
{\it Combinatorial group theory. Presentations of groups in terms of generations and relations}. 2nd Edition,
Dover Publications, Inc., New York (1976).
 

\bibitem[Nik]{Nik}
V.~Nikulin, 
{\it Integral symmetric bilinear forms and some of their applications.} 
Izv. Akad. Nauk SSSR Ser. Mat. 43:1, 111--177, 238 (1979).

\bibitem[OO]{OO}
Y.~Odaka, Y.~Oshima,
{\it Collapsing $K3$ surfaces, tropical geometry and moduli compactifications of Satake, Morgan-Shalen type.}
MSJ Memoirs, 40, Mathematical Society of Japan, Tokyo (2021).


\bibitem[OS]{os}
K.~Oguiso, T.~Shioda,
{\it The Mordell-Weil lattice of a rational elliptic surface.}
Comment. Math. Univ. St. Paul. 40, 83--99 (1991). 

\bibitem[ScSh]{ScSh}
M.~Sch\"utt, T.~Shioda,
{\it Mordell-Weil lattices.}
Springer Verlag, Ergebnisse der Mathematik 3. Folge, 70 (2019). 

\bibitem[Se]{Se}
A.Sebbar, 
{\it Classification of torsion-free genus zero congruence groups.} 
Proc. Amer. Math. Soc. 129:9, 2517--2527 (2001).

\bibitem[Shi1]{shimada} I.~Shimada,
{\it On elliptic K3 surfaces.} 
Michigan Math. J. 47:3, 423--446 (2000).\\
extended version including table 1 published as
\url{http://arxiv.org/pdf/math/0505140.pdf} (2005).

\bibitem[Shi2]{shimada2} I.~Shimada,
{\it Connected Components of the Moduli of Elliptic $K3$ Surfaces.} 
Michigan Math. J. 67:3, 511--559 (2018).

\bibitem[Shi3]{shimada3} I.~Shimada,
{\it Connected components of the moduli of elliptic K3 surfaces: computational data.}
\url{http://www.math.sci.hiroshima-u.ac.jp/~shimada/K3.html} (2016).

\bibitem[Shi4]{shimada4} I.~Shimada,
{\it A note on Mirand-Morrison theory.}\\
\url{http://www.math.sci.hiroshima-u.ac.jp/shimada/preprints/ConnEllK3/NoteMM.pdf} (2016).

\bibitem[Shio]{shioda} T.~Shioda,
{\it On the Mordell-Weil lattices.}
Comment. Math. Univ. St. Paul. 39, 211--240 (1990). 

\bibitem[Wo]{wohlfahrt} K.~Wohlfahrt,
{\it An Extension of F-Klein's Level Concept.} 
Illinois J.Math. 8, 529--535 (1964).

\bibitem[YaYo]{YaYo}
T.~Yamaguchi, S.~Yokura, 
{\it Poset-stratified space structures of homotopy sets.}
Homology Homotopy Appl.\  21:2, 1--22 (2019).

\bibitem[Za]{Zassenhaus}
H.~Zassenhaus,
{\it On the spinor norm.}
 Arch. Math., 13, 434--451 (1962).
    

\end{thebibliography}
\end{document}